%% file: bmp_paper.tex
\documentclass[final]{siamonline220329}
\usepackage{comment}
\usepackage{algorithm}
\usepackage[noend]{algpseudocode}
\usepackage{fancyvrb} 
\VerbatimFootnotes    
\input{my_preamble}

\input{tikz_tensor_object.tex}

\title{Tensor BM-Decomposition for Compression and Analysis of Video Data\footnotemark[1]}
\author{ 
Fan Tian\footnotemark[2], 
Misha E. Kilmer\footnotemark[2]\ , 
Eric Miller\footnotemark[3]\hspace{.07in}\footnotemark[6]\ , and
Abani Patra\footnotemark[1]\hspace{.07in}\footnotemark[4]\hspace{.05in}\footnotemark[5]
}

\begin{document}
\maketitle
\renewcommand{\thefootnote}{\fnsymbol{footnote}}
\footnotetext[1]{This material is based upon work supported by the National Science
Foundation under NSF HDR grant CCF-1934553. Misha E. Kilmer is also supported in part by NSF DMS-1821148.  Eric L. Miller is also supported in part by NSF grant 1935555. Abani K. Patra is also supported in part by 2004302.}
\footnotetext[2]{Department of Mathematics, Tufts University, Medford, MA 02115.}
\footnotetext[3]{Electrical and Computer Engineering, Tufts University, Medford, MA 02115.}
\footnotetext[4]{Computer Science Department, Tufts University, Medford, MA 02115.}
\footnotetext[5]{Data Intensive Studies Center, Tufts University, Medford, MA 02115.}
\footnotetext[6]{Tufts Institute for Artificial Intelligence, Tufts University, Medford, MA 02115.}\renewcommand{\thefootnote}{\arabic{footnote}}


\begin{abstract}
Given tensors $\TA, \TB, \TC$ of size $m \times 1 \times n$, $m \times p \times 1$, and $1\times p \times n$, respectively, their Bhattacharya-Mesner (BM) product will result in a third-order tensor of dimension $m \times p \times n$ and BM-rank of 1 (Mesner and Bhattacharya, 1990).
Thus, if an arbitrary $m \times p \times n$ third-order tensor can be written as a sum of a small number, relative to $m,p,n$, of such BM-rank 1 terms, this BM-decomposition (BMD) offers an implicitly compressed representation of the tensor. 
In this paper, we first show that
grayscale surveillance video can be accurately captured by a low BM-rank decomposition and give methods for efficiently computing this decomposition. To this end, we first give results
that connect rank-revealing matrix factorizations to the BMD. Next, we present a generative model that illustrates that spatio-temporal video data can be expected to have low BM-rank.  We combine these observations to derive a regularized alternating least squares (ALS) algorithm to compute an approximate BMD of the video tensor.  The algorithm itself is highly parallelizable since the bulk of the computations break down into relatively small regularized least squares problems that can be solved independently.  We study the ability of our algorithm to separate the video data into stationary and non-stationary components while simultaneously compressing the data.  We then introduce a new type of BM-product suitable for color video and provide an algorithm that shows an impressive ability to extract important temporal information from color video while simultaneously compressing the data.  Extensive numerical results compared against the state-of-the-art matrix-based DMD for surveillance video separation show our algorithms can consistently produce results with superior compression properties while simultaneously providing better separation of stationary and non-stationary features in the data.

\end{abstract}

\section{Introduction}
Low rank tensor decomposition methods have provided domain-specific insight into large, multidimensional data sets as well as a means of compressing these data \cite{kolda2009tensor}. Many such methods have been proposed including the CANDECOMP/PARAFAC or canonical polyadic (CP) decomposition \cite{kiers2000towards, mocks1988topographic}, the Tucker model or the higher-order SVD (HOSVD) method \cite{tucker1966some, de2000multilinear}, tensor-train decomposition \cite{oseledets2011tensor}, t-SVD \cite{kilmer2011factorization} and its more general form $\star_{M}$ tensor SVD (t-SVDM) \cite{kilmer2021tensor}.  The use of a specific method on a particular problem depends heavily on the underlying application (i.e., the properties of the data) as well as the processing objectives (compression, information extraction, etc.).  Of interest here 
{are the compression and the decomposition of video into stationary background and moving foreground components.} In \cite{karim2020accurate}, regularized CP decomposition was used for the video background estimation. 

Our approach builds on recent interest in the tensor Bhattacharya-Mesner (BM) product \cite{mesner1990association,mesner1994ternary} and associated BM-algebra \cite{gnang2011spectral,gnang2017spectra,gnang2020bhattacharya}. The fundamental difference between factoring a tensor using the BM-product representation and the tensor CP-decomposition lies in the order of the set of ``factor tensors" into which a given third-order tensor (the case of interest here) is decomposed. The CP approach is based on outer products generated from triplets of vectors while a BM-decomposition (BMD) employs triplets of matrices, as we show in Sec.~\ref{sec:background}.  Theoretical studies of third-order tensor spectral decomposition and singular value decomposition in terms of BM-product have been discussed in \cite{gnang2017spectra} and \cite{gnang2020symmetrization} respectively. However, no numerical algorithms on tensor BM-decomposition methods have been proposed until recently.  In 2022, we first presented an alternating least squares (ALS) algorithm to factor a third-order tensor into an unconstrained BMD \cite{siam2022talk}, which served as motivation for the present work.  Independently, in 2023, Luo et. al \cite{luo2023multidimensional} proposed an unconstrained tensor factorization framework based on the third-order tensor BMP, which they rename as matrix outer-product (MOP).  They, too, propose an ALS algorithm for an application in Bayesian inference, but do not address issues including starting guess or convergence.  

In this paper, we offer the following contributions:
\begin{itemize}
\item Theorems providing interpretations of rank-revealing matrix-factorizations, both on the unfolded data and on frontal faces of the data, in terms of the BMD that we will utilize in establishing starting guesses for our iterative algorithms;
\item Corollaries to the above that bound BM rank;
\item  A generative low BM-rank model for surveillance-type videos with stationary background and simple moving foreground illustrates the compressive power of the BMD for video separation and compression;
\item Discussion of implementation details for an unregularized ALS algorithm for the BMD;
\item New regularized ALS algorithm designed to overcome non-uniqueness issues and theoretical convergence issues associated with unregularized ALS and that also serves to enforce the separation of video into stationary and non-stationary components; 
\item Convergence and computational complexity analysis of the regularized algorithm;
\item A new, BM product-like definition suitable for generalizing the BMD for factoring color (a.k.a. fourth order) video;
\item A regularized ALS fourth-order BMD algorithm for color video compression/separation;  
\item Extensive numerical results and comparison with the state-of-the-art Dynamic Mode Decomposition (DMD) \cite{kutz2016dynamic} which is a matrix-based surveillance video separation technique, and a tensor-based method namely the spatiotemporal slice-based SVD (SS-SVD) \cite{kajo2018svd} method, in terms of separation quality and compression quality.   
\end{itemize}

The remainder of this paper is organized as follows. Background notation, tensor definitions, and notions of the BM-algebra are provided in Section~\ref{sec:background}. In Section~\ref{sec:BMDconnection}, we show how low-rank matrix factorizations, either applied to the unfolded tensor or to the frontal slices of the third-order tensor, can be reinterpreted within the BMD framework and can be used to give upper bounds on the BM-rank. In Section~\ref{sec:vid_application}, we give a brief overview of the task of video background and foreground separation, and discuss two popular methods including the DMD method and SS-SVD method, which can be used as starting guesses to the BMD.
Section~\ref{sec:generative} discusses the generative spatiotemporal video model, which illustrates the inherently low BM-rank property of (grayscale) video tensors. 
In Section~\ref{sec:lowBMrank}, we discuss the unconstrained low BM-rank approximation problem and an ALS algorithm for computing the tensor BMD. 
Section~\ref{sec:regALS} discusses the updates to the unconstrained algorithm to solve a regularized problem where the regularization is tailored to encouraging the desired separation into stationary and non-stationary components. 
Section~\ref{sec:numerical} contains an extensive numerical study of the third-order BMD method 
for the application of compressible video background/foreground reconstructions.
Finally, we describe a new BM-product-like 
method for fourth-order tensors and its application in decomposing color videos in Section~\ref{sec:colorBMD}.  

\section{Background and Notation}
\label{sec:background}
As shown in Fig.~(\ref{fig:tensorslice}), given $\TX \in \mathbb{R}^{m \times p \times n}$ a third-order tensor, there are three ways of slicing $\TX$:
\begin{itemize}
    \item The \textbf{lateral slices} of $\TX$ are given by $\TX_{:,j,:} \in \mathbb{R}^{m\times 1 \times n}$, $1\leq j \leq p$, which are obtained by slicing the tensor left to right.    \item The \textbf{frontal slices} of $\TX$ are given by $\TX_{:,:,k} \in \mathbb{R}^{m\times p\times 1}$, $1\leq k \leq n$, which are obtained by slicing the tensor front to back.
    \item The \textbf{horizontal slices} are denoted as $\TX_{i,:,:} \in \mathbb{R}^{1\times p\times n}$, $1\leq i \leq m$ and are obtained by  slicing the tensor top to bottom.
\end{itemize}
\textbf{Tube fibers} are denoted as $\TX_{i,j,:} \in \mathbb{R}^{1\times 1\times n}$, $1\leq i \leq m, 1\leq j \leq p$ and are obtained by holding both the first and the second indices fixed and varying the third index.

\begin{figure}[ht]
\centering
\begin{tikzpicture}
\tensor[dim1 = 1, dim2 = 1, dim3 = 0.75, fill = lightgray,
	slice type = lateral, number of slices = 5, slice width = 0.5] (0,0) {};
	\node at (0+0.25,-0.5) {\small lateral slices};
	\node at (0+0.25,-1) {\small $\TX_{:,j,:} \in \mathbb{R}^{m \times 1 \times n}$};
\tensor[dim1 = 1, dim2 = 1, dim3 = 0.75, fill = lightgray,
	slice type = frontal, number of slices = 5, slice width = 0.5] (3.25,0) {};
	 \node at (3.25 + 0.5,-0.5){\small frontal slices};
	 \node at (3.25 + 0.5,-1){\small $\TX_{:,:,k} \in \mathbb{R}^{m \times p \times 1}$};
\foreach \j in {0,0.18,0.36,0.54,0.72}{
		\tensor[dim1 = 0.09, dim2 = 1, dim3 = 0.75, fill = lightgray] (6.25,\j) {}; };
	\node at (6.5 + 0.5, -0.5) {\small horizontal slices};
    \node at (6.5 + 0.5, -1) {\small $\TX_{i,:,:} \in \mathbb{R}^{1 \times p \times n}$};
\foreach \i in {9.25, 9.25+0.18, 9.25+0.36, 9.25+0.54, 9.25+0.72}{
	\foreach \j in {0,0.18,0.36,0.54,0.72}{
		\tensor[dim1 = 0.09, dim2 = 0.09, dim3 = 0.75, fill = lightgray] (\i,\j) {}; }};	
		\node at (9.5 + 0.5, -0.5) {\small tube fibers};
		\node at (9.5 + 0.5, -1) {\small $\TX_{i,j,:} \in \mathbb{R}^{1 \times 1 \times n}$};
\end{tikzpicture}
\caption{Slices and tube fibers of a third-order tensor of size $m \times p \times n$ and corresponding indexing in {\sc Matlab} notation.}
\label{fig:tensorslice}
\end{figure}

Throughout the paper, we will use \textsc{Matlab}-derived \textbf{\textsc{squeeze}} command to convert tensor slices into matrices. The operation $\M{A}=\squeeze(\TA)$ turns a lateral slice $\TA \in \mathbb{R}^{m\times 1 \times n}$ or a horizontal slice $\TA \in \mathbb{R}^{1\times m \times n}$ into a matrix $\M{A}\in \mathbb{R}^{m \times n}$ \cite{kilmer2021tensor}. We also use the \textbf{\textsc{vec}} and \textbf{\textsc{reshape}} operators. The $\vvec$ operator maps a matrix $\M{A} \in \mathbb{R}^{m\times n}$ to a vector $\V{a} \in \mathbb{R}^{mn\times 1}$ by column-major ordering \cite{knuth1997art}. In {\sc Matlab} notation, we have $\vvec(\M{A}) \equiv \M{A}(:)$. The $\reshape(\V{A}, [m,n])$ operation folds a given vector $\V{a} \in \mathbb{R}^{mn\times 1}$ into a matrix $\M{A}$ of size $m\times n$ by filling this matrix one column at a time.  

Additionally, we define the \textbf{\textsc{Tvec}} and the \textbf{\textsc{Tfold}} operations for unfolding and folding tensors. Given a tensor $\TX \in \mathbb{R}^{m\times p\times n}$, we have
\begin{equation*}
\V{x}=\tvec(\TX) = \left[\begin{array}{c}
\V{x}^{(1,1)}\\
\vdots\\
\V{x}^{(i,j)}\\
\vdots\\
\V{x}^{(m,p)}
\end{array}\right] = 
\left[\begin{array}{c}
\squeeze\left(\TX_{1,1,:}\right)\\
\vdots\\
\squeeze\left(\TX_{i,j,:}\right)\\
\vdots\\
\squeeze\left(\TX_{m,p,:}\right)
\end{array}\right].
\end{equation*}
for all $1\leq i\leq m, 1\leq j\leq p$. This operation is also equivalent to storing the three-dimensional array in sequential memory location using row-major ordering \cite{knuth1997art}. The $\tfold$ operation is defined to be the inverse action of $\tvec$, i.e. $\tfold\left(\tvec(\TX) \right) = \TX.$ To make these definitions more concrete, we can think of the operation $\tvec(\TX)$ as stacking the tube fibers of $\TX$ into a long vector $\V{x}$, and the $\tfold$ operator reverse the flattening back into a third-order tensor (referring to Fig.~(\ref{fig:vecmat}.a) for visual illustration).

\begin{figure}[ht]
\begin{subfigure}[t]{0.51\textwidth}
\centering
\begin{tikzpicture} 
\tensor[dim1 = 0.1, dim2 = 0.1, dim3 = 0.75, fill = carminepink] (0, 0.91) {};
\tensor[dim1 = 0.1, dim2 = 0.1, dim3 = 0.75, fill = mayablue] (0.2, 0.91) {};
\tensor[dim1 = 0.1, dim2 = 0.1, dim3 = 0.75, fill = asparagus] (2*0.2, 0.91) {};
\tensor[dim1 = 1, dim2 = 1, dim3 = 0.75, fill = nofill, back edges = 1] (0, 0) {};
\node at (0.75, -0.25) {\small $\TX_{i,j,:} \in \mathbb{R}^{1 \times 1 \times n}$};
\node at (1+2.5, 1.2) {\small \squeeze($\T{X}_{i,j,:}$) };
\draw [arrows = {-Stealth[scale width=1]}] (1+1.2, 0.9) -- (1+4, 0.91);
\draw [arrows = {-Stealth[scale width=1]}] (1+4, 0.8) -- (1+1.2, 0.8) ;
\node at (1+2.7, 0.5) {\small $\reshape(\V{x}^{(i,j)}, [1,1,n]$)};
\tensor[dim1 = 0.3, dim2 = 0.1, dim3 = 0.075, fill = carminepink] (2+4, 1.2) {};
\tensor[dim1 = 0.3, dim2 = 0.1, dim3 = 0.075, fill = mayablue] (2+4, 0.9) {};
\tensor[dim1 = 0.3, dim2 = 0.1, dim3 = 0.075, fill = asparagus] (2+4, 0.6) {};
\tensor[dim1 = 0.6, dim2 = 0.1, dim3 = 0.075, fill = nofill, back edges = 1] (2+4, 0) {};
\node at (2+4, -0.25) {\small $\Vx^{(i,j)} \in \mathbb{R}^{n \times 1}$};
\node at (-0.5, 1.6){\small (a)};
\end{tikzpicture}
\end{subfigure}
~ 
\begin{subfigure}[t]{0.5\textwidth}
\centering
\begin{tikzpicture}
[block/.style ={rectangle, text width=1.5em, align=center, minimum height=1.4em}]
\node at (-0.4, 0.8) {\small $\M{H}=\mat(\TA,\TB)=$};
\node [matrix, draw=black] (my matrix) at (2.8, 0.8)
{
    \node [block, draw=black, fill=carminepink!30]{}; & & &\\
     & \node [block, draw=black, fill=mayablue!30]{}; & &\\
& & \node [block, draw=black, fill=asparagus!30]{}; & \\
& & & \node {$\ddots$}; \\
};
\node at (3.3, 0.62) {\tiny $\M{H}^{(1,3)}$};
\node at (2.45, 1.2) {\tiny $\M{H}^{(1,2)}$};
\node at (1.6, 1.77) {\tiny $\M{H}^{(1,1)}$};
\node at (-1.6, 1.6){\small (b)};
\end{tikzpicture}
\end{subfigure}%
\caption{(a) Illustration of the tensor $\tvec(\T{X})$ and $\tfold(\V{x})$ operations. The vectorization and the reshaping of the first three tube fibers $\TX_{1,1,:}, \TX_{1,2,:},$ and $\TX_{1,3,:}$ are shown in this figure, and the rest are omitted. (b) Illustration of the tensor $\mat$ operation that flattens the two tensors $\TA\in \mathbb{R}^{m\times \ell \times n},\TB\in \mathbb{R}^{\ell \times p \times n}$ into a block-diagonal matrix $\M{H}\in \mathbb{R}^{mpn \times mp\ell}$.}
\label{fig:vecmat}
\end{figure}
We also define the \textbf{\textsc{Mat}} operation which flattens two tensors of specific sizes into a block-diagonal matrix (see Fig.~\ref{fig:vecmat}.b). Given $\T{A} \in \mathbb{R}^{m\times \ell \times n}$ and $\T{B} \in \mathbb{R}^{\ell \times p \times n}$, $\mat(\TA,\TB)$ yields $\M{H} \in \mathbb{R}^{mpn \times mp\ell}$ defined as 
\begin{equation*}
     \M{H} = \mat(\TA, \TB) =\underset{i,j}{\oplus}\,\M{H}^{(i,j)}, \quad \forall \, 1\leq i \leq m, 1\leq j\leq p,
\end{equation*}
where $\M{H}^{(i,j)} \in \mathbb{R}^{n\times \ell}$ with its entries specified as $\M{H}^{(i,j)}_{k,t} = \TA_{i,t,k}\TB_{t,j,k}$.

We use the \textsc{Matlab} \textbf{\textsc{permute}} operator to rearrange the dimensions of an array. In particular, tensor transposes are defined based on the cyclic permutations of the indices of each entry. Details are given in Sec.~\ref{subsec:definitions}. 

The Frobenius norm \cite{kolda2009tensor} of a tensor $\TX \in \mathbb{R}^{m\times p \times n}$ is defined analogously to the matrix case as $\displaystyle{\|\TX \|_{F} = \sqrt{\sum_{i=1}^{m} \sum_{j=1}^{p} \sum_{k=1}^{n} |\TX_{i,j,k}|^2}
}$.

\subsection{Overview of the BM-algebra}
\label{subsec:definitions}
The BM-product on a third-order tensor triplet was first introduced by Mesner and Bhattacharya \cite{mesner1990association,mesner1994ternary} and later generalized to tensors of arbitrary orders by Gnang and Filmus \cite{gnang2017spectra,gnang2020bhattacharya}. In this section, we will focus on the third-order case.
\begin{definition}
    For a third-order conformable tensor triplet $\TA \in \mathbb{R}^{m \times \ell \times n};  \TB \in \mathbb{R}^{m \times p \times \ell}$ and $\TC \in \mathbb{R}^{\ell \times p \times n}$, the BM-product $\TX=\bmp\left(\TA, \TB,\TC \right) \in \mathbb{R}^{m\times p\times n}$ is given entry-wise by
\begin{equation}
    \TX_{i,j,k} = \sum_{1\leq t \leq \ell} \TA_{i,t,k}\TB_{i,j,t}\TC_{t,j,k}.
    \label{eq:bmp_def}
\end{equation}
\label{def:bmp-entrywise}
\end{definition}
The BMP also conveniently expresses the notion of the BM-outer product. When $\ell = 1$, a BM-outer product corresponds to the BMP of the conformable order-$2$ tensor (matrix) slices, i.e. a lateral slice $\TA \in \mathbb{R}^{m\times 1\times n}$, a frontal slice $\TB \in \mathbb{R}^{m\times p\times 1}$, and a horizontal slice $\TC \in \mathbb{R}^{1\times p\times n}$. Consequently, the BM-product $\TX=\bmp\left(\TA, \TB,\TC \right)$ given in Eq. (\ref{eq:bmp_def}) can be written equivalently as a sum of BM outer-products of matrix slices
\begin{equation}
     \TX = \sum_{1\leq t \leq \ell} \bmp\left(\TA_{:,t,:},\TB_{:,:,t},\TC_{t,:,:}\right).
\label{eq:bm-outer-prod}
\end{equation}
A figure illustration is shown in Fig.~(\ref{fig:bmp_outer}).
\begin{figure}[H]
\centering
\begin{tikzpicture}
\tensor[dim1=1, dim2=1, dim3=0.8, fill = lightgray] (0.2, 0) {};
\node at (0.7, -0.3) {\small $\TX$};
\node at (2.15, 0.8) {$=$};
\tensor[dim1=1, dim2=0.1, dim3=0.8, fill = lightgray] (2.75, 0) {};
\tensor[dim1=1, dim2=1, dim3=0.1, fill = lightgray] (3.4, -0.1) {};
\tensor[dim1=0.1, dim2=1, dim3=0.8, fill = lightgray] (3.5, 1.1) {};
\node at (2.8, -0.25) {\small $\T{A}_{:,1,:}$};
\node at (4, -0.4) {\small $\T{B}_{:,:,1}$};
\node at (4.4, 1.9) {\small $\T{C}_{1,:,:}$};
\node at (3+2.25, 0.8){$+$};
\tensor[dim1=1, dim2=0.1, dim3=0.8, fill = lightgray] (3+2.75, 0) {};
\tensor[dim1=1, dim2=1, dim3=0.1, fill = lightgray] (3+3.4, -0.1) {};
\tensor[dim1=0.1, dim2=1, dim3=0.8, fill = lightgray] (3+3.5, 1.1) {};
\node at (3+2.8, -0.25) {\small $\T{A}_{:,2,:}$};
\node at (3+4, -0.4) {\small $\T{B}_{:,:,2}$};
\node at (3+4.4, 1.9) {\small $\T{C}_{2,:,:}$};
\node at (6.6+2.15, 0.8){$+ \cdots +$};
\tensor[dim1=1, dim2=0.1, dim3=0.8, fill = lightgray] (7+2.75, 0) {};
\tensor[dim1=1, dim2=1, dim3=0.1, fill = lightgray] (7+3.4, -0.1) {};
\tensor[dim1=0.1, dim2=1, dim3=0.8, fill = lightgray] (7+3.5, 1.1) {};
\node at (7+2.8, -0.25) {\small $\T{A}_{:,\ell,:}$};
\node at (7+4, -0.4) {\small $\T{B}_{:,:,\ell}$};
\node at (7+4.4, 1.9) {\small $\T{C}_{\ell,:,:}$};
\end{tikzpicture}
\caption{Illustration of the BM-product of a conformable tensor triplet $\TA,\TB$, and $\TC$ as a sum of BM-outer products of matrix slices.}
\label{fig:bmp_outer}
\end{figure}

\begin{remark} 
Every BM-product expresses a sum of BM-outer products and vice versa.
\end{remark}
The BM-outer product in Eq.(\ref{eq:bm-outer-prod}) induces a natural notion of the tensor BM-rank. 
\begin{definition}[BM-rank \cite{gnang2020bhattacharya}]
    The BM-rank of $\TX \in \mathbb{R}^{m\times p\times n}$ is the minimum number of the BM-outer products of conformable matrix slices that sum up to $\TX$.
\label{def:bm-rank}
\end{definition}

Tensor transpose was introduced analogous to matrix transpose \cite{gnang2017spectra,gnang2020bhattacharya}.
\begin{definition}
Suppose $\TX$ is a third-order tensor of size $m\times p \times n$, then $\TX$ has the following transpose operations which are given by cyclic permutations of the indices of each entry:
\begin{equation*}
\T{X}^{\top}_{i,j,k} = \T{X}_{k,i,j},\quad \TX^{\top} \in \mathbb{R}^{p\times n\times m}; \qquad \TX^{\top^2}_{i,j,k} = \T{X}_{j,k,i}, \quad \TX^{\top^2} \in \mathbb{R}^{n \times m\times p}.
\end{equation*}
\end{definition}
Note that in {\sc Matlab}, we can conveniently perform the transpose by the command $\TX^{\top} = \perm(\TX, [2,3,1])$.
Then $\TX^{\top^2} = \perm(\TX,[3,1,2])$ and
$\TX^{\top^3} = \perm(\TX,[1,2,3])=\TX.$

When $\TX$ is a BM-product of tensors $\TA, \TB$ and $\TC$, then the transpose of the BM-product is a BM-product of transposes such that
\begin{equation*}
    \TX^{\top}=\bmp(\TA,\TB,\TC)^{\top} = \bmp\left(\TB^{\top},\TC^{\top},\TA^{\top}\right).
\end{equation*}
Moreover, the BM-rank of a third-order tensor is equal to the BM-rank of its transpose.

\section{Matrix-to-BMD Connections}\label{sec:BMDconnection}
Two primary goals for this paper are  to show that video can be accurately captured by a low BM-rank decomposition, and to give methods for efficiently computing this decomposition.  A first step in this process is to give results that connect rank-revealing matrix factorizations to the BMD.  We will then use these results both in analyzing our generative model for surveillance video and as starting guesses for our iterative methods. 

\subsection{Matrix based decomposition}
Let $\T{X}\in \mathbb{R}^{m\times p\times n}$. Construct a matrix $\M{X}\in \mathbb{R}^{mn\times p}$ such that $\M{X}_{:,j} = \vvec(\T{X}_{:,j,:})$. 
In the context of video data, the $j$th column of $\M{X}$ corresponds to the vectorized set of pixels at the $j$th time step.

Our first result shows that we can rewrite this matrix decomposition as a BMD:
\begin{theorem}\label{thm:matrix2bmd}
Suppose $\M{X}$ has the following decomposition $\M{X}=\M{U}\M{V}^{\top}$, where $\M{U}\in \mathbb{R}^{mn\times \ell}$ and $\M{V}^{\top}\in \mathbb{R}^{\ell \times p}$. The matrix decomposition of $\M{X}$ can be represented by a tensor BMD of $\T{X}\in \mathbb{R}^{m\times p\times n}$.
\end{theorem}
\begin{proof}
Define $\T{A}_{:,t,:} = \reshape(\M{U}_{:,t},[m, n])$ for all $1\leq t \leq \ell$, and $\T{C}_{:,:,k} = \M{V}^{\top}$ for all $1\leq k \leq n$. Let $\T{B} = \ones([m,p,\ell])$. Then $\T{X} = \bmp(\T{A},\T{B},\T{C})$. 
\end{proof}

Note we can bound the BM rank of $\TX$ in terms of ranks of the data arranged as one of three matrices, where the three matrices come from unfolding each of $\TX, \TX^{\top}, \TX^{\top^2}.$   
\begin{corollary}
Under the conditions of Theorem \ref{thm:matrix2bmd}, the BM-rank of $\TX$ is bounded above by the minimum rank of the unfoldings of $\TX, \TX^\top, \TX^{\top^2}$. 
\end{corollary}

So the matrix decomposition can be viewed as a BMD with maximum rank of $\ell$.   
However, the converse is generally not true:  given a rank-$\ell$ BMD, the unfolding does not have to correspond to a low-rank matrix decomposition.  Thus, the BMD is
not just a rearrangement of a low-rank matrix decomposition of the data set.  If the data can be approximated by a tensor of small BM rank relative to the dimensions, it offers the opportunity for a fundamentally new way of compressing the data.

\subsection{Slice-wise based decomposition}
\label{subsec:slice-wise}
The following operates on the tensor directly by factoring frontal slices, each of which can be done independently (in parallel).

\begin{theorem}\label{thm:slice2bmd}
Suppose $\T{X}$ has the following slice-wise matrix decomposition 
$$\T{X}_{:,:,k}=\M{W}^{(k)}(\M{V}^{(k)})^{\top},\quad \forall \, 1\leq k\leq n$$
where $\M{W}^{(k)}\in \mathbb{R}^{m\times \ell}, \M{V}^{(k)}\in \mathbb{R}^{p \times \ell}$. 
Then the slice-wise decomposition of $\T{X}$ can be represented by a tensor BM-decomposition of $\T{X}$ as follows:

Let $\T{A}_{:,:,k} = \M{W}^{(k)}$, and $\T{C}_{:,:,k} = (\M{V}^{(k)})^{\top}$ for all $1\leq k \leq n$. Let $\T{B} = \ones([m,p,\ell])$. Then $\T{X} = \bmp(\T{A},\T{B},\T{C})$ with BM-rank at most $\ell$. 
\end{theorem}

In \cite{gnang2020bhattacharya}, the BM-rank of a generic tensor $\TX \in \mathbb{R}^{m\times p\times n}$ is said to be bounded above by $\min\{m,p,n\}$. Theorem \ref{thm:slice2bmd} implies that the BM rank may actually be smaller.  We need only compute an SVD of each frontal slice $\TX_{:,:,k} = \M{U} \M{S} (\M{V})^{\top}$. Let $\ell = \max_k(r_k)$, where $r_k$ is the rank of the $k$-th frontal slice. Then the first $\ell$ columns of the product $\M{U}\M{S}$ for each face will become $\M{W}^{(k)}$, the first $\ell$ rows of each $\M{V}^{\top}$ become the $\ell$ rows of $(\M{V}^{(k)})^{\top}$. We call this the {\bf slicewise SVD}. 
Because we can repeat the above slicewise-SVD process for $\TX^\top$, $\TX^{\top^2}$ to get BM-decompositions:  
\begin{corollary} 
The BM-rank cannot exceed the smallest of the maximum slicewise rank of the $n+m+p$ matrix slices given by
\begin{equation*}
\TX_{:,:,k}, \quad \forall 1 \le k \le n; \qquad
\squeeze(\TX_{i,:,:}), \quad \forall \ 1 \leq i \leq m; \qquad
\squeeze(\TX_{:,j,:}), \quad \forall \  1 \leq j \leq p
\end{equation*}
\end{corollary}

By considering the generative model in Sec.~\ref{sec:generative}, we will see that this result is especially useful in the context of video analysis.

\section{Background on Video Surveillance Separation and Modeling}
\label{sec:vid_application}
\input{video_surveillance}

\section{Generative Model for Surveillance Video}
\label{sec:generative}
\input{generative_model}

\section{Low BM-rank Tensor Approximation}
\label{sec:lowBMrank}
Given a third-order video tensor $\TX \in \mathbb{R}^{m\times p\times n}$ with BM-rank $r$, our goal is to compute a decomposition with BM-rank $\ell$, $1\leq \ell \leq r$, which best approximates $\TX$ such that
\begin{equation}
    \min_{\Tilde{\TX}} \|\TX - \Tilde{\TX}\|_{F}^{2} \text{ with } \Tilde{\TX} = \sum_{t=1}^{\ell} \bmp \left(\TA_{:,t,:}, \TB_{:,:,t}, \TC_{t,:,:}\right).
    \label{eq:bmp_approx}
\end{equation}
As we will show in the next section, in fact the BMD is not unique, a point which was not addressed in \cite{luo2023multidimensional}.  Thus, our ultimate goal will be to replace this problem with a regularized one which imposes desirable features in the context of our application.  However, the regularized approach involves only straightforward modification to the unregularized subproblems below, so for the sake of readability, we present the unregularized version first.   

Since the BM-product is a ternary multiplication of the $m \times \ell \times n, m\times p \times \ell$ and $\ell \times p \times n$ factor tensors $\TA, \TB$, and $\TC$, respectively, finding a decomposition of $\TX$ in terms of the factor tensors could be done by treating this as a nonlinear least-squares problem. 
The Jacobian of the residual vector,  $\vvec(\TX - \bmp(\TA,\TB,\TC))$ will have $3\ell$ non-zero entries per row, $mpn$ rows (more, in the regularized case) and $\ell(mn\!+\!mp\!+\!np)$ columns. Computing each
search direction for a nonlinear least squares solver, therefore, would require a call to an iterative method for large tensors.

Keeping in mind issues of memory and data movement, and having a preference for parallelizability, we focus instead on deriving an alternating least-squares (ALS) algorithm to solve the BMD problem, which we call BMD-ALS.  We show the work involved per iteration can be decoupled into $mp$ small problems that can be solved independently of one another. The ALS algorithm has been widely used for computing the tensor CP decomposition \cite{carroll1970analysis, kolda2009tensor} and tensor block term decomposition \cite{de2008decompositions}. Despite the limitations of the algorithm such as the slow convergence, the dependency on the starting guesses, the swamp effect \cite{navasca2008swamp} and more \cite{li2013some, uschmajew2012local}, the simplicity of understanding and implementing the ALS algorithm with superior quality of results still marks it as today's ``workhorse'' algorithm for CP decomposition \cite{kolda2009tensor}. We will show in the following subsections that BMD-ALS is also straightforward to implement, with  relatively small, dense, independent subproblems that can be solved in parallel. Moreover, we will show that the slicewise SVDs serve as an excellent starting guess for BMD-ALS, particularly for our video application (see Sec.~\ref{sec:vid_application}), and for which we can compute the initial error easily. 

\subsection{Phase I - Starting Guess}
\label{subsec:phase1}
Two ways of obtaining starting guesses are outlined in Sec.~\ref{sec:BMDconnection}: we can either take a suitable low rank matrix approximation from the unfolded data, and refold accordingly to get $\TA^{0}, \TB^{0}, \TC^{0}$ as in Theorem.~\ref{thm:matrix2bmd}; or we can take a low-rank matrix approximation (with fixed allowable rank) to each frontal slice, and refold according to Theorem.~\ref{thm:slice2bmd} to get our starting guesses.  The first choice will allow us to directly improve upon the DMD method discussed in Sec.~\ref{subsec:dmd}. The 2nd choice will allow us a direct comparison to the SS-SVD approach given in Sec.~\ref{subsec:ss-svd}. 

\subsection{Alternating Least-Squares (ALS) Algorithm}
Given the data tensor $\T{T} \in \mathbb{R}^{m\times p \times n}$, and the initial factor tensor triplet $\T{A}^{0} \in \mathbb{R}^{m\times \ell \times n}$, $\T{B}^{0} \in \mathbb{R}^{m \times p \times \ell}$ and $\T{C}^{0} \in \mathbb{R}^{\ell \times p \times n}$ obtained from phase I (Sec.~\ref{subsec:phase1}), we solve the following ALS subproblems for iterations $k=0, 1, 2, \dots$
\begin{equation}\label{eq:als_tensor}
\begin{split}
     \T{B}^{k+1} &= \min_{\T{B}\in\mathbb{R}^{m\times p\times\ell}}\left\Vert \T{T}-\bmp\left(\T{A}^{k},\T{B},\T{C}^{k}\right)\right\Vert _{F}^{2},\\
    \left(\T{C}^{\top}\right)^{k+1} &= \min_{\T{C}^{\top} \in\mathbb{R}^{ p\times n \times \ell}}\left\Vert \T{T}^{\top}-\bmp\left(\left(\T{B}^{\top}\right)^{k+1},\T{C}^{\top},\left(\T{A}^{\top}\right)^{k}\right)\right\Vert _{F}^{2},\\
     \left(\T{A}^{\top^{2}}\right)^{k+1} &= \min_{\T{A}^{\top^2}\in\mathbb{R}^{n\times m\times\ell}}\left\Vert \T{T}^{\top^{2}}-\bmp\left(\left(\T{C}^{\top^{2}}\right)^{k+1},\T{A}^{\top^{2}},\left(\T{B}^{\top^{2}}\right)^{k+1}\right)\right\Vert _{F}^{2}.
\end{split}
\end{equation}
In each ALS subproblem, we are holding the first and the third factor tensors fixed and solving for the middle tensor. Taking the first subproblem in the ALS algorithm as an example, we show that the problem of updating $\T{B}$ reduces to a linear least-squares problem as discussed in Phase II below (Sec.~\ref{subsec:phase2}), and updating $\T{C}$ and $\T{A}$ factor tensors can be solved similarly using the same algorithm. 

\subsection{Phase II - Linear Least-Squares Problem}
\label{subsec:phase2} 
Taking the first subproblem for updating tensor $\T{B}\in \mathbb{R}^{m\times p\times \ell}$ given in the ALS algorithm (Eq.~\ref{eq:als_tensor}) as an example, we will show in this section that this tensor least-squares problem can be reduced to a linear least-squares problem. Furthermore, the problem can be decoupled into many ($mn$ for updating $\T{B}$, $np$ for updating $\T{C}$, and $mn$ for updating $\T{A}$) smaller sized independent subproblems making the computation massively parallelizable.

Holding the pair $\TA$ and $\TC$ fixed, we optimize for $\TB$ by solving
\begin{equation}\label{eq:bmp_middle}
    \widehat{\T{B}} = \min_{\TB \in \mathbb{R}^{m\times p\times \ell }}\|\T{T} - \bmp( \TA , \TB, \TC ) \|_F^2.
\end{equation}
We will show in the following derivations that the tensor least-squares problem given in Eq. (\ref{eq:bmp_middle}) can be equivalently written as a matrix least-squares problem.

By the definitions of the Frobenius norm and tensor BM-product, we have
\begin{equation*}
\left\Vert \T{T}-\bmp\left(\T{A},\T{B},\T{C}\right)\right\Vert _{F}^{2}=\sum_{i=1}^{m}\sum_{j=1}^{p}\sum_{k=1}^{n}\left|\T{T}_{i,j,k}-\sum_{t=1}^{\ell}\T{A}_{i,t,k}\T{B}_{i,j,t}\T{C}_{t,j,k}\right|^{2}.
\end{equation*}
Holding the indices $i$ and $j$ fixed, we have
\begin{equation*}
\begin{split}
    &\sum_{(i-1)p+j=1}^{mp}\sum_{k=1}^{n}\left|\T{T}_{i,j,k}-\sum_{t=1}^{\ell}\T{A}_{i,t,k}\T{C}_{t,j,k}\T{B}_{i,j,t}\right|^{2}\\
    =&\sum_{(i-1)p+j=1}^{mp} \sum_{k=1}^{n} \left|\V{y}_{\T{T}}^{\left(i,j\right)}(k,1)-\sum_{t=1}^{\ell}\M{H}^{(i,j)}_{\T{A},\T{C}}(k,t)\V{b}_{t}^{\left(i,j\right)}\right|^2
  =  \sum_{(i-1)p+j=1}^{mp} \left \|\V{y}_{\T{T}}^{\left(i,j\right)}- \M{H}^{(i,j)}_{\T{A},\T{C}} \V{b}^{\left(i,j\right)} \right \|_F^2,
\end{split} 
\end{equation*}
where $\V{y}_{\T{T}}^{(i,j)} = \squeeze(\T{T}_{i,j,:}) \in \mathbb{R}^{n\times 1}$, $\V{b}^{(i,j)} = \squeeze(\T{B}_{i,j,:}) \in \mathbb{R}^{\ell \times 1}$, and $\M{H}_{\T{A},\T{C}}^{(i,j)}\in \mathbb{R}^{n\times \ell}$ with the $(k,t)$-th entry specified as $\M{H}^{(i,j)}_{\T{A},\T{C}}(k,t) = \T{A}_{i,t,k}\TC_{t,j,k}$. 

Let us stack the vectors $\V{y}_{\T{T}}^{(i,j)}$ and $\V{b}^{(i,j)}$ as
\begin{equation*}
\V{y}_{\T{T}} := \left[\begin{array}{c}
\V{y}_{\T{T}}^{(1,1)}\\
\vdots\\
\V{y}_{\T{T}}^{(i,j)}\\
\vdots\\
\V{y}_{\T{T}}^{(m,p)}
\end{array}\right]=\tvec(\T{T})\in \mathbb{R}^{mpn\times 1} \text{, and }
\V{b} := \left[\begin{array}{c}
\V{b}^{(1,1)}\\
\vdots\\
\V{b}^{(i,j)}\\
\vdots\\
\V{b}^{(m,p)}
\end{array}\right] = \tvec(\TB) \in \mathbb{R}^{mp\ell \times 1}.
\end{equation*}
Moreover, we take direct-sums of $\M{H}^{(i,j)}_{\T{A},\T{C}}$ and define
$\M{H}_{\T{A},\T{C}} := \underset{i,j}{\oplus}\M{H}^{(i,j)}_{\T{A},\T{C}} =\mat(\TA,\TC)$. Then the tensor least-squares problem given in Eq. (\ref{eq:bmp_middle}) reduces to the following matrix least-squares problem
\begin{equation} \label{eq:updateB}
    \widehat{\V{b}} = \min_{\V{b}\in \mathbb{R}^{mn\ell\times 1}} \left\|\V{y}_{\T{T}} - \M{H}_{\T{A},\T{C}}\V{b} \right\|_F^2.
\end{equation}
Furthermore, Eq.~(\ref{eq:updateB}) decouples into $mp$ least-squares subproblems of size $n \times \ell$ such that
\begin{equation} \label{eq:decoupledB}
    \widehat{\Vb}^{(i,j)} =  \min_{\Vb^{(i,j)} \in \mathbb{R}^{\ell \times 1}} \left\|\V{y}_{\T{T}}^{(i,j)} - \M{H}_{\T{A},\T{C}}^{(i,j)}\Vb^{(i,j)} \right\|_F^2.
\end{equation}
We adopt the same flattening scheme described above to obtain the following equivalent linear least-squares subproblems for updating the factor $\T{A}$:
\begin{equation}\label{eq:updateC}
    \widehat{\V{C}} = \min_{\V{C} \in\mathbb{R}^{pn\ell \times 1}}\left\Vert \V{y}_{\T{T}^{\top}}-\M{H}_{\T{B},\T{A}}\V{c}\right\Vert _{F}^{2},
\end{equation}
where $\V{c} = \tvec\left(\T{C}^{\top}\right)$, $\V{y}_{\T{T}^{\top}} = \tvec\left(\T{T}^{\top}\right)$, and $\M{H}_{\T{B},\T{A}} = \mat\left(\T{B}^{\top},\T{A}^{\top}\right)$, and the decoupled problem is given by
\begin{equation} \label{eq:decoupledC}
    \widehat{\Vc}^{(j,k)} =  \min_{\Vc^{(j,k)} \in \mathbb{R}^{\ell \times 1}} \left\|\V{y}_{\T{T}^{\top}}^{(j,k)} - \M{H}_{\T{B}^{\top},\T{A}^{\top}}^{(j,k)}\Vc^{(j,k)} \right\|_F^2.
\end{equation}
We update $\T{C}$ by solving:
\begin{equation}\label{eq:updateA}
     \widehat{\V{A}} = \min_{\V{A}\in\mathbb{R}^{n m\ell\times1}}\left\Vert \V{y}_{\T{T}^{\top^{2}}}-\M{H}_{\T{C},\T{B}}\V{A}\right\Vert _{F}^{2},
\end{equation}
where $\V{A} = \tvec\left(\T{A}^{\top^2}\right)$, $\V{y}_{\T{T}^{\top^2}} = \tvec\left(\T{T}^{\top^2}\right) $, and $\M{H}_{\T{C},\T{B}} = \mat\left(\T{C}^{\top^2},\T{B}^{\top^2}\right)$, and the decoupled problem is given by
\begin{equation}\label{eq:decoupledA}
     \widehat{\V{A}}^{(k,i)} = \min_{\V{A}^{(k,i)}\in\mathbb{R}^{n m\ell\times1}}\left\Vert \V{y}^{(k,i)}_{\T{T}^{\top^{2}}}-\M{H}^{(k,i)}_{\T{C},\T{B}}\V{A}\right\Vert _{F}^{2}.
\end{equation}

The implementation of the ALS algorithm for computing tensor BMD is illustrated in Algorithm.~(\ref{alg:bmp-als}). The termination criteria for the algorithm is chosen to be either reaching the maximum number of iterations $N$ or the relative change in successive iterations becomes sufficiently small, i.e. $\|\T{T}^{n+1} - \T{T}^{n}\|_F/\|\T{T}^{n}\|_F < \epsilon$ for some tolerance parameter $\epsilon>0$. A complete iterate requires to compute all three-factor tensors $\TA$, $\TB$, and $\TC$.

\begin{algorithm}
\caption{Tensor BMD-ALS}\label{alg:bmp-als}
\begin{algorithmic}[1]
\Procedure{$[\T{A}^{N}, \T{B}^{N}, \T{C}^{N}]$=BMD-ALS}{$\T{T}, \T{A}^0, \T{B}^0, \T{C}^0, N, \epsilon$}
\For{$n=0,1,2,\dots, N$}

    \smallskip
    \State Update the $(i,j)$-th tube fiber $\TB^{n+1}_{i,j,:}$ by solving Eq.~\ref{eq:decoupledB}.
    
    \smallskip
    
    \State Update the $(j,k)$-th tube fiber $(\TC^{\top})^{n+1}_{j,k,:}$ by solving Eq.~\ref{eq:decoupledC}.
    
    \smallskip
    
    \State Update the $(k,i)$-th tube fiber $(\TA^{\top^2})^{n+1}_{k,i,:}$ by solving Eq.~(\ref{eq:decoupledA}).
    
    \smallskip
    
    \State $\T{T}^{n+1} = \bmp(\T{A}^{n+1}, \T{B}^{n+1}, \T{C}^{n+1})$
\EndFor
\If{$\|\T{T}^{n+1} - \T{T}^{n}\|_F/\|\T{T}^{n}\|_F < \epsilon$}
    \State $n \gets N$
\EndIf
\EndProcedure
\end{algorithmic}
\end{algorithm}

\paragraph{Remark} The cost of implementing SVD for a matrix of size $m \times n$ with $m \geq n$ is $\mathcal{O}(mn^2+n^2)$ flops.
Moreover, the cost of solving a least-squares problem with the matrix of size $m \times n$, $m \geq n$, has the same time complexity as computing the SVD of the matrix \cite{golub2013matrix}. In phase I, the first choice of obtaining a starting guess is to compute a DMD of the vectorized video matrix. The total cost for DMD is dominated by taking SVD of the data matrix of size $mn\times (p-1)$ where $mn\geq (p-1)$ and solving a least-squares problem with the coefficient matrix of size $mn\times \ell$. Hence, the total cost is $\mathcal{O}\left((mn+1)((p-1)^2+\ell^2)\right)$ for the DMD choice. For the second choice, the total cost is
$\mathcal{O}\left(n(m+1)p^2\right)$ for computing the $n$ frontal slice matrix SVDs, assuming $m\geq p$. In phase II, the total cost of solving the $mp$ smaller least-squares problem is $\mathcal{O}\left(mp(n+1)\ell^2 \right)$. The least-squares solutions are obtained using e.g., {\sc{Matlab}}'s {\tt{lsqminnorm}} function for the minimum norm solution. 

We also emphasize that both phase I (with the slice-wise SVD choice) and phase II computations can be done in parallel since we can compute matrix SVDs simultaneously on the individual frontal slices of the input tensor, and the second phase of the algorithm is also parallelizable since the $mp$ smaller least-squares problems can be solved independently and concurrently. So parallel computation methods can potentially improve the execution time significantly, though further discussion is beyond the scope of this study and hence will not be discussed further in the current work. 

We note that the ALS algorithm given in Alg.~(\ref{alg:bmp-als}) for computing low BM-rank tensor approximation is also closely connected to the nonlinear GS method. Hence, several convergence results follow directly from its framework. By Proposition 2.1 in \cite{yang2022global}, let $\Vv^{k}$ denote the $k$-th solution vector, i.e. $\Vv^{k}=\left(\Va^{k}, \Vb^{k}, \Vc^{k}\right)$. When the normal equation matrix in each of the linear least-squares subproblems (Eq.~\ref{eq:updateB} -- Eq.~\ref{eq:updateA}) is positive definite, i.e. $\M{H}^{\top}\M{H}\succ 0$ for any $\M{H}\in\{\M{H}_{\TC,\TB}, \M{H}_{\TA,\TC}, \M{H}_{\TB,\TA}\}$, then each least-squares subproblem has a unique solution and the sequence $\{\Vv^k\}_{k\in \mathbb{N}}$ generated by ALS converges to a limit point. Moreover, by Theorem 3.1 in \cite{yang2022global}, if $\{\V{v}^{k}\}_{k\in \mathbb{N}}$ is bounded, then the limit point of the sequence is a stationary point of the problem. Detailed discussions on the ALS convergence analysis is referred to the Supplement. 

As discussed in \cite{li2013some, navasca2008swamp}, there is no guarantee for the ALS algorithm in CP-decomposition to converge to a stationary point when the coefficient matrices in the subproblems are rank-deficient, in which case, the objective function is not strictly quasiconvex. This is also the case for the ALS subproblems for third-order tensor BM-decomposition. The coefficient matrices given may not have full column rank. As a result, each of the least-squares subproblems will not be strictly quasiconvex, and hence the objective function $f(\Vv)$ is not strictly quasiconvex with respect to each component of $\V{v}$. Therefore, similar to the CP-decomposition, the BMD-ALS algorithm may produce a sequence with limit points that are not critical points to the original problem resulting in slow convergence. The remedy for the convergence issue of the ALS algorithm for CP decomposition has been discussed in several studies \cite{grippo2000convergence,li2013some,navasca2008swamp, yang2022global}. 
In the next section, we address the BMD-ALS convergence issue by adding regularization.

\section{Regularized ALS}\label{sec:regALS}
Grouping the slicewise singular values with the corresponding left singular vectors as in Theorem \ref{th:slicewise} was arbitrarily made in order to preserve the scaling constants of the left-singular vectors.  This illustrates a  nonuniqueness property of the BMD.  To see this more generally, let $\MA$ be $m \times n$, $\MB$ be $m \times p$, and $\M{C}$ be $p \times n$. Then define two tensor triplets $(\TA,\TB,\TC), (\Tilde{\TA},\Tilde{\TB},\Tilde{\TC}) \in \mathbb{R}^{m\times 1\times n} \times \mathbb{R}^{m\times p\times 1}\times \mathbb{R}^{1\times p\times n}$ such that
\begin{equation*}
    \begin{split}
    \TA_{:,1,:} = \MA, \qquad \TB_{:,:,1}&=\MB, \qquad \TC_{1,:,:} = \M{C}, \\
    \Tilde{\TA}_{:,1,:} = \MD_1 \MA, \qquad  \Tilde{\MB}_{:,:,1} =& \MD_1^{-1} \MB \MD_2^{-1},
    \qquad \Tilde{\TC}_{1,:,:} = \MD_2 \M{C}, 
    \end{split}
\end{equation*}
where $\MD_1$ and $\MD_2$ are invertible $m \times m$ and $p \times p$ diagonal matrices respectively. Then
\begin{equation*}
    \bmp(\TA,\TB,\TC) = \bmp(\Tilde{\TA},\Tilde{\TB},\Tilde{\TC}),
\end{equation*}
showing that the factors are non-unique.

To this end, we might wish to add regularization. Recall our generative model suggests that one of the lateral slices of $\TA$ should give information about the background scene, and the subsequent SS-SVD analysis suggests that it should correspond to the largest singular value. The $\TB$ and $\TC$ tensors each have information about one spatial dimension and time.  So, the simpliest regularization model would be
\begin{equation}\label{eq:reg-objF}
    \min_{\T{A},\T{B},\T{C}} \left\Vert \T{T} - \bmp\left(\T{A},\T{B},\T{C}\right) \right\Vert_{F}^{2} + \frac{1}{2}\left(\Vert \M{L}_a(\T{A}) \Vert_{F}^{2} + \Vert \M{L}_b(\T{B}) \Vert_{F}^{2} +\Vert \M{L}_c(\T{C}) \Vert_{F}^{2}\right)
\end{equation}
where $\M{L}_a, \M{L}_b$, and $\M{L}_c$ are linear operators acting on the vectorized factor tensors $\T{A},\T{B}$, and $\T{C}$ respectively. More specifically, for all $1\leq i \leq m, 1\leq j \leq p, 1\leq k \leq n$, $\M{L}_{a}=\underset{k,i}{\oplus} \M{L}$  where $\M{L} = \diag([\lambda_1, \lambda_2, \dots, \lambda_{\ell}])$, $\M{L}_{b} = \underset{i,j}{\oplus} \beta \M{I}_{\ell\times \ell}$ and $\M{L}_{c} = \underset{j,k}{\oplus} \gamma\M{I}_{\ell\times \ell}$.  Here, we set $\lambda_1$ to be at least an order of magnitude smaller than $\lambda_2$.  This is consistent with the fact that the dominant feature is the background (requiring the least penalty).  

Importantly, this choice keeps the subproblems decoupled:  the only change required is to add a constraint term to each of Eq.~(\ref{eq:updateB}), Eq.~(\ref{eq:updateC}), and Eq.~(\ref{eq:updateA}) respectively, in the form of either $ \kappa \| \V{v} \|_2^2$ for $\V{v} = \V{b}^{(i,j)}$ or $\V{v} = \V{c}^{(i,j)}$ (Eq.~\ref{eq:decoupledB}, Eq.~\ref{eq:decoupledC} respectively) or $0.5\| \M{L} \V{a}^{(k,i)} \|_2$ (to Eq.~\ref{eq:decoupledA}). 

Let $\Psi(\boldsymbol{\omega})$ denote the objective function given in Eq.~(\ref{eq:reg-objF}) 
where $\boldsymbol{\omega}$ is the solution vector, i.e. $\boldsymbol{\omega} = (\V{a}, \V{b}, \V{c})$. Then the sequence $\{\boldsymbol{\omega}^{k}\}_{k\in \mathbb{N}}$ generated by the regularized alternating least-squares algorithm converges to a critical point of $\Psi$.  We provide the proof in the Supplement.

\section{Numerical Results}
\label{sec:numerical}

In this section, we illustrate the performance of (grayscale) video background and foreground separation with the following methods (1) SS-SVD (2) regularized BMD-ALS with spatiotemporal slice-wise SVD initial guess ($\text{BMD-ALS}_{SVD}$) (3) Dynamic mode decomposition (DMD) \cite{grosek2014dynamic} (4) regularized BMD-ALS with DMD initial guess ($\text{BMD-ALS}_{DMD}$) on six video datasets:
\begin{enumerate}
    \item First is the ``Simulated Video'' based on the BM-rank 3 generative spatiotemporal video model discussed in Sec.~\ref{sec:generative}. We use a $50\times 50$ cloud image \cite{miguel2021cloud} and simulate 2 rectangular objects intensities 85 and 15 respectively and are moving over 30 time steps across the background. 

    \item The second, ``Car'' video is from a surveillance video of moving vehicles on a highway. The video data is available in {\sc Matlab} and can be loaded using the command \verb+VideoReader('traffic.mj2')+. This video consists of 120 grayscale frames each of size $120 \times 160$. The camera is pointing at an entrance of a highway where the cars are traveling from the top right corner to the bottom of the image as frames progress.

    \item The third video, referred to as the ``Escalator'' video, is taken from a surveillance video of escalators and people. The data was originally used in \cite{zhang2014novel} which is now available in the author's github repository \cite{ZhangGit2014}. This video consists of 200 frames in grayscale and each frame is of size $130 \times 160$. The camera is pointing directly from above at three parallel escalators. The staircase of the escalators is moving periodically while people either walk across the platforms above the escalator, stand on the escalators or walk down the escalator on the right. 

    \item The last three videos are from the the Scene Background Initialization (SBI) dataset \cite{bouwmans2017scene, maddalena2015towards}. The three videos include ``Hall and Monitor'', ``Human Body'', and ``IBM Test''. The video frames are re-scaled to sizes $100 \times 147$, $100 \times 134$, and $100 \times 134$ respectively. For videos ``Hall and Monitor'' and ``Human Body'', the first 150 frames are included for the following experiments. The total number of frames of the ``IBM Test'' video is 90.  
\end{enumerate}
Specific (grayscale) video frames of each of the six video datasets are displayed in Fig.~(\ref{fig:tst_vid}). In Sec.~\ref{sec:colorBMD}, we will demonstrate a way to extend the third-order tensor BMD to fourth-order color video tensors and show the background/foreground separation results for the color videos available including the ``Simulated Video'' and the SBI video data tensors.
\begin{figure}[ht]
\centering
\begin{tikzpicture}
\node at (0,0) {\includegraphics[width=0.98\linewidth]{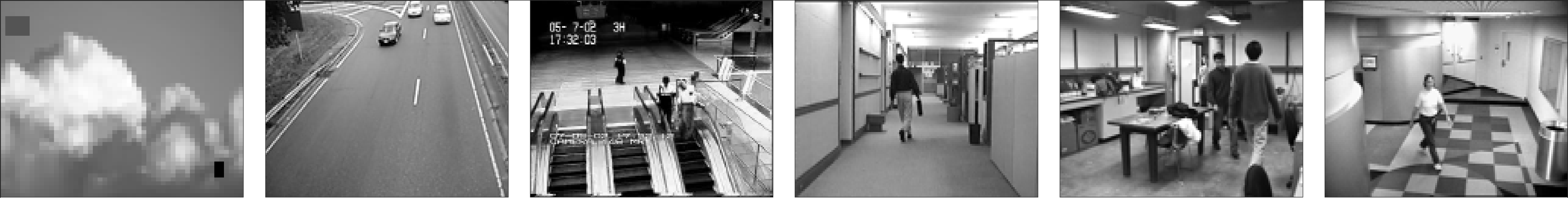}};
\node at (-6.5, 1.2) {\footnotesize Simulated Video};
\node at (-3.8, 1.2) {\footnotesize Car};
\node at (-1.3, 1.2) {\footnotesize Escalator};
\node at (1.25, 1.2) {\footnotesize Hall and Monitor};
\node at (3.85, 1.2) {\footnotesize Human Body};
\node at (6.35, 1.2) {\footnotesize IBM Test};
\node at (-6.5, -1.2) {\footnotesize Frame 1};
\node at (-3.8, -1.2) {\footnotesize Frame 60};
\node at (-1.3, -1.2) {\footnotesize Frame 60};
\node at (1.25, -1.2) {\footnotesize Frame 60};
\node at (3.85, -1.2) {\footnotesize Frame 60};
\node at (6.35, -1.2) {\footnotesize Frame 60};
\end{tikzpicture}
\caption{Video frames selected to display from the six testing videos.}
\label{fig:tst_vid}
\end{figure}

The choices of these videos are based on the following considerations. The first video is a simulated data tensor which we use to check the generative video model results discussed in Sec.~\ref{sec:generative}. The second and the third video tensors differ by the level of the background image complexities. While the highway road is a simpler image, the escalator scene represents a more complex background image which a higher tensor BM-rank is expected to be used in order to achieve a similar level of reconstruction accuracy as the ``Car'' video. The last three videos taken from the SBI datasets contain foreground objects of different sizes comparing to the size of the image. 

For all numerical experiments in the third-order grayscale video decomposition case, we found that following regularization parameters in the alternating least-squares algorithm provide sufficiently good reconstruction results. For penalizing the $\TA$ factor tensor, we set $\lambda_1=0.01$ and $\lambda_t=1, \forall\,2\leq t\leq \ell$, where $\M{L}_a=\underset{k,i}{\oplus}\M{L}$ with $\M{L}=\diag([\lambda_1,\lambda_2,\dots,\lambda_{\ell}])$, $1\leq k\leq p$ and $1\leq i\leq m$. We found in experiments that these regularization parameters are robust as long as the first parameter $\lambda_1$ is approximately 100 times smaller than other parameters $\lambda_2, \dots, \lambda_{\ell}$. For penalizing the $\T{B}$ and $\T{C}$ factor tensors, we set $\beta=\gamma=1$ so that $\M{L}_{b} = \underset{i,j}{\oplus} \M{I}_{\ell \times \ell}$ and $\M{L}_{c} = \underset{j,k}{\oplus}\M{I}_{\ell \times \ell}$, for all $1\leq i\leq m$, $1\leq j\leq n$, and $1\leq k\leq p$. 

Moreover, we set the termination criteria for the alternating least-squares algorithm to be either when a total number of 100 iterations is reached and or the relative error of consecutive iterates, i.e. $\|\TX^{k+1} - \TX^{k}\|_{F}/\|\TX^{k}\|_{F}$, is less than $10^{-5}$. For the following numerical experiments, all video frames displayed are re-scaled to the pixel range of $[0, 255]$ on grayscale. For both of the DMD method and the $\text{BMD-ALS}_{DMD}$ method, since the resulting dynamic modes may be complex, we followed the code included in \cite{kutz2016dynamic} and the real parts of the solutions were taken for the video reconstructions. 

In Fig.~(\ref{fig:rank_RE}), we compare the relative error (RE\footnote{Relative Error: $\text{RE} = \frac{\|\TX - \hat{\TX}\|_{F}}{\|\TX\|_{F}}$.}) of video reconstruction results of the BMD-ALS algorithm with both the SS-SVD and the DMD initial guesses for different BM-rank $\ell$, $\ell=2,\dots,10$. The BM-decomposition is computed on the entire video tensor. As we can see in Fig.~(\ref{fig:rank_RE}), The BMD-ALS algorithm with both starting guesses performed similarly well for different BM-ranks on real video tensors. As for the simulated video constructed with three BM-rank 1 terms, a BM-rank 2 approximating by the $\text{BMD-ALS}_{SVD}$ algorithm outperforms the $\text{BMD-ALS}_{DMD}$ algorithm by a factor of 10. Additionally, we also observe that the approximation relative error for the simulated video is small even when $\ell$ is larger than the number of BM-rank 1 terms of the video. Overall, as we can see in the plot that, as $\ell$ increases, the amount of the improvement of the approximation quality decreases since the differences between RE become smaller for larger $\ell$. This result suggests that surveillance type videos inherently have a low BM-rank. Based on Fig.~(\ref{fig:rank_RE}), we determined the target BM-ranks for the real videos having a similar level of reconstruction accuracy with RE approximately $0.04$.  In Table.~\ref{tbl:BMD_results}, we summarize the target BM-ranks for each video as well as the compression ratio (CR\footnote{Compression Ratio: $\text{CR} = \displaystyle{\frac{\text{Uncompressed size}}{\text{Compressed size}}}$.}) for the BMD method.

Next we illustrate in detail the quality of separating the background and foreground using the BM-decomposition method, and compare with the methods of the SS-SVD and DMD. We note that for the DMD method, video sequences are usually processed on segments. Choosing a proper video segment size can reduce the processing time and potentially improve the video reconstruction accuracy. For the real-video streams, the sequences are broken into segments of $30$ frames, and the rank truncation parameter is taken as $\ell=30-1=29$ as suggested in \cite{kutz2016dynamic}. For the simulated video, we found that choosing a segment size of 10 with rank $\ell=9$ yields a much better result. All other methods including SS-SVD, $\text{BMD-ALS}_{SVD}$, and $\text{BMD-ALS}_{DMD}$ are applied to the entire video sequence. 

We also note that both the SS-SVD and the DMD methods are algorithms for the background initialization application in computer vision \cite{garcia2020background}. These methods aim to model a clear background image without foreground objects from a video sequence. Obtaining the foreground objects are usually done by subtracting the background image from the overall video sequence. For the SS-SVD method, we followed the suggestions given in \cite{kajo2018svd} and selected the first rank-1 component from the slice-wise SVDs for the video stationary background sequence reconstruction. However, it is also possible to include more components to improve the reconstructed background quality as the authors discussed in \cite{kajo2018svd}. Additionally, the first image frame of the sequence is used to represent the approximated background, since the stationary background image sequence consists of the same image repeated $p$ times. As for the DMD method, we take $t_1=0$ in Eq.~(\ref{eq:dmd_bg}) to represent the background image. Detailed derivations and discussions for the DMD method are referred to the original article \cite{kutz2016dynamic}.

For our method, we model both the background image sequence and the foreground objects' motion using a few BM-rank 1 terms. In particular, the background video sequence is reconstructed by ${\T{X}}^{\text{bg}}_{\text{BMD}} = \bmp\left(\T{A}_{:,1,:}, \T{B}_{:,:,1}, \T{C}_{1,:,:}\right)$, and the foreground video sequence is reconstructed by $\displaystyle{\T{X}^{\text{fg}}_{\text{BMD}} =\sum_{t=2}^{\ell} \bmp\left(\T{A}_{:,t,:}, \T{B}_{:,:,t}, \T{C}_{t,:,:}\right)}$, where $\ell$ is the target BM-rank.

In Table.~\ref{tbl:bk_evals}, we summarize the video background reconstruction comparisons of the four methods: SS-SVD, $\text{BMD-ALS}_{SVD}$, DMD, and $\text{BMD-ALS}_{DMD}$, for the simulated video, and the three SBI video sequences where the ground-truth background images are available. The following metrics for the background evaluations were measured: Average Gray-level Error (AGE), Percentage of Error Pixels (pEPs), Percentage of Clustered Error Pixels (pCEPs), Peak-Signal-to-Noise-Ratio (PSNR), and Multi-Scale Structural Similarity Index (MS-SSIM). For the metics AGE, pEPs, and pCEPs, the lower the values, the better is the background estimate. For the PSNR and MS-SSIM values, the higher the values, the better is the background estimate. Detailed descriptions of the these metrics and \textsc{Matlab} implementations are available in \cite{SBIdataset2016}. For the SS-SVD and the DMD methods, the first frame of the reconstructed background sequence are compared with the ground-truth background images. For the BMD method with both SS-SVD and DMD initial guesses, the metrics are measured against every frame in the reconstructed background video sequence, and the average values over all frames are taken and summarized in the table. As we can see from Table.~(\ref{tbl:bk_evals}), the $\text{DMD-ALS}_{SVD}$ method performs well for the ``Simulated'', the ``Hall and Monitor'', and the ``Human Body'' video sequences. 

\begin{figure}[ht]
\centering
\begin{tikzpicture}[every text node part/.style={align=center}]
\node at (0,0) {\includegraphics[width=0.9\linewidth]{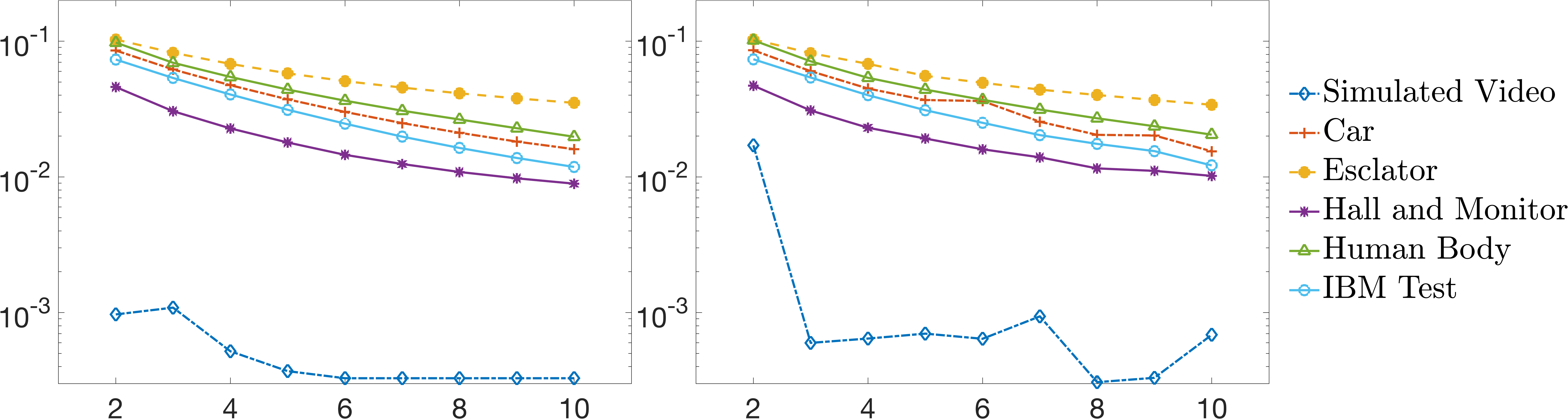}};
\node at (-4, 2.2) {\footnotesize $\text{BMD-ALS}_{SVD}$};
\node at (2, 2.2){\footnotesize $\text{BMD-ALS}_{DMD}$};
\node[rotate=90] at (-7.3, 0) {\footnotesize RE};
\node at (-4, -2.2) {\footnotesize BM-rank $\ell$};
\node at (2, -2.2){\footnotesize BM-rank $\ell$};
\end{tikzpicture}
\caption{Comparison of the $\text{BMD-ALS}_{SVD}$ and the $\text{BMD-ALS}_{DMD}$ methods with change in BM-rank for the six testing videos.}
\label{fig:rank_RE}
\end{figure}

In Fig.~(\ref{fig:sim_result1}), the first video background and foreground frames are displayed for the simulated video. From left to right are respectively: the ground--truth frames, and frames reconstructed by the four methods: SS-SVD, DMD, $\text{BMD-ALS}_{SVD}$, and $\text{BMD-ALS}_{DMD}$. As we can see from the images, the BMD-ALS method with both initial guesses obtained great separation results of the stationary background and the moving foreground objects compared to the ground-truth frames. In particular, the $\text{BMD-ALS}_{SVD}$ results are almost identical to the ground-truth frames. 

Furthermore, as discussed in Sec.~\ref{sec:generative}, for the simulated video with two objects of different intensities and are traveling with different trajectories, the first BM-rank 1 term of the generative video model consists of the background image captured by $\TA_{:,1,:}$. The foreground component, denoted $\T{X}^{\text{fg}}=\bmp(\TA_{:,2,:},\TB_{:,:,2},\TC_{2,:,:})+\bmp(\TA_{:,3,:},\TB_{:,:,3},\TC_{3,:,:})$, consists of the vertical locations of the objects captured by columns of the slice $\TB_{:,:,2}$ for the first object, and $\TB_{:,:,3}$ for the second object respectively. The horizontal positions of the two objects are modeled by rows of the slices $\squeeze(\TC_{2,:,:})$ and $\squeeze(\TC_{3,:,:})$. In Fig.~(\ref{fig:sim_svd2}.a), we demonstrated the ground-truth factor tensor slices of the BM-rank 3 generative video model in the first row. Additionally, the BM-rank 2 decomposition of the simulated video by both the $\text{BMD-ALS}_{SVD}$, and the $\text{BMD-ALS}_{DMD}$ methods are shown in the second and the third rows of Fig.~(\ref{fig:sim_svd2}.a) respectively. We note that in the tensor $\text{BMD-ALS}_{SVD}$ results, the background slice $\TA_{:,1,:}$ have inverted pixel intensities compared to the ground-truth image and hence we multiplied $\TA$ by $-1$ for display purposes. To make the overall BM-product positive, the factor tensor $\TC$ is multiplied by $-1$. 

Overall, as we can see in Fig.~(\ref{fig:sim_svd2}.a), the BMD-ALS results with both initial guesses recover the stationary background image captured in $\TA_{:,1,:}$. However, regarding the foreground motion, the decomposition with the DMD initial guess is less accurate for obtaining information depicting the trajectory of the motion comparing with using the SS-SVD initial guess. Moreover, the object trajectories in a BM-rank 2 decomposition using the $\text{BMD-ALS}_{SVD}$ method is well aligned with the ground-truth that was constructed by a sum of three BM-rank 1 terms. This observation suggests that this simulated video could potentially be modeled exactly by a BM-rank 2 model, but further investigations are needed. Finally, the ALS convergence plots are shown in Fig.~(\ref{fig:sim_svd2}.b). As we can see in the plots, the $\text{BMD-ALS}_{SVD}$ algorithm outperforms the $\text{BMD-ALS}_{DMD}$ algorithm with a much smaller relative error at convergence. 

\begin{figure}[ht]
\centering
\begin{tikzpicture}[every text node part/.style={align=center}]
\node at (0,0) {\includegraphics[width=0.85\linewidth]{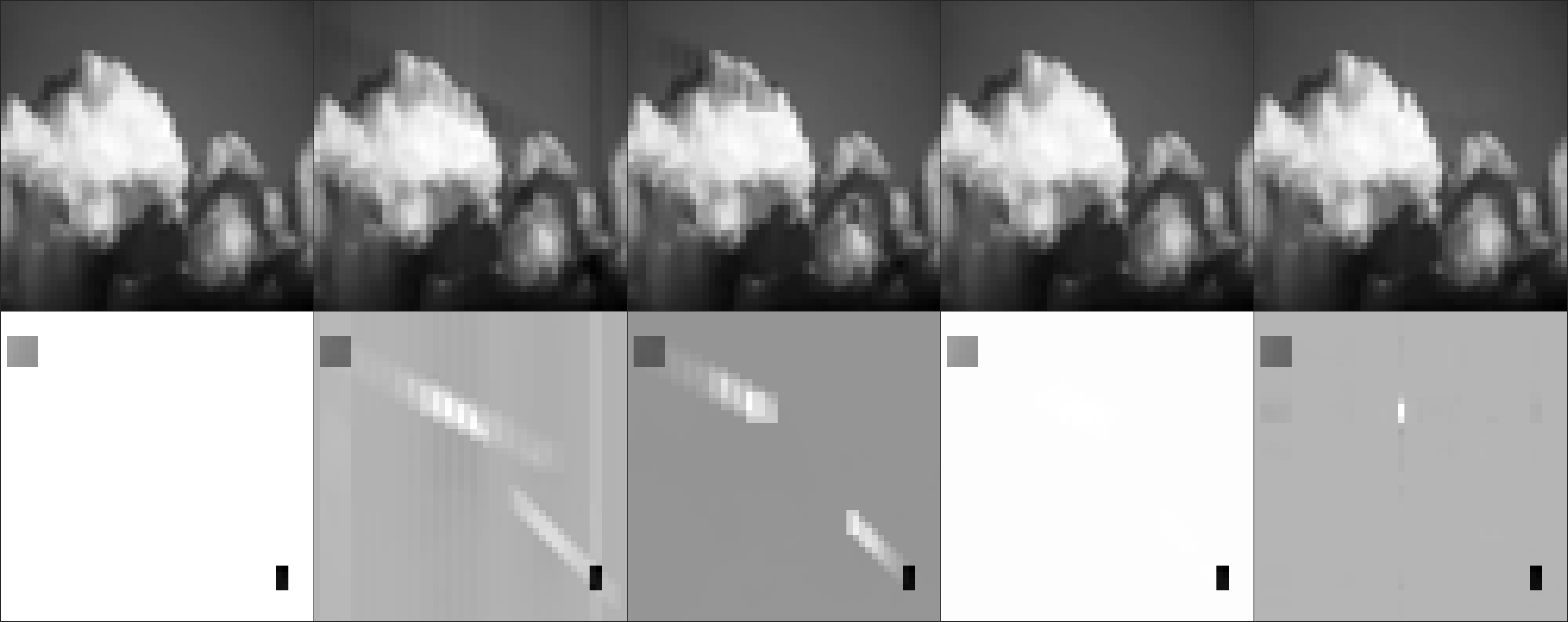}};
\node[rotate=90] at (-6.9, 1.3) {\scriptsize Background};
\node[rotate=90] at (-6.9, -1.3){\scriptsize Foreground};
\node at (-5.2, 3-0.1) {\scriptsize Ground Truth};
\node at (-5.2+2.6, 3-0.1) {\scriptsize SS-SVD};
\node at (-5.2+2.6*2, 3-0.1) {\scriptsize DMD};
\node at (-5.2+2.6*3, 3-0.1) {\scriptsize $\text{BMD-ALS}_{SVD}$};
\node at (-5.2+2.6*4, 3-0.1) {\scriptsize $\text{BMD-ALS}_{DMD}$};
\end{tikzpicture}
\caption{Video background/foreground separation results to the simulated BM-rank 3 video with $\ell=2$ 
approximations. Comparison of the ground-truth, the SS-SVD method, the BMD-ALS method with slice-wise SVD initial guess ($\text{BMD-ALS}_{SVD}$), the DMD algorithm, and the BMD-ALS method with the DMD initial guess ($\text{BMD-ALS}_{DMD}$). The first frame is selected to display.}
\label{fig:sim_result1}
\end{figure}
\begin{figure}[ht]
\centering
\begin{tikzpicture}[every text node part/.style={align=center}]
\node at (-7.85, 4.6) {\footnotesize $(a)$};
\node at (-4.5, 0) {\includegraphics[width=0.38\linewidth]{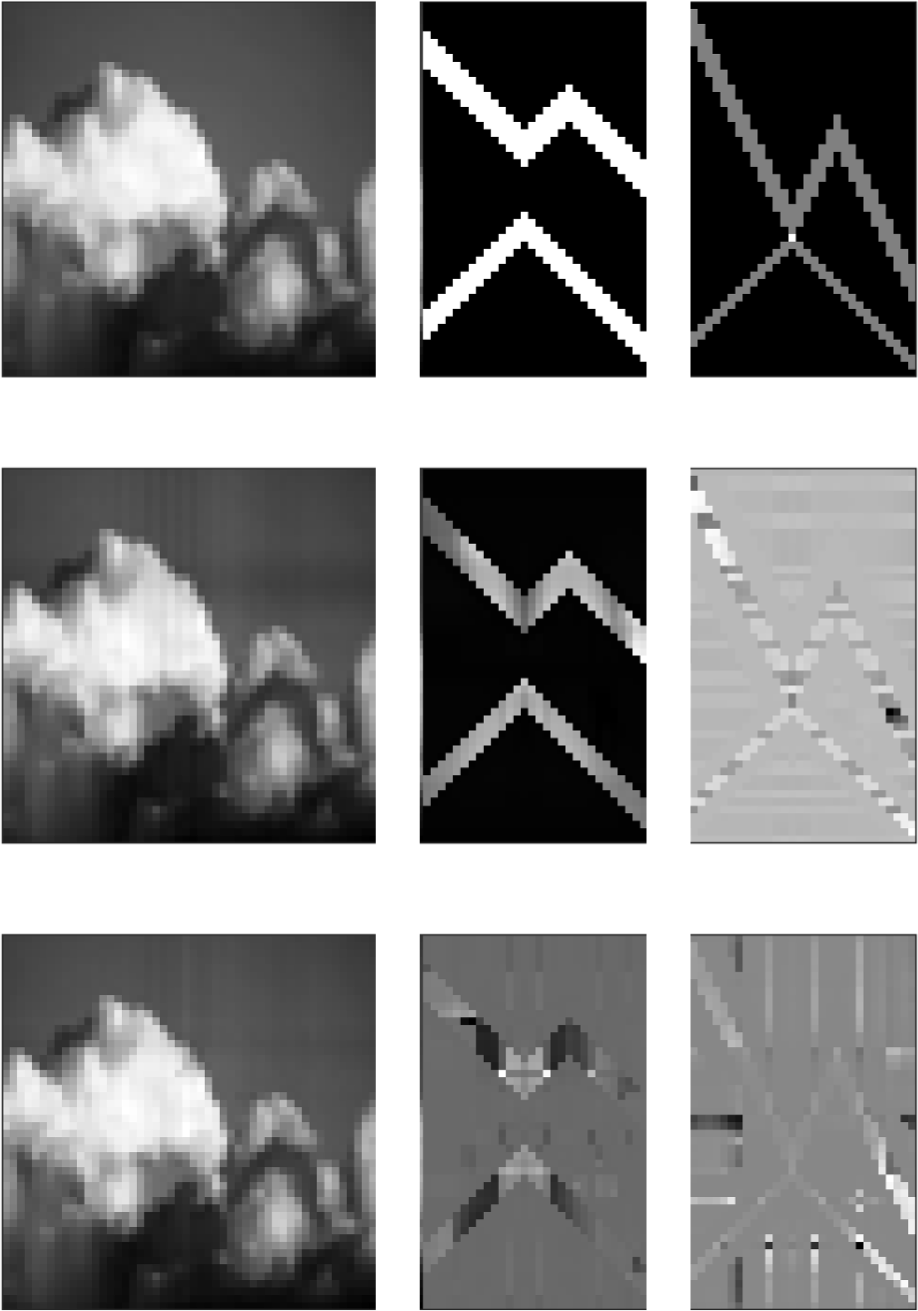}};
\node[rotate=90] at (-7.8, 2.95) {\footnotesize Ground-Truth};
\node[rotate=90] at (-7.8, 0) {\footnotesize $\text{BMD-ALS}_{SVD}$};
\node[rotate=90] at (-7.8, -3) {\footnotesize $\text{BMD-ALS}_{DMD}$};
\node at (-6.25, 4.5) {\footnotesize $\squeeze(\TA_{:,1,:})$};
\node at (-4, 4.5) {\footnotesize $\TB_{:,:,2}+\TB_{:,:,3}$};
\node at (-1.2, 4.5) {\footnotesize $\squeeze(\TC_{2,:,:}+\TC_{3,:,:})^{\top}$};
\node at (-6.25, 1.5) {\footnotesize $\squeeze(\hat{\TA}_{:,1,:})$};
\node at (-4, 1.5) {\footnotesize $\hat{\TB}_{:,:,2}$};
\node at (-2.2, 1.5) {\footnotesize $\squeeze(\hat{\TC}_{2,:,:})^{\top}$};
\node at (-6.25, -1.5) {\footnotesize $\squeeze(\hat{\TA}_{:,1,:})$};
\node at (-4, -1.5) {\footnotesize $\hat{\TB}_{:,:,2}$};
\node at (-2.2, -1.5) {\footnotesize $\squeeze(\hat{\TC}_{2,:,:})^{\top}$};

\node at (-0.7+0.5, 2.7) {\footnotesize $(b)$};
\node at (2.7+0.5, 0) {\includegraphics[width=0.4\linewidth]{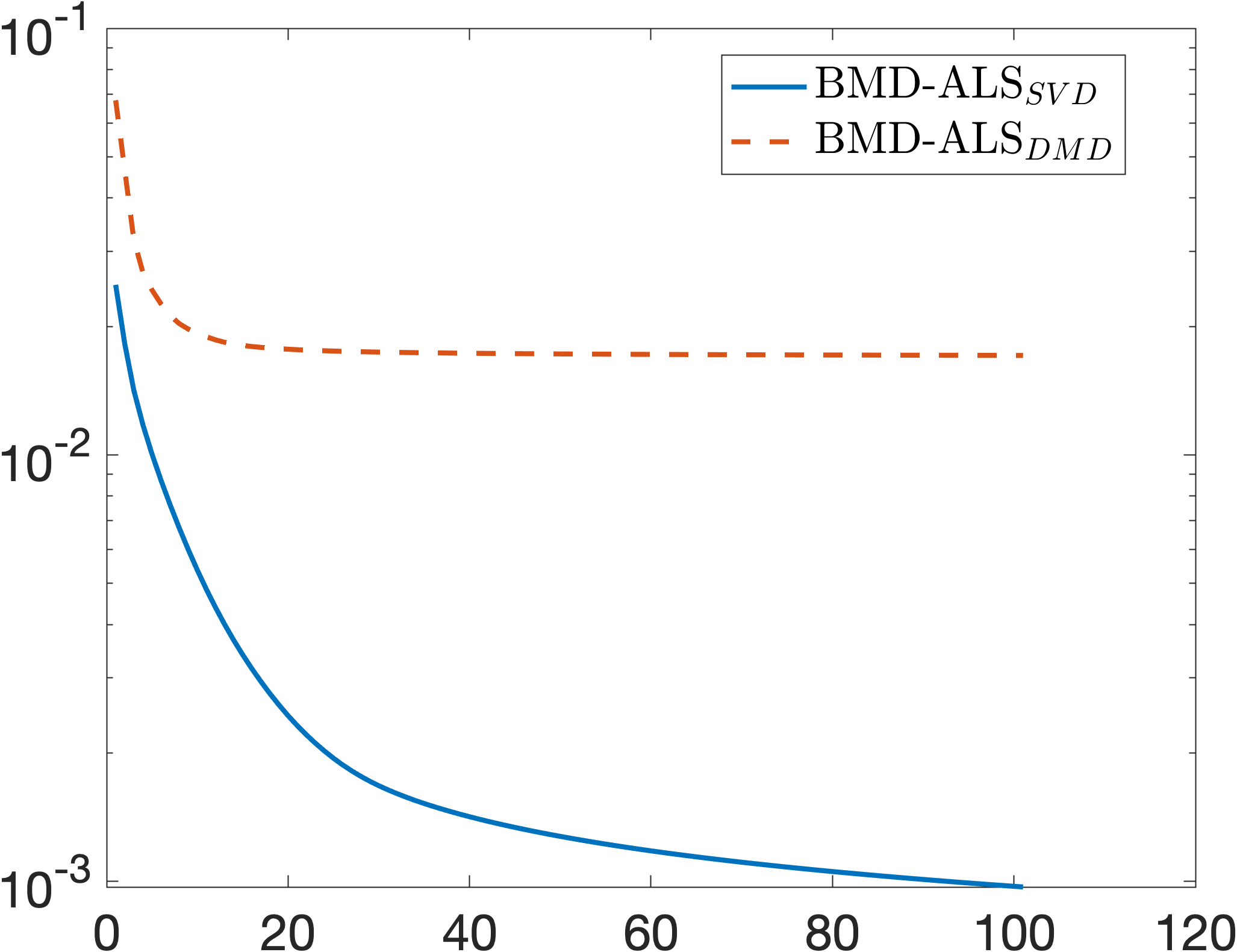}};
\node[rotate=90] at (-0.7+0.5, 0) {\footnotesize Relative Error};
\node at (3+0.5, 2.6) {\footnotesize ALS convergence};
\node at (3+0.5, -2.8) {\footnotesize No. of iterations};
\end{tikzpicture}
\caption{Comparison of the BM-rank 2 approximation factor tensors and ALS convergence to the simulated video with spatiotemporal slice-wise SVD initial guess ($\text{BMD-ALS}_{SVD}$) and the DMD initial guess ($\text{BMD-ALS}_{DMD}$). (a) The ground-truth factor tensor slices with $\squeeze(\TA_{:,1,:})$ capturing the stationary background image, $\TB_{:,:,2}$ depicting the left-motion of the object and  $\squeeze(\TC_{2,:,:})^{\top}$ depicting the right-motion of the object. (b) The ALS convergence comparisons for the two methods. The relative error is computed for the overall video tensor $\frac{\|\TX-\hat{\TX}\|_F}{\|\TX\|_F}$.}
\label{fig:sim_svd2}
\end{figure}
Fig.~(\ref{fig:real_bg}) shows the ground-truth background images, the first frames of the SS-SVD and DMD reconstructed background, and the 60th frames of the BMD reconstructed background video sequences. We compare qualitatively the four methods. As we can see in the figure, all four methods perform comparably well for approximating the stationary road background in the ``Car'' video sequence. Although for the DMD method, the foreground car motion trajectories are visible in the reconstructed background image. For both the``Escalator'' and the ``Human Body'' video sequences, vertical artifacts appeared for the SS-SVD, BMD-ALS$_{SVD}$ and BMD-ALS$_{DMD}$ methods. As for the DMD method, we observe similar foreground motions present in the background, which is also the case for the other two videos ``Hall and Monitor'' and ``IBM Test''. Additionally, in the ``IBM Test'' video, we observe a rectangular artifact appeared at the top left corner for the BMD-ALS$_{DMD}$ result. In our experiments, we noticed this artifact could be eliminated if we increase the BM-rank. Overall, all four methods provide visually appealing stationary background reconstructions. This result aligns with the significance of application of the SS-SVD method and the DMD method in background subtraction for surveillance video. In Fig.~(\ref{fig:real_fg}), the 60-th frame of each foreground video sequence is displayed for the four methods of interest. Both the SS-SVD and the DMD foreground frames are obtained by subtracting the background frame shown in Fig.~(\ref{fig:real_bg}) from the original video. For the BMD method, we show the 60-th foreground frame from the reconstructed sequence $\displaystyle{\T{X}^{\text{fg}}_{\text{BMD}}}$.Overall, we observe that the background subtraction methods have the drawback of unable to remove the background information completely when the background is not perfectly approximated. For instance, in the ``Human Boday'' and ``IBM Test'' video sequences, the office background and the checkerboard floor are also visible in the foreground frames as shown in Fig.~(\ref{fig:real_fg}). However, in the BMD-ALS$_{SVD}$ reconstructed foreground frames, we can see that the background information are clearly removed. Despite the artifacts, the people and their cloths are well approximated with great details comparable to the ground-truth. Additionally, we see that when the object size is small relative to the frame size, the foreground frames are approximated with better visual quality for all four methods, such as the frames of the ``Car'', ``Escalator'' and ``IBM Test'' video sequences shown in Fig.~(\ref{fig:real_bg}). For the videos ``Hall and Monitor'' and ``Human Body'', we observe more artifacts in the BMD results when the foreground objects are relatively large comparing to the frame size. A larger BM-rank is needed to recover better foreground results. Moreover, we notice that between the two different initializations of SS-SVD and DMD, the BMD-ALS algorithm with SS-SVD initial guess often outperforms the DMD initial guess making it a more suitable choice for the video background/foreground separation application. In particular, in the next section we will demonstrate the impressive fourth-order color BMD results using only the SS-SVD initial guess. 

\input{table}

\begin{figure}[ht]
\centering
\begin{tikzpicture}
\node at (0,0) {\includegraphics[width=0.85\linewidth]{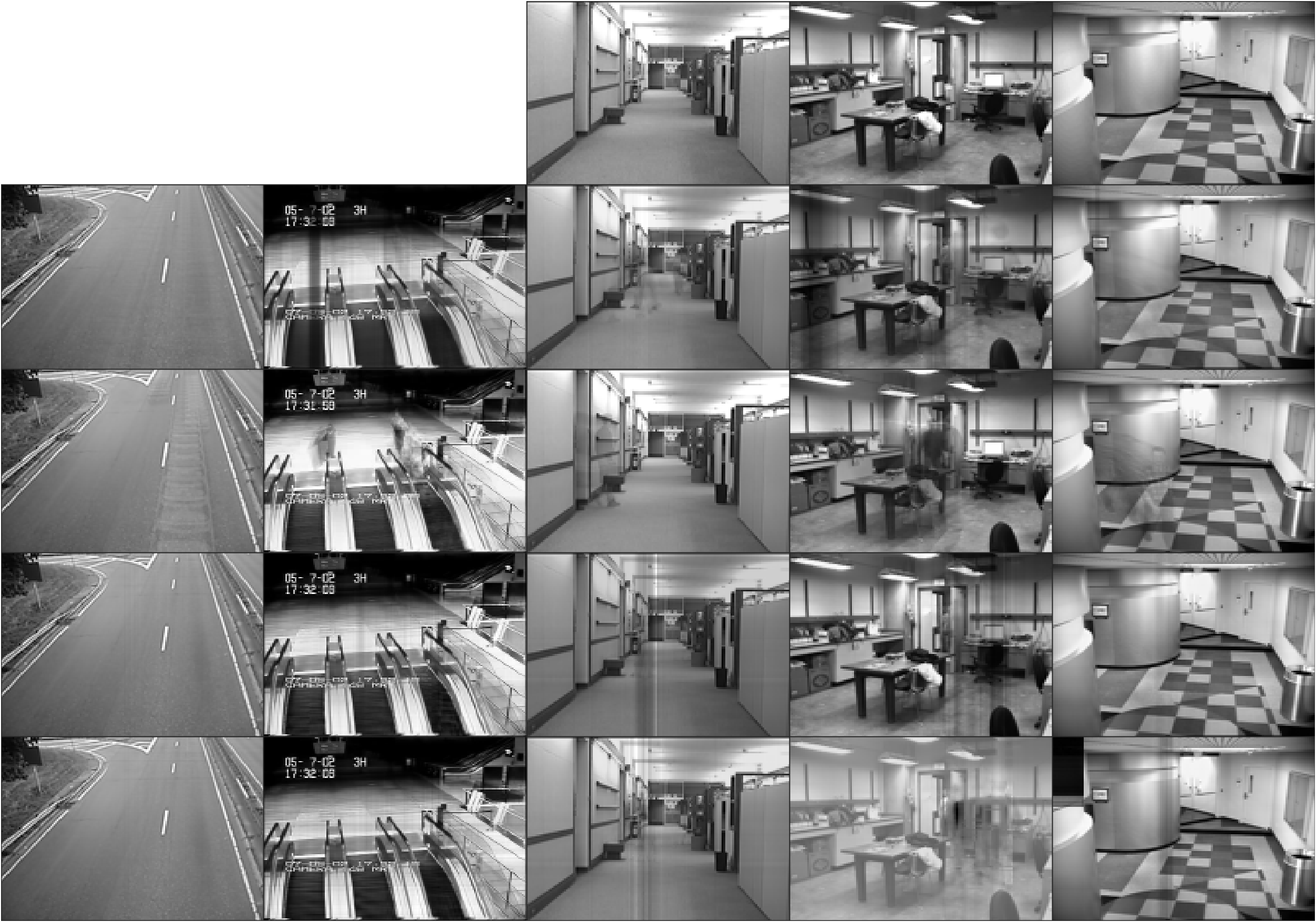}};
\node[rotate=-90] at (7, 3.8) {\scriptsize Ground-truth};
\node[rotate=-90] at (7, 1.8) {\scriptsize SS-SVD};
\node[rotate=-90] at (7, 0.2) {\scriptsize DMD};
\node[rotate=-90] at (7, -1.7) {\scriptsize $\text{BMD-ALS}_{SVD}$};
\node[rotate=-90] at (7, -4) {\scriptsize $\text{BMD-ALS}_{DMD}$};
\node at (0, 5.1) {\footnotesize Background Reconstruction};
\node at (-5.3, -5) {\scriptsize Car};
\node at (-2.7, -5) {\scriptsize Escalator};
\node at (0, -5) {\scriptsize Hall and Monitor};
\node at (2.7, -5) {\scriptsize Human Body};
\node at (5.4, -5) {\scriptsize IBM Test};
\end{tikzpicture}
\caption{Background reconstruction. Comparison of the four methods: 1. SS-SVD; 2. $\text{BMD-ALS}_{SVD}$; 3. DMD; 4. $\text{BMD-ALS}_{DMD}$.}
\label{fig:real_bg}
\end{figure}

\begin{figure}[!ht]
\centering
\begin{tikzpicture}
\node at (0,0) {\includegraphics[width=0.85\linewidth]{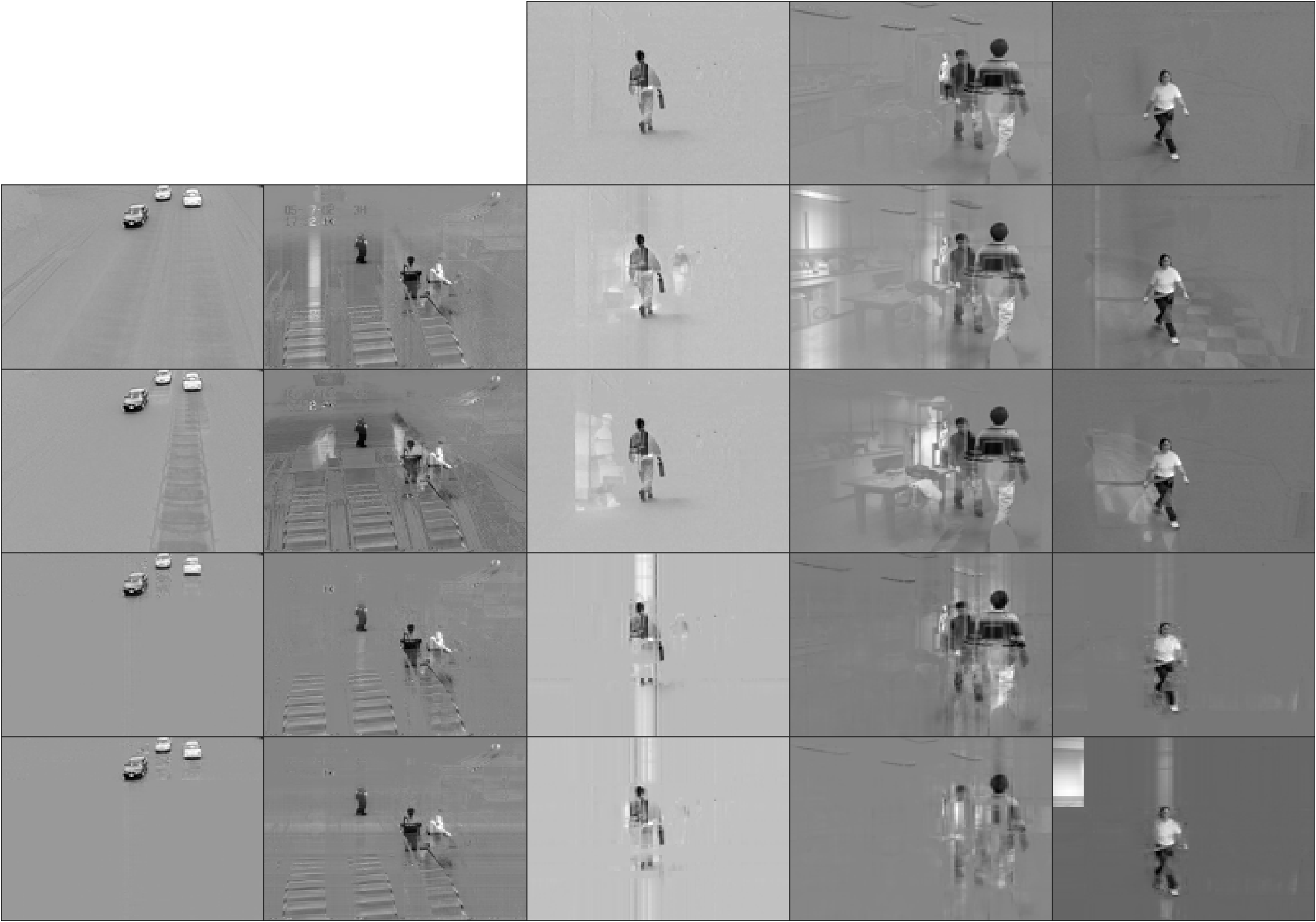}};
\node[rotate=-90] at (7, 3.8) {\scriptsize Ground-truth};
\node[rotate=-90] at (7, 1.8) {\scriptsize SS-SVD};
\node[rotate=-90] at (7, 0.2) {\scriptsize DMD};
\node[rotate=-90] at (7, -1.7) {\scriptsize $\text{BMD-ALS}_{SVD}$};
\node[rotate=-90] at (7, -4) {\scriptsize $\text{BMD-ALS}_{DMD}$};
\node at (0, 5.1) {\footnotesize Foreground Reconstruction};
\node at (-5.3, -5) {\scriptsize Car};
\node at (-2.7, -5) {\scriptsize Escalator};
\node at (0, -5) {\scriptsize Hall and Monitor};
\node at (2.7, -5) {\scriptsize Human Body};
\node at (5.4, -5) {\scriptsize IBM Test};
\end{tikzpicture}
\caption{Foreground reconstruction: frame 60. Comparison of the four methods: 1. SS-SVD; 2. $\text{BMD-ALS}_{SVD}$; 3. DMD; 4. $\text{BMD-ALS}_{DMD}$.}
\label{fig:real_fg}
\end{figure}

\section{Decomposing a Fourth-Order Tensor}
\label{sec:colorBMD}
The formal definition of the BM product for fourth order tensors described in \cite{gnang2017spectra} requires four factor tensors. The computation time and storage quickly become intractable for high order BM products.  In the context we consider here, using a fourth order BMP would be overkill and too expensive. We therefore define a new BM-product-like operation, called \textbf{\textsc{bmp4}}, and corresponding BM-product-like decomposition for fourth order tensors from triples of fourth order tensors: 

\begin{definition}\label{def:bmp4th}
    Given a tensor triplet $\T{A}\in\mathbb{R}^{m\times\ell\times n\times q}, \T{B}\in\mathbb{R}^{m\times p\times\ell\times q}$, and $\T{C}\in\mathbb{R}^{\ell\times p\times n\times q}$, which are conformable on the first three indices, then $\T{X} = \bmpF\left(\T{A},\T{B},\T{C}\right) \in\mathbb{R}^{m\times p\times n\times q}$, is specified entry-wise as follows,
\begin{equation}
    \T{X}_{i,j,k,z}=\sum_{1\leq t\leq\ell}\T{A}_{i,t,k,z}\T{B}_{i,j,t,z}\T{C}_{t,j,k,z}
\end{equation}
for all $1\leq i\leq m,1\leq j\leq p,1\leq k\leq n$, and $1\leq z\leq q$.  
\end{definition}
When $\ell=1$, a $\text{BM}_4$-outer product corresponds to the $\text{BMP}_4$ of the conformable order-3 tensor slices $\T{A}\in \mathbb{R}^{m\times 1\times n\times q}$, $\T{B}\in \mathbb{R}^{m\times 1\times n\times q}$, and $\T{C}\in \mathbb{R}^{1\times p\times n\times q}$. Then the corresponding notion of rank can be defined as follows.
\begin{definition}
The $\text{BM}_4$-rank of $\TX \in \mathbb{R}^{m\times p\times n \times q}$ is the minimum number of the $\text{BM}_4$-outer products of conformable tensor slices that sum up to $\TX$.
\end{definition}

We note that our definition of the fourth-order BM-product given above multiplies a tensor triplet, in which case, the multiplication is orientation dependent with one of the dimensions held fixed. This is different than the orientation independent fourth order BMP given in \cite{gnang2017spectra} which would have required an additional tensor factor and would not scale well for the video application.  Our product is designed to respect relationship among the color channels.     

\subsection{An ALS Algorithm of BMD for Fourth-Order Tensors}\label{sec:als4d}
Using {\tt bmp4} defined above, we can extend the BM-decompostion to fourth-order video tensor. Given $\T{X} \in \mathbb{R}^{m\times p\times n\times q}$, a low $\text{BM}_4$-rank approximation to $\T{X}$ is obtained by solving
\begin{equation}
    \min_{\T{A},\T{B},\T{C}}\left\Vert \T{X}-\text{bmp4}\left(\T{A},\T{B},\T{C}\right)\right\Vert _{F}^{2},
\end{equation}
where $\left(\T{A},\T{B},\T{C}\right) \in \mathbb{R}^{m\times \ell \times n\times q}\times \mathbb{R}^{m\times p \times \ell\times q} \times  \mathbb{R}^{\ell\times p \times n\times q}$. We can similarly solve the problem by an alternating least-squares algorithm to obtain the fourth order tensor BMD. 

First, we set the initial guess $\T{B}^{0}$ to be all ones of size $m\times p\times \ell \times q$. We then obtain the initial guess pair $\T{A}^{0}$ and $\T{C}^{0}$ by fixing the fourth dimension and applying either a matrix decomposition based or tensor slice-wise decomposition based methods as described in Phase I (Sec.~\ref{subsec:phase1}) to the order-3 tensor slices $\T{X}_{:,:,:,z}$, for all $1\leq z \leq q$. The order-3 tensor slice decomposition results in the pair $\T{A}^{0}$ and $\T{C}^{0}$. Next, to update $\TB$, we modify the algorithm given in Phase II (Sec.~\ref{subsec:phase2}) which requires solving the following problem
\begin{equation*}
    \widehat{\T{B}} = \min_{\T{B}\in\mathbb{R}^{m\times p\times\ell\times q}}\left\Vert \T{X}-\bmpF\left(\T{A}^{0},\T{B},\T{C}^{0}\right)\right\Vert _{F}^{2}.
\end{equation*}
Similar to the third-order case, we show that this problem can be reduced to parallelizable smaller least-squares subproblems. 

By the definitions of the Frobenius norm and forth-order tensor BMP, we have
\begin{equation}\label{eq:FnormBMP4}
\left\Vert \T{X}-\bmpF\left(\T{A}^{0},\T{B},\T{C}^{0}\right)\right\Vert _{F}^{2}=\sum_{i=1}^{m}\sum_{j=1}^{p}\sum_{k=1}^{n}\sum_{z=1}^{q}\left|\T{X}_{i,j,k,z}-\sum_{t=1}^{\ell}\T{A}_{i,t,k,z}^{0}\T{B}_{i,j,t,z}\T{C}_{t,j,k,z}^{0}\right|^{2}.
\end{equation}
Holding the indices $i$, $j$, and $z$ fixed, we regroup the remaining indices and rewrite the right-hand side of Eq.~(\ref{eq:FnormBMP4}) as
\begin{equation*}
    \sum_{(i-1)p+j=1}^{mp} \sum_{(z-1)n+k=1}^{nq} \left|\V{x}_{k}^{\left(i,j\right)}-\sum_{t=1}^{\ell}\M{H}_{k,t}^{\left(i,j\right)}\V{b}_{t}^{\left(i,j\right)}\right|^2,
\end{equation*}
where $\V{x}^{(i,j)} = \vvec(\squeeze(\T{X}_{i,j,:,:})) \in \mathbb{R}^{nq\times 1}$, $\V{b}^{(i,j)} = \vvec(\squeeze(\T{B}_{i,j,:,:})) \in \mathbb{R}^{\ell q \times 1}$, and $\M{H}^{(i,j)} \in \mathbb{R}^{nq \times \ell q}$ with entries specified as $\M{H}^{(i,j)}_{k,t} = \T{A}_{i,t,k,z}^{0}\TC_{t,j,k,z}^{0}$. This expression suggests that $\T{B}$ can be updated via solving following parallelizable smaller linear least-squares subproblems
\begin{equation}\label{eq:4Dsub-ls}
    \widehat{\V{b}}^{\left(i,j\right)} = \min_{\V{b}^{\left(i,j\right)}} \left\Vert \V{x}^{\left(i,j\right)}-\M{H}^{\left(i,j\right)}\V{b}^{\left(i,j\right)}\right\Vert _{F}^{2}.
\end{equation}

We additionally define the fourth order tensor transposes to help expressing the ALS algorithm using {\tt bmp4}.
\begin{definition}
Given $\T{X}\in\mathbb{R}^{m\times p\times n\times q}$, we define
\begin{equation*}
\T{X}^{\top_4}=\perm\left(\T{X},\left[2,3,1,4\right]\right), \quad
\T{X}^{\top^{2}_4}=\perm\left(\T{X},\left[3,1,2,4\right]\right).
\end{equation*}
\end{definition}
Note the above definition of the fourth order tensor transpose is the same for the third-order case holding the last dimension fixed.

Under this definition, we can show the following convenient fact around which we can base a fourth-order BMD-ALS algorithm.
\begin{theorem} Given $\T{X} = \bmpF\left(\T{A},\T{B},\T{C}\right)$, 
\begin{equation*}
    \T{X}^{\top_4} = \bmpF\left(\T{A},\T{B},\T{C}\right)^{\top_4} = \bmpF\left(\T{B}^{\top_4},\T{C}^{\top_4},\T{A}^{\top_4}\right).
\end{equation*}
\end{theorem}
We can then formulate the following ALS algorithm based on the order-4 BMD of a given tensor $\T{X}\in \mathbb{R}^{m\times p\times n\times q}$.
\begin{equation*}
\begin{split}
     \T{B}^{k+1} &= \min_{\T{B}\in\mathbb{R}^{m\times p\times\ell \times q}}\left\Vert \T{X}-\bmpF\left(\T{A}^{k},\T{B},\T{C}^{k}\right)\right\Vert _{F}^{2},\\
    \left(\T{C}^{\top_4}\right)^{k+1} &= \min_{\T{C}^{\top_4} \in\mathbb{R}^{ p\times n \times \ell \times q}}\left\Vert \T{X}^{\top_4}-\bmpF\left(\left(\T{B}^{\top_4}\right)^{k+1},\T{C}^{\top_4},\left(\T{A}^{\top_4}\right)^{k}\right)\right\Vert _{F}^{2},\\
    \left(\T{A}^{\top^{2}_4}\right)^{k+1} &= \min_{\T{A}^{\top^2_4}\in\mathbb{R}^{n\times m\times\ell\times q}}\left\Vert \T{X}^{\top^{2}_4}-\bmpF\left(\left(\T{C}^{\top^{2}_4}\right)^{k+1},\T{A}^{\top^{2}_4},\left(\T{B}^{\top^{2}_4}\right)^{k+1}\right)\right\Vert _{F}^{2}.
\end{split}
\end{equation*}
Similar to the discussion for the third order case, each ALS subproblem can be solved in parallel as smaller sized linear least-squares problems.
\subsection{Application to Color Video with Regularization}
Consider a given video with three color channels $\T{X}\in\mathbb{C}^{m\times p\times n\times3}$. Each frame of the video is of size $m\times n \times 3$, where the third dimension encodes  color. The total number of frames is $p$. Fixing the color channel dimension, our goal is to obtain a tensor triplet $\T{A}\in\mathbb{R}^{m\times\ell\times n\times3}$, $\T{B}\in\mathbb{R}^{m\times n\times\ell\times3}$, and $\T{C}\in\mathbb{R}^{\ell\times p\times n\times3}$ from the fourth order BMD of $\T{X}$ such that $\T{X} \approx \bmpF\left(\T{A},\T{B},\T{C}\right)$.

Since the different tensor slices along the color channel for the factors associated with the temporal dimension (for example, if the frames are oriented as lateral slices, the factor tensors with temporal dimension is $\T{B}$ and $\T{C}$) should be the same, or the difference is small. We consider adding the following constraints to mitigate the non-uniqueness of the solutions present in the order-4 BMD. We note that comparing to the regularized ALS given in Sec.~\ref{sec:regALS}, this constraint is specific for color video decomposition applications but does not promote the background/foreground separation task. 

Consider solving for the factor tensor $\T{B}\in\mathbb{C}^{m\times p\times\ell\times3}$, we add the following regularization
\begin{equation}\label{eq:Bconstraint}
    \left\Vert \T{B}_{i,j,t,1}-\T{B}_{i,j,t,2}\right\Vert_2 <\epsilon \text{ and } \left\Vert \T{B}_{i,j,t,2}-\T{B}_{i,j,t,3}\right\Vert_2 <\epsilon.
\end{equation}
Notice that, from Eq.~(\ref{eq:4Dsub-ls}), holding the indices $(i,j)$ fixed, we obtain the vector $\V{b}^{(i,j)} \in \mathbb{R}^{3\ell \times 1}$ by column major order vectorization of the slice $\T{B}_{i,j,:,:}$, i.e. $\V{b}^{(i,j)} = \vvec(\squeeze(\T{B}_{i,j,:,:}))$, for all $1\leq i\leq m, 1\leq j\leq p$.

Let us define the difference operator $\M{R} \in \mathbb{R}^{2\ell \times 3\ell}$ such that, for all $1\leq t \leq \ell, 1\leq v\leq 2$,
\begin{equation*}
    \M{R}_{(v-1)\ell+t, v\ell+t} = -1 \text{ and } \M{R}_{(v-1)\ell+t, (v-1)\ell+t}=1
\end{equation*}
So the constraint given in Eq.~(\ref{eq:Bconstraint}) is equivalent to applying the difference operator $\M{R}$ on the vectorized slice $\T{B}_{i,j,:,:}$ such that
$\left\Vert\M{R}\V{b}^{(i,j)} \right\Vert_2 <\epsilon$.

Similarly, for the factor tensor $\T{C}^{\top_4} \in\mathbb{R}^{p\times n\times\ell\times 3}$, we have 
\begin{equation*}
    \left\Vert \T{C}^{\top_4}_{j,k,t,1}-\T{C}^{\top_4}_{j,k,t,2}\right\Vert_2 <\epsilon \text{ and } \left\Vert \T{C}^{\top_4}_{j,k,t,2}-\T{C}^{\top_4}_{j,k,t,3}\right\Vert_2 <\epsilon.
\end{equation*}
we can rewrite in an equivalent form with the difference operator $\M{R}$ applied to the vectorized tensor slice $\V{c}^{(j,k)} = \vvec(\squeeze(\T{C}^{\top_4}_{j,k,:,:}))$ as $\left\Vert\M{R}\V{c}^{(j,k)} \right\Vert <\epsilon$, for all $1\leq j\leq p, 1\leq k\leq n$.

Thus, a constrained ALS algorithm for the fourth order BMD of a the color video tensor $\T{X}$ has the following decoupled subproblems
\begin{equation}
\hat{\V{b}}^{(i,j)}=\min_{\V{b}^{\left(i,j\right)}\in\mathbb{R}^{3\ell\times1}}\left\Vert \V{y}_{\T{X}}^{\left(i,j\right)}-\M{H}_{\T{A},\T{C}}^{\left(i,j\right)}\V{b}^{\left(i,j\right)}\right\Vert _{F}^{2} + \left\Vert\M{R}\V{b}^{(i,j)} \right\Vert_{F}^{2},
\end{equation}
and the factor tensor ${\T{B}}$ is approximated by letting $\widehat{\T{B}}_{i,j,:,:} = \reshape\left(\hat{\V{b}}^{(i,j)}, [\ell,3]\right)$. Similarly, we update ${\T{C}}$ by solving
\begin{equation}
\hat{\V{c}}^{(j,k)}=\min_{\V{c}^{\left(j,k\right)}\in\mathbb{R}^{3\ell\times1}}\left\Vert \V{y}_{\T{X}^{\top_4}}^{\left(j,k\right)}-\M{H}_{\T{B},\T{A}}^{\left(j,k\right)}\V{c}^{\left(j,k\right)}\right\Vert _{F}^{2}+\left\Vert\M{R}\V{c}^{(j,k)} \right\Vert_{F}^{2},
\end{equation}
and hence $\widehat{\T{C}}^{\top_4}_{j,k,:,:} = \reshape\left(\hat{\V{c}}^{(j,k)}, [\ell,3]\right)$. Lastly, we update ${\T{A}}$ by solving
\begin{equation}
\hat{\V{a}}^{(k,i)}=\min_{\V{a}^{\left(k,i\right)}\in\mathbb{R}^{3\ell\times1}}\left\Vert \V{y}_{\T{X}^{\top^{2}_4}}^{\left(k,i\right)}-\M{H}_{\T{C},\T{B}}^{\left(k,i\right)}\V{a}^{\left(k,i\right)}\right\Vert _{F}^{2}.
\end{equation}
Then $\hat{\T{A}}^{\top^{2}_4}_{k,i,:,:} = \reshape\left(\hat{\V{a}}^{(k,i)}, [\ell,3]\right)$.

\subsection{Color Video Decomposition Results}
In Fig.~(\ref{fig:color}), we demonstrate the background and foreground separation results using the $\text{BMD-ALS}_{SVD}$ method for the fourth order color videos. The same BM-ranks of 2, 3, 6 and 4 are used for the ``Simulated'', ``Hall and Monitor'', ``Human Body'', and ``IBM Test'' color video sequences as the grayscale videos. Frame 60 of the original videos and the reconstructed background/foreground videos are selected to display. As we can see from the figure, the color background frame has excellent reconstruction quality compared to the ground-truth except small artifacts appeared in the ``Hall and Monitor'' video sequence. In the reconstructed foreground frames, vertical artifacts also appeared around the people moving across the background with small effect on visualizing the people clearly. 

\begin{figure}[!ht]
\centering
\begin{tikzpicture}
\node at (0,0) {\includegraphics[width=0.8\linewidth]{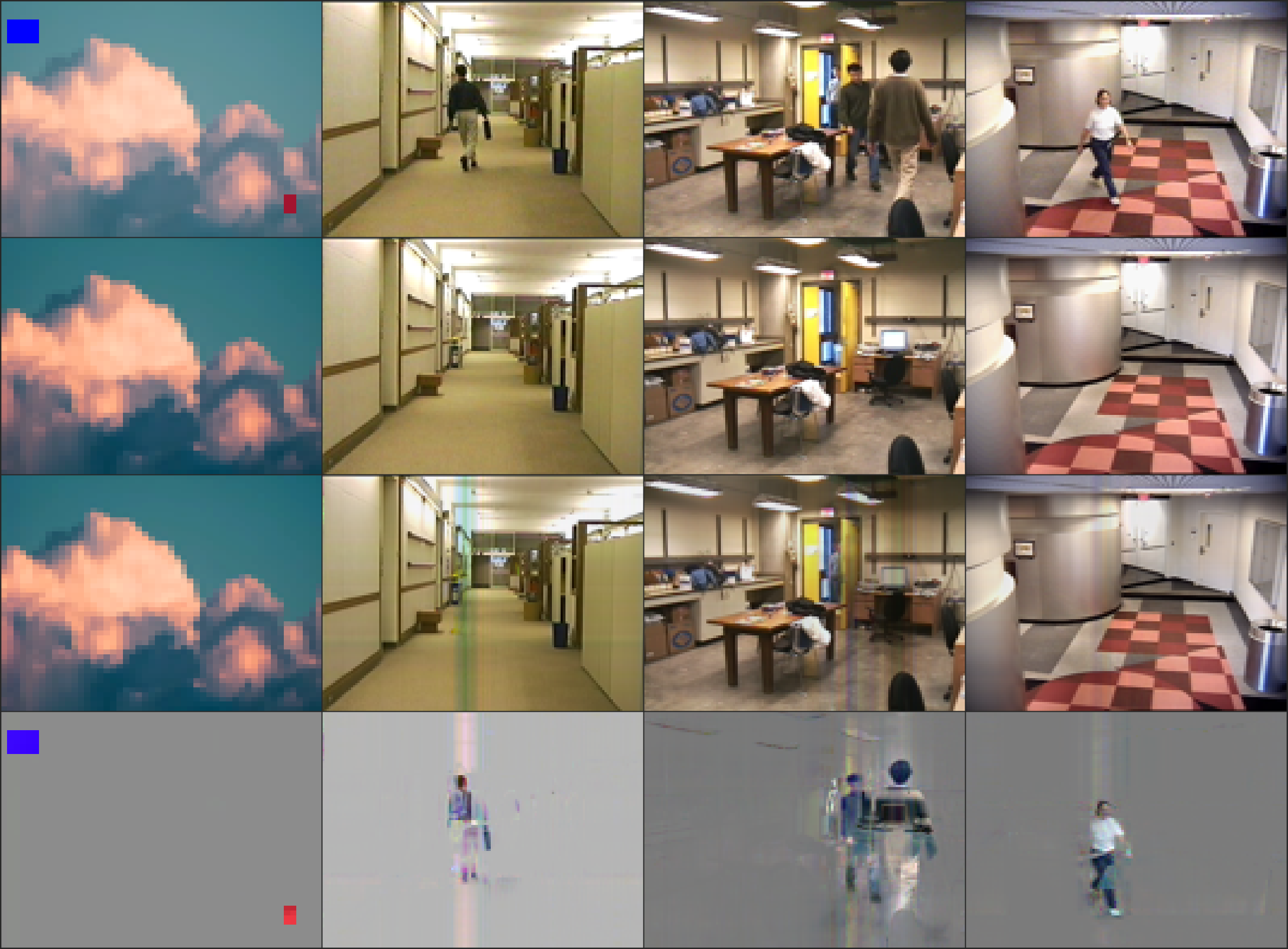}};
\node[rotate=-90] at (6.6, 3.5) {\scriptsize Original};
\node[rotate=-90] at (7.1, 1.1) {\scriptsize Ground-truth};
\node[rotate=-90] at (6.6, 1.1) {\scriptsize Background};
\node[rotate=-90] at (6.6, -1.1) {\scriptsize Background};
\node[rotate=-90] at (6.6, -3.5) {\scriptsize Foreground};
\draw [decorate,decoration={brace,amplitude=5pt,raise=2ex}]
  (6.6,-0.5) -- (6.6,-4) node[midway,xshift=2em,rotate=-90]{\scriptsize Reconstructed};
\node at (-4.7, 5) {\scriptsize Simulated Video};
\node at (-1.5, 5) {\scriptsize Hall and Monitor};
\node at (1.5, 5) {\scriptsize Human Body};
\node at (4.5, 5) {\scriptsize IBM Test};
\end{tikzpicture}
\caption{Color video background/foreground separation by the fourth-order $\text{BMD-ALS}_{SVD}$ method.}
\label{fig:color}
\end{figure}

\section{Conclusions and future work}
We have demonstrated that the third-order tensor BM-decomposition can be effectively used for grayscale video background/foreground separation with compression. Using our generative spatiotemporal video model, we have shown that videos with a stationary background are naturally low BM-rank tensors. As a result, we can model the stationary background with significant image quality comparable to existing methods of tensor SS-SVD and DMD, while obtaining compressible representation of the the foreground video. Without the need of using the background subtraction technique, our BMD estimated foreground video is less affected by the errors resulted from a poor background reconstruction quality. Moreover, we have shown that the alternating least-squares algorithm for computing the BMD is closely connected to the block nonlinear Gauss–Seidel (GS) method, from which the convergence analysis of the algorithm follows naturally. Additionally, we gave new theoretical insight into the BM-rank bounds and non-uniqueness, where the later was overcame by adding a regularization tailored to the video background/foreground separation application. Lastly, we extend the regularized ALS algorithm for color video compression and separation. In the future, we will investigate the use of constraints for a better posed problem, the feasibility of the extension of computation to higher dimensions, and we will investigate the utility of the BMD applied to other types of spatiotemporal data.

\bibliographystyle{plain}
\bibliography{references.bib} 

\newpage
\appendixpageoff
\appendixtitleoff
\begin{appendices}
\section*{Supplementary material}
\input{Supplement}
\end{appendices}

\end{document}

%% file: my_preamble.tex
\usepackage{amsmath,amsfonts,amssymb}
\usepackage{draftwatermark}
\usepackage{stmaryrd}
\usepackage{boxedminipage}
\usepackage{url}
\usepackage{tikz}
\usepackage{multicol}
\usepackage{multirow}
\usepackage{diagbox}
\usepackage{graphicx}
\usepackage{caption}
\usepackage{subcaption}
\usepackage[toc,page,titletoc]{appendix}

\usetikzlibrary{decorations.pathreplacing}
\usetikzlibrary {arrows.meta}
\usetikzlibrary{shapes.multipart}

\usepackage[mathscr]{eucal}

\usepackage{amsbsy}
\usepackage{mathabx}

\usepackage{bm}

\usepackage{paralist}

\usepackage{xspace}

\newcommand{\Sec}[1]{\hyperref[sec:#1]{\S\ref*{sec:#1}}} 
\newcommand{\Eqn}[1]{\hyperref[eq:#1]{(\ref*{eq:#1})}} 
\newcommand{\Fig}[1]{\hyperref[fig:#1]{Figure~\ref*{fig:#1}}} 
\newcommand{\Tab}[1]{\hyperref[tab:#1]{Table~\ref*{tab:#1}}} 
\newcommand{\Thm}[1]{\hyperref[thm:#1]{Theorem~\ref*{thm:#1}}} 
\newcommand{\Proposition}[1]{\hyperref[thm:#1]{Proposition~\ref*{prop:#1}}} 
\newcommand{\Lem}[1]{\hyperref[lem:#1]{Lemma~\ref*{lem:#1}}} 
\newcommand{\Prop}[1]{\hyperref[prop:#1]{Property~\ref*{prop:#1}}} 
\newcommand{\Cor}[1]{\hyperref[cor:#1]{Corollary~\ref*{cor:#1}}} 
\newcommand{\Def}[1]{\hyperref[def:#1]{Definition~\ref*{def:#1}}} 
\newcommand{\Ex}[1]{\hyperref[ex:#1]{Example~\ref*{ex:#1}}} 

\newtheorem{remark}{Remark}

\newcommand{\vvec}{ {\mbox{\tt vec}}}
\newcommand{\reshape}{{\mbox{\tt reshape}}}
\newcommand{\squeeze}{{\mbox{\tt squeeze}}}
\newcommand{\diag}{{\mbox{\tt diag}}}
\newcommand{\bmp}{{\mbox{\tt bmp}}}
\newcommand{\bmpF}{{\mbox{\tt bmp4}}}
\newcommand{\perm}{{\mbox{\tt permute}}}

\newcommand{\mat}{{\mbox{\tt Mat}}}
\newcommand{\tvec}{{\mbox{\tt Tvec}}}
\newcommand{\tfold}{{\mbox{\tt Tfold}}}
\newcommand{\ones}{{\mbox{\tt ones}}}

\newcommand{\V}[1]{{\bm{\mathbf{\MakeLowercase{#1}}}}} 

\newcommand{\M}[1]{{\bm{\mathbf{\MakeUppercase{#1}}}}} 

\newcommand{\T}[1]{\boldsymbol{\mathscr{\MakeUppercase{#1}}}} 

\newcommand{\TA}{\T{A}}
\newcommand{\TB}{\T{B}}
\newcommand{\TC}{\T{C}}
\newcommand{\TX}{\T{X}}

\newcommand{\Vx}{\V{x}}
\newcommand{\Vu}{\V{u}}
\newcommand{\Va}{\V{a}}
\newcommand{\Vb}{\V{b}}
\newcommand{\Vc}{\V{c}}

\newcommand{\Vv}{\V{v}}

\newcommand{\MA}{\M{A}}
\newcommand{\MC}{\M{C}}
\newcommand{\MX}{\M{X}}

\newcommand{\MD}{\M{D}}
\newcommand{\MB}{\M{B}}

\newcommand{\MW}{\M{W}}

 \newcommand{\bea}{ \left[ \begin{matrix} }
 \newcommand{\eea}{ \end{matrix} \right] }

\SetWatermarkText{ }
\SetWatermarkScale{3}

\definecolor{blue}{rgb}{0,0,1}
\definecolor{red}{rgb}{1,0,0}
\definecolor{green}{rgb}{.5,.8,.5}
\definecolor{mayablue}{rgb}{0.45, 0.76, 0.98}
\definecolor{ceil}{rgb}{0.57, 0.63, 0.81}
\definecolor{tblue}{rgb}{0.19, 0.45, 0.7}
\definecolor{brightlavender}{rgb}{0.75, 0.58, 0.89}
\definecolor{asparagus}{rgb}{0.53, 0.66, 0.42}
\definecolor{lightseagreen}{rgb}{0.13, 0.7, 0.67}
\definecolor{rosevale}{rgb}{0.67, 0.31, 0.32}
\definecolor{turquoisegreen}{rgb}{0.63, 0.84, 0.71}

\definecolor{mldgreen}{RGB}{189, 209, 187}
\definecolor{mldpink}{RGB}{253, 171, 191}
\definecolor{mlddarkgreen}{RGB}{154, 183, 185}
\definecolor{carminepink}{rgb}{0.92, 0.3, 0.26}

%% file: tikz_tensor_object.tex
\usepackage{tikz} 		
\usepackage{pgfkeys} 	
\usepackage{ifthen} 		
\usetikzlibrary{matrix}

\pgfkeys{
/tensor/.cd, 
dim1/.initial = 1,		dim1/.get = \dimOne,		dim1/.store in = \dimOne,
dim2/.initial = 1,		dim2/.get = \dimTwo,		dim2/.store in = \dimTwo,
dim3/.initial = 1,		dim3/.get = \dimThree,	dim3/.store in = \dimThree,
xshift/.initial = 0, 	xshift/.get = \xShift, 		xshift/.store in = \xShift,
yshift/.initial = 0, 	yshift/.get = \yShift, 		yshift/.store in = \yShift,
xspec/.initial = 0, 	xspec/.get = \xSpec, 		xspec/.store in = \xSpec,
yspec/.initial = 0, 	yspec/.get = \ySpec, 		yspec/.store in = \ySpec,
scale/.initial = 1, 	scale/.get = \myScale, 	scale/.store in = \myScale,
fill/.initial = white, 	fill/.get = \myFill, 		fill/.store in = \myFill,
back edges/.initial = 0, 	back edges/.get = \myBack, 	back edges/.store in = \myBack,
slice type/.initial = none, 		slice type/.get = \sliceType, 			slice type/.store in = \sliceType, 
number of slices/.initial = 1, 	number of slices/.get = \nSlices, 		number of slices/.store in = \nSlices,
slice width/.initial = 1, 		slice width/.get = \sWidth, 				slice width/.store in = \sWidth, 
}

\newcommand{\tensor}[1][]{\@tensor[#1]}
\def\@tensor[#1] (#2,#3) #4; {{ 

\pgfkeys{/tensor/.cd,#1}

\def\depthScale{0.5} 

\pgfmathsetmacro{\numSlicesMinusOne}{\nSlices-1}
\pgfmathsetmacro{\numSlicesPlusOne}{\nSlices+1}

\pgfmathsetmacro{\sliceLength}{\myScale*\dimOne}

\ifthenelse{\equal{\sliceType}{lateral}}
	{
	
	\pgfmathsetmacro{\sliceWidth}{\myScale*\sWidth*0.9*\dimTwo/\nSlices}
	\pgfmathsetmacro{\sliceGap}{\myScale*\dimTwo/(\nSlices-1) - \nSlices*\sliceWidth/(\nSlices-1)}
	\pgfmathsetmacro{\sliceDepth}{\myScale*\dimThree}
	
	} 
	{
	\ifthenelse{\equal{\sliceType}{frontal}}
		{
		
		\pgfmathsetmacro{\sliceDepth}{\myScale*\sWidth*0.9*\dimThree/\nSlices}
		\pgfmathsetmacro{\sliceGap}{\myScale*\dimThree/(\nSlices-1) - \nSlices*\sliceDepth/(\nSlices-1)}
		\pgfmathsetmacro{\sliceWidth}{\myScale*\dimTwo}
	
		}
		{
		\pgfmathsetmacro{\sliceWidth}{\myScale*\dimTwo}
		\pgfmathsetmacro{\sliceDepth}{\myScale*\dimThree}
		}

	}

\def\xFront{#2 + \xShift}	
\def\yFront{#3 + \yShift}
\def\xBack{#2 + \xShift + \depthScale*\sliceDepth + \xSpec*\sliceDepth}
\def\yBack{#3 + \yShift + \depthScale*\sliceDepth + \ySpec*\sliceDepth}

\def\aFront{(\xFront, \yFront)}
\def\bFront{(\xFront, \yFront + \sliceLength)}
\def\cFront{(\xFront + \sliceWidth, \yFront + \sliceLength)}
\def\dFront{(\xFront + \sliceWidth, \yFront)}

\def\aBack{(\xBack, \yBack)}
\def\bBack{(\xBack, \yBack + \sliceLength)}
\def\cBack{(\xBack + \sliceWidth, \yBack + \sliceLength)}
\def\dBack{(\xBack+ \sliceWidth, \yBack)}

\ifthenelse{\NOT\equal{\myFill}{nofill}}
	{
	\def\tempTensor{
		\fill[\myFill!25] \bFront -- \bBack -- \cBack -- \cFront -- cycle; 
		\fill[\myFill!75] \dFront -- \dBack -- \cBack -- \cFront -- cycle; 
		\fill[\myFill!50] \aFront rectangle \cFront;  
	
		\draw \aFront rectangle \cFront; 
		\draw \bFront -- \bBack; 
		\draw \cFront -- \cBack;
		\draw \dFront -- \dBack;
	
		\draw \bBack -- \cBack;
		\draw \cBack -- \dBack;
		}
	}
	{ 

	\def\tempTensor{
		\draw \aFront rectangle \cFront; 
		
		\ifthenelse{\NOT\equal{\myBack}{0}}
		{
			\draw[dashed] \bBack -- \aBack -- \dBack;
		}{}
		
		\draw \dBack -- \cBack -- \bBack;

		\ifthenelse{\NOT\equal{\myBack}{0}}
		{
			\draw[dashed] \aFront -- \aBack;
		}{}
		
		\draw \bFront -- \bBack;
		\draw \cFront -- \cBack;
		\draw \dFront -- \dBack;
		}
	}

\ifthenelse{\equal{\sliceType}{lateral}}
	{
	\foreach\sliceCount in {0,...,\numSlicesMinusOne}
		{	
		\begin{scope}[shift ={(\sliceCount*\sliceWidth + \sliceCount*\sliceGap, 0)}]
			\tempTensor;
		\end{scope}
		}
	
	}
	{
	
	\ifthenelse{\equal{\sliceType}{frontal}}
	{
	
	\pgfmathsetmacro{\xStep}{\sliceDepth/2 + \sliceGap/2 + \myScale*\dimThree*\xSpec/(\nSlices-(1-\sWidth))}
	\pgfmathsetmacro{\yStep}{\sliceDepth/2 + \sliceGap/2 +  \myScale*\dimThree*\ySpec/(\nSlices-(1-\sWidth))}
	
	\foreach\sliceCount in {-\numSlicesMinusOne,...,0}
		{	
		
		\begin{scope}[shift = {(-\sliceCount*\xStep, -\sliceCount*\yStep)}]
			\tempTensor;
		\end{scope}
	
		}
	
	}
	{
	\tempTensor;
	}
	
	}

\node at (#2 + \dimTwo/2, #3 + \dimOne/2) {#4};

}} 


%% file: video_surveillance.tex
The task of video background and foreground separation is an important computer vision application \cite{candes2011robust, li2004statistical, tian2005robust,garcia2020background}. Background subtraction is often used in this task in order to detect moving foreground objects. Hence, accurate modeling of the video background under complex, diverse, and cluttered conditions is of paramount importance for most of the background/foreground separation methods. One of the more promising research directions in the field focuses on utilizing matrix decomposition methods. By vectorizing video frames into vectors and stacking them column-wise into a matrix, decomposition methods such as the robust principal component analysis (RPCA) \cite{candes2011robust}, and the dynamic mode decomposition (DMD) \cite{kutz2016dynamic} method can be used to separate the video data matrix into a stationary background and moving foreground components which are usually assumed to be low-rank and sparse respectively. 

However, flattening three-dimensional video data into a matrix presents disadvantages. For instance, vectorization destroys the intrinsic spatial structure within frames. Moreover, complex disturbances within the background can be ignored after flattening, which would potentially lead to a lower background reconstruction quality \cite{li2022tensor}. To overcome these difficulties for matrix-based methods, an increasing number of tensor decomposition methods have been proposed dealing with background subtraction  \cite{li2022tensor, karim2020accurate, kajo2018svd, sobral2015online}. Similar to matrix-based methods, the tensor-based techniques model the background video with a low tensor rank component and obtain the sparse foreground by subtracting the background from the video data. Although for most of the aforementioned methods the low-rank property of the background suggests that it is also possible to compress the stationary background for more efficient data storage, the step of background subtraction does not guarantee a compressed representation of the foreground video.

\subsection{Dynamic Mode Decomposition (DMD)}\label{subsec:dmd}
Though initially introduced in the fluid dynamics community for extracting spatiotemporal patterns  \cite{kutz2016dynamic}, DMD has also been effective in the video context for separating the stationary background from foreground motions by differentiating between the DMD modes with near-zero frequency and remaining modes with frequencies bounded away from the origin \cite{grosek2014dynamic}. Better visual results and superior computational efficiency compared to the Robust-PCA algorithm of the DMD method have drawn great attention in the computer vision community with new algorithms developed based on DMD for improved accuracy and efficiency. A few examples include the compressed DMD \cite{erichson2019compressed}, randomized DMD \cite{erichson2016randomized}, DMD via dictionary learning\cite{haq2020dynamic}, and multi-resolution DMD for object tracking with varying motion rates \cite{kutz2015multi}. 

More recently, the connection between the general DMD method and the CP decomposition has been studied \cite{redman2021koopman}. When multiple experiments were conducted, data matrices generated from the experiments were collected and ordered as frontal slices of a third-order tensor. Then performing the DMD on each frontal slice of the data tensor decomposes it into a sum of vector outer products of the DMD vector triplets consisting of the DMD modes, DMD eigenfunctions, and the corresponding eigenvalues. It is easily shown that this decomposition is equivalent to taking a CP decomposition of the third-order tensor when the experimental data are collected from distinct sources with a single exponential growth or decay and/or oscillatory dynamics. 

More specifically, for the video background/foreground separation application \cite{kutz2016dynamic}, the DMD stationary background video sequence is reconstructed by taking 
\begin{equation} \label{eq:dmd_bg}
    \M{X}_{\text{DMD}}^{\text{Low-Rank}} = b_p\V{\varphi}_{p}e^{\omega_{p}\V{t}} \in \mathbb{R}^{mn\times 1},
\end{equation}
where $\omega_{p}$, $p \in \{1,2,...,\ell\}$, is the Fourier mode satisfying $\Vert \omega_{p} \Vert \approx 0$. The vector $\V{\varphi}_{p}$ is the associated $p$-th DMD mode, and $b_p$ is the initial amplitude of the corresponding mode. The vector $\V{t}=[t_1, t_2,\dots,t_{p}]$ contains the times at which the frames were collected.

In the present work, we can view the DMD results obtained from decomposing a single data matrix as a tensor BM decomposition by Theorem \ref{thm:matrix2bmd}. This is particularly meaningful to our video application since the DMD method for video background/foreground separation developed originally is applied to a single data matrix flattened from a third-order video tensor \cite{grosek2014dynamic}. Detailed derivations on rewriting the DMD results into order-3 tensor BMP is provided in the Supplement, where we will also show that due to the nonlinearity of the foreground motion in general, the BMD factor tensor triplet does not directly convert back to DMD modes and eigenvalues. 

\subsection{Spatiotemporal Slice-based SVD (SS-SVD)}
\label{subsec:ss-svd}
In the recent study by Kajo et al. \cite{kajo2018svd}, the slice-wise SVD discussed previously was applied to the spatiotemporal slices of an input video tensor to extract the background information. The authors refer to this method as spatiotemporal slice-based SVD (SS-SVD). Since it has a connection to our initialization step in the BMD-ALS algorithm, we will briefly describe their method here.

Given a video of $p$ frames with size $m \times n$, we order the frames as lateral slices to form a third-order tensor $\TX$ of size $m \times p\times n$. Each frontal slice, $\TX_{:,:,k}, 1\leq k \leq n$, or horizontal slice $\TX_{i, :, :}, 1\leq i \leq m$ is called a spatiotemporal slice containing both space and time information. Here we construct the SS-SVD using frontal spatiotemporal slices. Specifically, the SS-SVD method applies a low-rank approximation to each frontal spatiotemporal slice using truncated SVDs with a target matrix rank $\ell$:
\begin{equation} \label{eq:sssvd}
\TX_{:,:,k} \approx \sum_{t=1}^{\ell} \Vu^{(k)}_t \sigma_t^{(k)} \left(\V{v}_t^{(k)}\right)^{\top}, 1\leq k \leq n.
\end{equation}

When $t=1$, the rank-1 matrix reconstruction corresponds to the largest singular value of each slice which the authors of \cite{kajo2018svd} argue captures mainly the dominant background scene across slices\footnote{We note that if the data is non-negative, the first rank-1 triples will be non-negative by the Perrone-Frobenius theorem.}. We denote this set of tensor slices as $\hat{\TX}^{\text{bg}}_{:,:,k}$ for all $ k=1,\dots,n$ and is given by 
\begin{equation}
    \hat{\TX}^{\text{bg}}_{:,:,k} =  \Vu^{(k)}_{1} \sigma_{1}^{(k)} \left(\V{v}_{1}^{(k)}\right)^{\top}. 
    \label{eq:ssvd_bg}
\end{equation}
The foreground in \cite{kajo2018svd} is  taken to be the difference between the given video data and the reconstructed background, i.e. $\hat{\TX}^{\text{fg}} = \TX - \hat{\TX}^{\text{bg}}$, {\it which is not guaranteed to be a compressed representation}.  

However, the approximation (\ref{eq:sssvd}), which we will call the {\bf Slicewise SVD}, by our results in Section.~\ref{subsec:slice-wise}, {\it can be interpreted as a BMD of BM-rank at most $\ell$}. 
\begin{theorem} \label{th:slicewise}
The tensor $\hat{\TX}$ defined with 
\[ \hat{\TX}_{:,:,k} = \MW^{(k)} (\M{V}^{(k)})^T,\]
where $\V{u}_t^{(k)} \sigma_t^{(k)}$ from (\ref{eq:sssvd}) are the $\ell$ columns of $\MW^{(k)}$ and $\V{v}_t^{(k)}$ are the $\ell$ columns of $\M{V}^{(k)}$ has BM-rank at most $\ell$ and the error in the approximation is
\[ \| \TX - \hat{\TX} \|_F^2 = \sum_{k=1}^{n} \sum_{i=\ell+1}^{r} (\sigma_{i}^{(k)})^2. \]
\end{theorem}
In the next section, using our generative model for insight, we will explain the utility of this result.

%% file: generative_model.tex
In this section, we consider a BM product-based model of one or two moving objects on a static background.  
Since in our model the stationary background image and the moving foreground object can be incorporated into separate BM-rank 1 components, a low BM-rank tensor decomposition separates the video data into compressed representations of both the background and the foreground simultaneously. 

\subsection{Generative Spatiotemporal Model With Low BM-rank}
Suppose we have a static background image $\M{X} \in \mathbb{R}^{m \times n}$, and an object (of constant intensity $\alpha$ and constant rectangular size $r_1 \times r_2$ for the time being) moving across the background over $p$ time steps. At time $k$, $1\leq k \leq p$, the object is located at $\mathcal{I}_k \times \mathcal{J}_k: = [i_k, i_k+r_1] \times [j_k, j_k+r_2]$ with $1\leq i_k \leq m-r_1$ and $1\leq j_k \leq n-r_2$ . Consider a binary image of the same size as the background image denoted $\M{M}_{\mathcal{I}_k, \mathcal{J}_k}$, where
\begin{equation}
    \M{M}_{\mathcal{I}_k,\mathcal{J}_k} [i,j] = \begin{cases}
\alpha, \text{ if } (i,j) \in \mathcal{I}_k \times \mathcal{J}_k\\
0, \text{ otherwise }\\
\end{cases}.
\end{equation}
Next, define a binary vector $\Vb^{(k)} \in \mathbb{R}^{m}$ such that $\Vb^{(k)}_i = 1$ when $i\in \mathcal{I}_k$ and is $0$ otherwise. Similarly, define $\Vc^{(k)} \in \mathbb{R}^{n}$ such that $\Vc^{(k)}_j = 1$ when $j\in \mathcal{J}_k$ and is $0$ otherwise. Finally,  let $\M{1}_{m\times n} \in \mathbb{R}^{m \times n}$ be a matrix of all-ones. Then the rectangle image at time $k$ can be expressed as
\begin{equation}
\M{M}_{\mathcal{I}_k,\mathcal{J}_k} = \alpha \, \diag (\Vb^{(k)} ) \cdot \M{1}_{m\times n}  \cdot \diag (\Vc^{(k)} ), 
\end{equation}
where $\diag(\V{v})$ means the square, diagonal matrix with the diagonal entries provided by the vector argument $\V{v}$.
Hence, the $k$-th video frame that captures both the background image and the moving object, denoted as $\M{T}^{(k)}$, can be expressed as
\begin{equation}
    \begin{split}
        \M{T}^{(k)} & := \M{X} - \diag (\Vb^{(k)} ) \cdot \M{X}  \cdot \diag (\Vc^{(k)} ) + \M{M}_{\mathcal{I}_k,\mathcal{J}_k} \\
        & = \diag(\M{1}_{m}) \M{X} \diag(\M{1}_{n}) +  \diag (\Vb^{(k)} ) \left(-\M{X} + \alpha \M{1}_{m\times n}  \right) \diag (\Vc^{(k)} ).
    \end{split}
    \label{eq:img}
\end{equation}
The term $-\diag (\Vb^{(k)} ) \M{X} \diag (\Vc^{(k)} )$ ``zeros out" the entries in the stationary image where the object is living in this frame, and the term $\M{M}_{\mathcal{I}_k,\mathcal{J}_k}$ puts the constant value rectangle over those pixels.  Both are necessary to ensure that the object retains its constant value. 

Now we show this simple video can be modeled using a low-rank BMD. Specifically, 
\begin{enumerate}
    \item Define the $m \times 2 \times n$ tensor $\TA$ with $\TA_{:,1,:}=\MX$ and $\TA_{:,2,:}= \left(-\MX + \alpha \M{1}_{m\times n}\right)$,
    \item Define the $m \times p \times 2$ tensor $\TB$ with $ \TB_{:,:,1} = \M{1}_{m\times p}$ and $\TB_{:,k,2}= \Vb^{(k)}$,
    \item Define the $2 \times p \times n$ tensor $\TC$ with $\TC_{1,:,:} = \M{1}_{p\times n}$ and $\TC_{2,k,:} = \Vc^{(k)}$. 
\end{enumerate}
Then the video tensor $\TX=\bmp(\TA,\TB,\TC)$ is a sum of two BM-rank 1 tensors: the stationary background tensor and the foreground moving object tensor, i.e. 
\begin{equation}  \label{eq:fgbg}
    \TX = \TX^{\text{bg}}+\TX^{\text{fg}} := \bmp(\TA_{:,1,:},\TB_{:,:,1},\TC_{1,:,:}) + \bmp(\TA_{:,2,:},\TB_{:,:,2},\TC_{2,:,:}).
\end{equation}
The $k$-th lateral slice of tensor $\TX$ is given exactly by Eq. (\ref{eq:img}), i.e. $\squeeze(\TX_{:,k,:}) = \M{T}^{(k)}$.  

We can augment this generative spatiotemporal video model in straightforward ways. For example, if there is a second object that is moving either in the same horizontal or same vertical direction as the first, 
we can augment the columns of $\TB_{:,2,:}$ (vertical) or $\TC_{2,:,:}$ (horizontal) to capture the other object's motion without increasing the BM rank (see Figure \ref{fig:genModel_rk2}). If the two objects are moving with different trajectories, and possibly with different intensities, we can model the video with at most three BM-rank 1 terms (see Figure \ref{fig:genModel_rk3}).

The point is that our generative model illustrates why we might reasonably expect to be able to capture the foreground and the background with the compressive power of a low BM-rank approximation.  Moreover, in one of the numerical experiments, we will use this generative model to illustrate superiority of the our low BM rank model vs. a state-of-the-art DMD-based approach.    

\begin{figure}[ht]
\centering
\begin{tikzpicture}
\foreach \j in {0, 0.65, 0.65*2, 0.65*3, 0.65*4}{
\tensor[dim1 = 1.5, dim2 = 0.15, dim3 = 1, fill = ceil] (-0.5+\j, 0) {};
}
\tensor[dim1 = 1.5, dim2 = 0.45, dim3 = 1, fill = ceil] (-0.5+0.65*5, 0) {};
\tensor[dim1 = 0.15, dim2 = 0, dim3 = 0.12, fill = green] (0, 0.5) {};
\tensor[dim1 = 0.15, dim2 = 0, dim3 = 0.12, fill = green] (0.65-0.06, 0.5-0.075) {}; 
\tensor[dim1 = 0.15, dim2 = 0, dim3 = 0.12, fill = green] (0.65*2-0.06, 0.5+0.075) {};
\tensor[dim1 = 0.15, dim2 = 0, dim3 = 0.12, fill = green] (0.65*3-0.06*2, 0.5+0.075*2) {};
\tensor[dim1 = 0.15, dim2 = 0, dim3 = 0.12, fill = green] (0.65*4-0.06*3, 0.5+0.075*2) {};
\tensor[dim1 = 0.15, dim2 = 0, dim3 = 0.12, fill = green] (0, 1.15) {};
\tensor[dim1 = 0.15, dim2 = 0, dim3 = 0.12, fill = green] (0.65-0.06, 1.15-0.075*2) {};
\tensor[dim1 = 0.15, dim2 = 0, dim3 = 0.12, fill = green] (0.65*2-0.06, 1.15-0.075*4) {};
\tensor[dim1 = 0.15, dim2 = 0, dim3 = 0.12, fill = green] (0.65*3-0.06*2, 1.15-0.075*3) {};
\tensor[dim1 = 0.15, dim2 = 0, dim3 = 0.12, fill = green] (0.65*4-0.06*3, 1.15-0.075*2) {};
\draw [arrows = {-Stealth[scale width=1]}] (0, 2.25) -- (3.2, 2.25);
\node at (1.6, 2.5) {\small Time};
\node[rotate=90] at (-0.85, 0.75) {\small Height};
\node[rotate=48] at (-0.65, 1.85) {\small Width};
\node at (4.15, 1) {$=$};
\node at (1.1, -0.4) {\small $\M{T}^{(k)}$};
\tensor[dim1 = 1.5, dim2 = 0.15, dim3 = 1, fill = ceil] (4.5, 0) {};
\tensor[dim1 = 1.5, dim2 = 1.2, dim3 = 0.12, fill = gray] (4.5+0.9, -0.15) {};
\tensor[dim1 = 0.15, dim2 = 1.2, dim3 = 1, fill = gray] (4.5+0.95, 1.6) {};
\node at (5.5, -1) {\small Background};
\node at (7.2, 1) {$+$};
\node at (4.55, -0.35) {$\T{A}_{:,1,:}$};
\node at (6, -0.5) {$\T{B}_{:,:,1}$};
\node at (6.5, 2.5) {$\T{C}_{1,:,:}$};
\tensor[dim1 = 1.5, dim2 = 0.12, dim3 = 1, fill = mlddarkgreen] (8-0.15, 0) {};
\tensor[dim1 = 1.5, dim2 = 1.2, dim3 = 0.12, fill = nofill] (8+0.75, -0.15) {};
\foreach \j in {0.15, 0.15*2, 0.15*3, 0.15*4, 0.15*5}{
\tensor[dim1 = 1.5, dim2 = 0, dim3 = 0, fill = nofill] (8+0.75+\j, -0.15) {};
\tensor[dim1 = 0, dim2 = 0, dim3 = 0.12, fill = nofill] (8+0.75+\j, 1.35) {};
}
\tensor[dim1 = 0.15, dim2 = 0.15, dim3 = 0, fill = black] (8.75, 0) {};
\tensor[dim1 = 0.15, dim2 = 0.15, dim3 = 0, fill = black] (8.75+0.15, 0) {};
\tensor[dim1 = 0.15, dim2 = 0.15, dim3 = 0, fill = black] (8.75+0.15*2, 0.15) {};
\tensor[dim1 = 0.15, dim2 = 0.15, dim3 = 0, fill = black] (8.75+0.15*3, 0.15*2) {};
\tensor[dim1 = 0.15, dim2 = 0.15, dim3 = 0, fill = black] (8.75+0.15*4, 0.15*2) {};
\tensor[dim1 = 0.15, dim2 = 0.15, dim3 = 0, fill = black] (8.75, 0.88) {};
\tensor[dim1 = 0.15, dim2 = 0.15, dim3 = 0, fill = black] (8.75+0.15, 0.88-0.15) {};
\tensor[dim1 = 0.15, dim2 = 0.15, dim3 = 0, fill = black] (8.75+0.15*2, 0.88-0.15) {};
\tensor[dim1 = 0.15, dim2 = 0.15, dim3 = 0, fill = black] (8.75+0.15*3, 0.88) {};
\tensor[dim1 = 0.15, dim2 = 0.15, dim3 = 0, fill = black] (8.75+0.15*4, 0.88+0.15) {};
\tensor[dim1 = 0.15, dim2 = 1.2, dim3 = 1, fill = nofill] (8+0.75, 1.6) {};
\foreach \j in {0.15, 0.15*2, 0.15*3, 0.15*4, 0.15*5}{
\tensor[dim1 = 0, dim2 = 0, dim3 = 1, fill = nofill] (8+0.75+\j, 1.75) {};
\tensor[dim1 = 0.15, dim2 = 0, dim3 = 0, fill = nofill] (8+0.75+\j, 1.6) {};
}
\tensor[dim1 = 0, dim2 = 0.15, dim3 = 0.12, fill = black] (9.1, 2.1) {};
\tensor[dim1 = 0, dim2 = 0.15, dim3 = 0.12, fill = black] (9.1+0.09, 2.1-0.06) {};
\tensor[dim1 = 0, dim2 = 0.15, dim3 = 0.12, fill = black] (9.1+0.15+0.09, 2.1-0.06) {};
\tensor[dim1 = 0, dim2 = 0.15, dim3 = 0.12, fill = black] (9.1+0.15+0.09*2, 2.1-0.12) {};
\tensor[dim1 = 0, dim2 = 0.15, dim3 = 0.12, fill = black] (9.1+0.15+0.09*3, 2.1-0.06*3) {};
\node at (9, -1) {\small Foreground};
\node at (4.55+3.35, -0.35) {$\T{A}_{:,2,:}$};
\node at (6+3.35, -0.5) {$\T{B}_{:,:,2}$};
\node at (6.5+3.35, 2.5) {$\T{C}_{2,:,:}$};
\end{tikzpicture}
\caption{Illustration of the generative spatiotemporal video model with two constant valued, same intensity, rectangular objects moving in the same horizontal direction, on a constant background.}
\label{fig:genModel_rk2}
\end{figure}

\begin{figure}[ht]
\centering
\begin{tikzpicture} 
\foreach \j in {0, 0.65, 0.65*2, 0.65*3, 0.65*4}{
\tensor[dim1 = 1.5, dim2 = 0.15, dim3 = 1, fill = ceil] (-0.5+\j, 0) {};
}
\tensor[dim1 = 1.5, dim2 = 0.45, dim3 = 1, fill = ceil] (-0.5+0.65*5, 0) {};
\tensor[dim1 = 0.15, dim2 = 0, dim3 = 0.12, fill = red] (0, 0.5) {};
\tensor[dim1 = 0.15, dim2 = 0, dim3 = 0.12, fill = red] (0.65-0.06, 0.5-0.075) {}; 
\tensor[dim1 = 0.15, dim2 = 0, dim3 = 0.12, fill = red] (0.65*2-0.06, 0.5+0.075) {};
\tensor[dim1 = 0.15, dim2 = 0, dim3 = 0.12, fill = red] (0.65*3-0.06*2, 0.5+0.075*2) {};
\tensor[dim1 = 0.15, dim2 = 0, dim3 = 0.12, fill = red] (0.65*4-0.06*3, 0.5+0.075*2) {};
\tensor[dim1 = 0.15, dim2 = 0, dim3 = 0.12, fill = green] (-0.17, 1.15) {};
\tensor[dim1 = 0.15, dim2 = 0, dim3 = 0.12, fill = green] (0.65-0.17, 1.15-0.075*2) {};
\tensor[dim1 = 0.15, dim2 = 0, dim3 = 0.12, fill = green] (0.65*2-0.17, 1.15-0.075*4) {};
\tensor[dim1 = 0.15, dim2 = 0, dim3 = 0.12, fill = green] (0.65*3-0.17-0.06, 1.15-0.075*3) {};
\tensor[dim1 = 0.15, dim2 = 0, dim3 = 0.12, fill = green] (0.65*4-0.17-0.06*2, 1.15-0.075*2) {};
\draw [arrows = {-Stealth[scale width=1]}] (0, 2.25) -- (3.2, 2.25);
\node at (1.6, 2.5) {\small Time};
\node[rotate=90] at (-0.85, 0.75) {\small Height};
\node[rotate=48] at (-0.65, 1.85) {\small Width};
\node at (4.15, 1) {$=$};
\node at (1.1, -0.4) {\small $\M{T}^{(k)}$};
\tensor[dim1 = 1.5, dim2 = 0.15, dim3 = 1, fill = ceil] (4.5, 0) {};
\tensor[dim1 = 1.5, dim2 = 1.2, dim3 = 0.12, fill = gray] (4.5+0.9, -0.15) {};
\tensor[dim1 = 0.15, dim2 = 1.2, dim3 = 1, fill = gray] (4.5+0.95, 1.6) {};
\node at (5.5, -1) {\small Background};
\node at (7.2, 1) {$+$};
\node at (4.55, -0.35) {$\T{A}_{:,1,:}$};
\node at (6, -0.5) {$\T{B}_{:,:,1}$};
\node at (6.5, 2.5) {$\T{C}_{1,:,:}$};
\tensor[dim1 = 1.5, dim2 = 0.12, dim3 = 1, fill = rosevale] (8-0.15, 0) {};
\tensor[dim1 = 1.5, dim2 = 1.2, dim3 = 0.12, fill = nofill] (8+0.75, -0.15) {};
\foreach \j in {0.15, 0.15*2, 0.15*3, 0.15*4, 0.15*5}{
\tensor[dim1 = 1.5, dim2 = 0, dim3 = 0, fill = nofill] (8+0.75+\j, -0.15) {};
\tensor[dim1 = 0, dim2 = 0, dim3 = 0.12, fill = nofill] (8+0.75+\j, 1.35) {};
}
\tensor[dim1 = 0.15, dim2 = 0.15, dim3 = 0, fill = black] (8.75, 0) {};
\tensor[dim1 = 0.15, dim2 = 0.15, dim3 = 0, fill = black] (8.75+0.15, 0) {};
\tensor[dim1 = 0.15, dim2 = 0.15, dim3 = 0, fill = black] (8.75+0.15*2, 0.15) {};
\tensor[dim1 = 0.15, dim2 = 0.15, dim3 = 0, fill = black] (8.75+0.15*3, 0.15*2) {};
\tensor[dim1 = 0.15, dim2 = 0.15, dim3 = 0, fill = black] (8.75+0.15*4, 0.15*2) {};
\tensor[dim1 = 0.15, dim2 = 1.2, dim3 = 1, fill = nofill] (8+0.75, 1.6) {};
\foreach \j in {0.15, 0.15*2, 0.15*3, 0.15*4, 0.15*5}{
\tensor[dim1 = 0, dim2 = 0, dim3 = 1, fill = nofill] (8+0.75+\j, 1.75) {};
\tensor[dim1 = 0.15, dim2 = 0, dim3 = 0, fill = nofill] (8+0.75+\j, 1.6) {};
}
\tensor[dim1 = 0, dim2 = 0.15, dim3 = 0.12, fill = black] (9.1, 2.1) {};
\tensor[dim1 = 0, dim2 = 0.15, dim3 = 0.12, fill = black] (9.1+0.09, 2.1-0.06) {};
\tensor[dim1 = 0, dim2 = 0.15, dim3 = 0.12, fill = black] (9.1+0.15+0.09, 2.1-0.06) {};
\tensor[dim1 = 0, dim2 = 0.15, dim3 = 0.12, fill = black] (9.1+0.15+0.09*2, 2.1-0.12) {};
\tensor[dim1 = 0, dim2 = 0.15, dim3 = 0.12, fill = black] (9.1+0.15+0.09*3, 2.1-0.06*3) {};
\node at (10.5, -1) {\small Foreground};
\node at (4.55+3.35, -0.35) {$\T{A}_{:,2,:}$};
\node at (6+3.35, -0.5) {$\T{B}_{:,:,2}$};
\node at (6.5+3.35, 2.5) {$\T{C}_{2,:,:}$};

\tensor[dim1 = 1.5, dim2 = 0.12, dim3 = 1, fill = mlddarkgreen] (11.5-0.15, 0) {};
\tensor[dim1 = 1.5, dim2 = 1.2, dim3 = 0.12, fill = nofill] (11.5+0.75, -0.15) {};
\foreach \j in {0.15, 0.15*2, 0.15*3, 0.15*4, 0.15*5}{
\tensor[dim1 = 1.5, dim2 = 0, dim3 = 0, fill = nofill] (11.5+0.75+\j, -0.15) {};
\tensor[dim1 = 0, dim2 = 0, dim3 = 0.12, fill = nofill] (11.5+0.75+\j, 1.35) {};
}
\tensor[dim1 = 0.15, dim2 = 0.15, dim3 = 0, fill = black] (12.25, 0.88) {};
\tensor[dim1 = 0.15, dim2 = 0.15, dim3 = 0, fill = black] (12.25+0.15, 0.88-0.15) {};
\tensor[dim1 = 0.15, dim2 = 0.15, dim3 = 0, fill = black] (12.25+0.15*2, 0.88-0.15) {};
\tensor[dim1 = 0.15, dim2 = 0.15, dim3 = 0, fill = black] (12.25+0.15*3, 0.88) {};
\tensor[dim1 = 0.15, dim2 = 0.15, dim3 = 0, fill = black] (12.25+0.15*4, 0.88+0.15) {};
\tensor[dim1 = 0.15, dim2 = 1.2, dim3 = 1, fill = nofill] (11.5+0.75, 1.6) {};
\foreach \j in {0.15, 0.15*2, 0.15*3, 0.15*4, 0.15*5}{
\tensor[dim1 = 0, dim2 = 0, dim3 = 1, fill = nofill] (11.5+0.75+\j, 1.75) {};
\tensor[dim1 = 0.15, dim2 = 0, dim3 = 0, fill = nofill] (11.5+0.75+\j, 1.6) {};
}
\tensor[dim1 = 0, dim2 = 0.15, dim3 = 0.12, fill = black] (12.4, 1.9) {};
\tensor[dim1 = 0, dim2 = 0.15, dim3 = 0.12, fill = black] (12.4+0.15, 1.9) {};
\tensor[dim1 = 0, dim2 = 0.15, dim3 = 0.12, fill = black] (12.4+0.15*2, 1.9) {};
\tensor[dim1 = 0, dim2 = 0.15, dim3 = 0.12, fill = black] (12.4+0.15*2+0.09, 1.9-0.06) {};
\tensor[dim1 = 0, dim2 = 0.15, dim3 = 0.12, fill = black] (12.4+0.15*2+0.09*2, 1.9-0.06*2) {};
\node at (10.7, 1) {$+$};
\node at (4.55+3.35+3.5, -0.35) {$\T{A}_{:,3,:}$};
\node at (6+3.35+3.5, -0.5) {$\T{B}_{:,:,3}$};
\node at (6.5+3.35+3.5, 2.5) {$\T{C}_{3,:,:}$};
\end{tikzpicture}
\caption{Illustration of the generative spatiotemporal video model with two constant valued, different intensity, rectangular objects moving in different directions, on a constant background. Three BM-rank 1 terms are sufficient to capture the foreground objects' motion and the background}
\label{fig:genModel_rk3}
\end{figure}

\subsection{SS-SVD - BMD Connection for Our Generative Model}
\label{subsec:ssvd_bmd}
We are now in a position, using the generative model, to explain why the frontal slices of the surveillance video will have low (matrix) rank and connect our BMD-based model to the SS-SVD. 
Consider our generative model in which only a single pixel-size object is traversing from left to right in the $i$-th row of the background image $\M{X}$. This video tensor $\TX$ will have a BM-rank of two. Each frame is a {\it lateral} slice $\TX_{:,j,:}$. Then it is easy to see that the rank of each {\it frontal} slice $\TX_{:,:,k}$ is either 1 or 2, because $\TX_{:,j,k}$ is either $\M{X}_{:,k}$ (if no object is present at time $j$ on the $k$-th column of the background image) or 
$\M{X}_{:,k}+ (-\M{X}_{i,k}+\alpha) \V{e}_i$ (the single pixel object is present), where $\V{e}_i$ is the $i$-th standard basis vector. Thus, indeed, a constant multiple of the left singular vector $\V{u}_{1}^{(k)}$ will accurately approximate the background column pixel $\M{X}_{:,k}$.  However, by orthogonality and the fact that $c_k\V{u}_1^{k} \approx \M{X}_{:,k}$ for some constant $c_k \neq 0$, $\V{u}_2^{(k)}$ is approximately some multiple of $(-\M{x}_{i,k}+\alpha)\V{e}_i - \kappa_k \M{X}_{:,k}$ where $\kappa_k$ is also a non-zero constant. 
In sum, for $t =1 $, we approximate
$\TX^{\rm bg}$ in (\ref{eq:fgbg}) and the $t=2$ term approximates $\TX^{\rm fg}$ in (\ref{eq:fgbg}).  So both the foreground and the background have compressed representations in the BMD form. We will utilize this observation in our choice of starting guess for our algorithm.

%% file: table.tex
\begin{table}[ht]
\resizebox{1\linewidth}{!}
{\begin{minipage}{1.5\linewidth}
\centering
\begin{tabular}{|c|c|cccccc|} 
 \hline
\multicolumn{2}{|c|}{Video sequence}  & Simulated & Car & Escalator & Hall and
Monitor & Human Body & IBM Test \\ 
 \hline
\multicolumn{2}{|c|}{\footnotesize  Height $\times$ \# Frames  $\times$ Width} & $50\times 30 \times 50$ & $120\times 120 \times 160$ & $130\times 200 \times 160$ & $100\times 150 \times 147$ & $100\times 150 \times 134$ & $100\times 90 \times 134$\\
\hline
\multicolumn{2}{|c|}{BM-rank $\ell$}  & 2 & 5 & 8 & 3 & 6 & 4 \\ 
\hline
\multicolumn{2}{|c|}{CR } & 0.1467 & 0.1146 & 0.1515 & 0.0704 & 0.1448 & 0.1143 \\ 
\hline
\multirow{2}{3em}{\centering RE} &
$\text{BMD-ALS}_{SVD}$ & \textbf{0.0010} & 0.0374 &  0.0413  & \textbf{0.0304} & \textbf{0.0365} & 0.0405 \\
\cline{2-8}
& $\text{BMD-ALS}_{DMD}$ & 0.0111 & \textbf{0.0364} &  \textbf{0.0401} & 0.0311 & 0.0372 &  \textbf{0.0402} \\
 \hline
\end{tabular}
\end{minipage}}
\caption{Comparison of the video reconstruction results for the BMD-ALS algorithm with both the SS-SVD and the DMD initializations.} 
\label{tbl:BMD_results}
\end{table}

\begin{table}[ht]
\resizebox{1\linewidth}{!}
{\begin{minipage}{1.5\linewidth}
\centering
\begin{tabular}{|c|c|ccccc|} 
 \hline
\multirow{2}{10em}{\centering Video Sequence}  & \multirow{2}{7em}{\centering Methods} & \multicolumn{5}{c|}{Background Evaluation Metrics} \\
\cline{3-7}
& & AGE & pEPs & pCEPs & MS-SSIM & PSNR\\
\hline
\multirow{4}{10em}{\centering Simulated} & SS-SVD  & 4.1104&	0.0332&	0.0120&	0.9800	&30.1105\\
& DMD & 2.33 & 0.0392 &0.0192 & 0.9753 &27.15\\
& BMD-ALS$_{SVD}$ & \textbf{0.2750} & \textbf{0} & \textbf{0} & \textbf{0.9998} & \textbf{54.0179} \\
& BMD-ALS$_{DMD}$  &1.3276 &0.0111	&0.0041	&0.9927	&37.8710 \\
\hline
\multirow{4}{10em}{\centering Hall and Monitor} & SS-SVD & 11.6629	&0.0540	&0.0265	&0.9670	&25.3568\\
& DMD & \textbf{3.6932}	&0.0493	&0.0278	&\textbf{0.9745}	&29.8246 \\
& BMD-ALS$_{SVD}$& 5.9975	&\textbf{0.0439}	&\textbf{0.0242}	&0.9742	& \textbf{30.4387} \\
& BMD-ALS$_{DMD}$ &10.4643	&0.1357	&0.1059	&0.9594	&27.8144\\
\hline
\multirow{4}{10em}{\centering Human Body} & SS-SVD &22.6678	&0.3798	&0.2626	&0.8954	&17.7232\\
& DMD & 11.7410	&\textbf{0.1357}	&\textbf{0.0769}	&0.8718&19.8134\\
& BMD-ALS$_{SVD}$ & \textbf{11.6589} & 0.1772& 0.1166 & \textbf{0.9277} &	\textbf{21.5720}\\
& BMD-ALS$_{DMD}$ & 35.0373 &0.5935 &	0.4518&	0.7640	&15.9829\\
\hline
\multirow{4}{10em}{\centering IBM Test} & SS-SVD & 6.4078	&\textbf{0.0425}	&\textbf{0.0142}	&0.9822&	\textbf{29.4235}\\
& DMD & \textbf{5.1016}	&0.0904	&0.0619	&0.9616	&26.5307\\
& BMD-ALS$_{SVD}$ & 9.9211	&0.1410	&0.1004	&\textbf{0.9827}&27.9591\\
& BMD-ALS$_{DMD}$ & 11.6214	&0.0934	&0.0690	&0.9456	&18.5423\\
\hline
\end{tabular}
\end{minipage}}
\caption{Background evaluation metrics, comparison of the four methods: SS-SVD, DMD, $\text{BMD-ALS}_{SVD}$, and $\text{BMD-ALS}_{DMD}$ for video sequences with available ground-truth background image.} 
\label{tbl:bk_evals}
\end{table}

%% file: Supplement.tex

\section{ALS Convergence Analysis}
The alternating least-squares algorithm has been widely used for computing the tensor CP decomposition \cite{kolda2009tensor} and the block term decomposition \cite{navasca2008swamp, de2008decompositions} among others. The local and global convergence of the ALS algorithm for the tensor CP decomposition has been studied in \cite{li2013some, uschmajew2012local, yang2022global} based on its connection to the block nonlinear Gauss–Seidel (GS) method. In this section, we will show that the ALS algorithm for computing low BM-rank tensor approximation is also closely connected to the nonlinear GS method, and hence several convergence results follow directly from its framework. 

Recall the nonlinear GS method solves the following minimization problem
\begin{equation*}
\begin{split}
    \min \ & f\left(\V{x}\right)\\
    \text{subject to } &\V{x}\in X=X_{1}\times X_{2}\times\cdots\times X_{M}\subset\mathbb{R}^{N\times 1},
\end{split}
\end{equation*}
where $f$ is a continuously differentiable function from $\mathbb{R}^{N\times 1}$ to $\mathbb{R}$ and $X$ is a Cartesian product of closed, nonempty and convex subsets $X_{i}\subset\mathbb{R}^{N_{i}\times 1}$, for $i=1,\dots,M$ with $\sum_{i=1}^{M}N_{i}=N$. If the vector $\V{x}\in\mathbb{R}^{N}$ is partitioned into $M$ component vectors $\V{x}_{i}\in\mathbb{R}^{N_{i}\times 1}$, then we can consider $f$ is a function from $\mathbb{R}^{N_{1}\times 1}\times\mathbb{R}^{N_{2}\times 1}\times\cdots\times\mathbb{R}^{N_{M}\times 1}$ to $\mathbb{R}$ with $f\left(\V{x}\right)=f\left(\V{x}_{1},\V{x}_{2},\dots,\V{x}_{M}\right)$.
The solution to the nonlinear optimization function can then be found by the block Gauss-Seidel method via the following iteration in a cyclic order,
\begin{equation*}
    \V{x}_{i}^{k+1}=\underset{\V{y}_{i}\in X_{i}}{ \min}f\left(\V{x}_{1}^{k+1},\dots,\V{x}_{i-1}^{k+1},\V{y}_{i},\V{x}_{i+1}^{k},\dots,\V{x}_{M}^{k}\right),
\end{equation*}
which updates the components of $\Vx$. The iterative technique starts from a given initial guess $\Vx^{0} = \left(\Vx_1^{0},\Vx_2^{0},\dots, \Vx_M^{0}\right)$ and generates a sequence $\{\Vx^{k}\} = \left \{\left(\Vx_1^{k},\Vx_2^{k},\dots, \Vx_M^{k}\right) \right\}$.

The connection between the nonlinear block Gauss-Seidel method and the ALS-BMD algorithm is made evident by noting the cost function we want to minimize,
\begin{equation*}
    \left\Vert \T{X}-\hat{\T{X}}\right\Vert _{F}^{2}=\sum_{i,j,k}\left(\T{X}_{i,j,k}-\sum_{t=1}^{\ell} \T{A}_{i,t,k} \T{B}_{i,j,t} \T{C}_{t,j,k} \right)^{2}=f\left(\T{A},\T{B},\T{C}\right),
\end{equation*}
is a function $f:\mathbb{R}^{\left(mn+mp+pn\right)\ell \times 1}\rightarrow\mathbb{R}$. By letting $\V{v}=\left[\V{a};\V{b};\V{c}\right]\in\mathbb{R}^{\left(mn+mp+pn\right)\ell \times 1}$ where $\V{a},\V{b}$ and $\V{c}$ are the vectorized factor tensors, then we can see that
\begin{equation*}
    f\left(\V{v}\right)=f\left(\T{A},\T{B},\T{C}\right).
\end{equation*}
The BMD problem therefore can be reformulated into the following problem
\begin{equation*}
    \begin{split}
        \min \ & f\left(\V{v}\right) \\
\text{subject to } &	\V{v}\in\mathbb{R}^{mn\ell\times 1}\times\mathbb{R}^{mp\ell\times 1}\times \mathbb{R}^{pn\ell\times 1}.
    \end{split}
\end{equation*}
The ALS algorithm updates the components of $\V{v}$ by
\begin{equation*}
    \begin{split}
        \V{b}^{k+1} &= \underset{\V{y}\in\mathbb{R}^{mp\ell\times 1}}{ \min}f\left(\V{a}^{k},\V{y},\V{c}^{k}\right),\\
        \V{c}^{k+1} &= \underset{\V{y}\in\mathbb{R}^{pn\ell\times 1}}{ \min}f\left(\V{a}^{k},\V{b}^{k+1},\V{y}\right),\\
        \V{a}^{k+1} &= \underset{\V{y}\in\mathbb{R}^{mn\ell\times 1}}{\min}f\left(\V{y},\V{b}^{k+1},\V{c}^{k+1}\right).
    \end{split}
\end{equation*}
This is exactly the nonlinear block Gauss-Seidel method.

By Proposition 2.1 in \cite{yang2022global}, let $\Vv^{k}$ denote the $k$-th solution vector, i.e. $\Vv^{k}=\left(\Va^{k},\Vb^{k}, \Vc^{k}\right)$. When the normal equations matrix in each of the linear least-squares subproblems given in Eq.~(6.3), Eq.~(6.5), and Eq.~(6.6) in the main text is positive definite, i.e. $\M{H}^{\top}\M{H}\succ 0$ for any $\M{H}\in\{\M{H}_{\TA\TC}, \M{H}_{\TB\TA}, \M{H}_{\TC\TB}\}$, then each least-squares subproblem has a unique solution and the whole sequence $\{\Vv^k\}$ generated by ALS converges to a limit point. In particular,
\begin{equation*}
    f\left(\Vv^{k+1}\right) \leq f\left(\Va^{k}, \Vb^{k+1}, \Vc^{k+1}\right) \leq f\left(\Va^{k}, \Vb^{k+1}, \Vc^{k}\right) \leq f\left(\Vv^{k}\right) \leq \cdots \leq f\left(\Vv^{0}\right),
\end{equation*}
shows that ALS monotonically reduces the cost function. Since $f$ is bounded below, $\{f(\Vv^k)\}$ has a limit point $f^* \geq 0$. Moreover, by Theorem 3.1 in  \cite{yang2022global}, if $\{\V{v}^{k}\}$ is bounded, then the limit point of the sequence is a stationary point of the problem.

\subsection{Regularized ALS Convergence Analysis}
\input{rals_convergence}

\section{Nonlinear Solution for Computing BMD}
Given $\T{T}\in \mathbb{R}^{m\times p\times n}$, solving BMD with non-linear optimization
\begin{equation*}
    \underset{\TA,\TB,\TC}{\arg\min} \|\T{T} - \bmp(\TA,\TB,\TC)\|_{F}^{2}
\end{equation*}
for the factor tensor triplet $(\TA,\TB,\TC)\in  \mathbb{R}^{m\times \ell\times n}\times  \mathbb{R}^{m\times p\times \ell} \times  \mathbb{R}^{\ell \times p\times n}$. 

Let us take the vectorized factor tensors 
\begin{equation*}
    \V{a} =\vvec(\TA);\quad 
    \V{b} =\vvec(\TB);\quad 
    \V{c} =\vvec(\TC),
\end{equation*}
and let $\V{v} = [\V{a}, \V{b}, \V{c}] \in \mathbb{R}^{(mn + mp +  pn)\ell \times 1}$. Define the residual vector $\V{r}(\V{v})$ to be
\[
    \V{r}(\V{v}) =  \vvec\left(\T{T} \right) - \vvec\left(\bmp\left(\TA,\TB,\TC\right)\right). 
\]
Every entry of $\V{r}$ corresponds to one $(i,j,k)$ triple, so we adopt this indexing scheme, 
\[ \V{r}_s = r_{ijk} = t_{ijk} - \sum_{t=1}^\ell a_{itk} b_{ijt} c_{tjk}, \]
where $s = i+(j-1)m + (k-1)mp$.
Then the Jacobian matrix $\M{J}$ of the residual $\V{r}(\Vv)$ is of size $mpn \times (m\ell n + mp\ell+\ell pn)$ and is given by
\begin{equation*}
 \M{J}	=\left(\begin{array}{ccccccc}
 & \vdots &  & \vdots &  & \vdots\\
\cdots & \frac{\partial {r}_{ijk}}{\partial {a}_{itk}} & \cdots & \frac{\partial {r}_{ijk}}{\partial {b}_{ijt}} & \cdots & \frac{\partial {r}_{ijk}}{\partial {c}_{tjk}} & \cdots\\
 & \vdots &  & \vdots &  & \vdots
\end{array}\right),
\end{equation*}
where, for fixed $(i,j,k)$, the only non-zero entries are 
$\frac{\partial r_{ijk}}{\partial a_{itk}}=-b_{ijt}c_{tjk}, \frac{\partial r_{ijk}}{\partial b_{ijt}}=-a_{itk}c_{tjk}$ and $\frac{\partial r_{ijk}}{\partial c_{tjk}}=-a_{itk}b_{ijt}$, for 
$1 \leq t \leq \ell$.

If $\M{J}(\V{v}^{k})$ has full column rank, then the Gauss-Newton method updates the solution by
\begin{equation*}
    \V{v}^{k+1} = \V{v}^{k} - \left(\M{J}(\V{v}^{k})^{\top}\M{J}(\V{v}^{k}) \right)^{-1} \M{J}(\V{v}^{k})^{\top} \V{r}(\V{v}^{k})
\end{equation*}
However, the Jacobian matrix for the BM-decomposition is singular in general. 
As a result, the Gauss-Newton method does not converge to a local minimum. On the other hand, the Levenberg-Marquardt method adds an additional penalty term in the least-squares problem by solving the following problem
\begin{equation*}
    \V{v}^{k+1} =\underset{\V{v}}{\arg\min} \left\Vert \V{r}(\V{v}) +\M{J}(\Vv^k)(\Vv - \Vv^k)\right\Vert_F^2+\lambda_k \left\Vert \V{v}-\V{v}^{k}\right\Vert_F^2. 
\end{equation*}
where $\V{r}(\V{v}) +\M{J}(\Vv^k)(\Vv - \Vv^k)$ is the first-order linear approximation to $\V{r}(\V{v})$. Then the solution is given by
\begin{equation*}
    \V{v}^{k+1} = \V{v}^{k} - \left(\M{J}(\V{v}^{k})^{\top}\M{J}(\V{v}^{k}) + \lambda_k \M{I} \right)^{-1} \M{J}(\V{v}^{k})^{\top} \V{r}(\V{v}^{k}).
\end{equation*}

\section{From DMD to Tensor BMP}
Given a video data of $p$ frames and each frame is of size $m\times n$, let $\Vx_k\in \mathbb{R}^{mn\times 1}$ be the vectorized $k$-th video frame for $1\leq k \leq p$, then the DMD method groups the data into matrices $\M{X}_1, \M{X}_2 \in \mathbb{R}^{mn\times (p-1)}$ as follows
\begin{equation*}
    \M{X}_1 = \left[\Vx_1 \ \Vx_2 \ \Vx_3 \cdots \Vx_{p-1} \right];\quad \M{X}_2 = \left[\Vx_2 \ \Vx_3 \ \Vx_4 \cdots \Vx_{p} \right].
\end{equation*}
The method then finds a linear map $\M{A} \in \mathbb{R}^{mn\times mn}$ such that $\Vx_{k+1} = \M{A}\Vx_{k}$. Thus $\M{X}_2\approx \M{A}\M{X}_1$.

The SVD of the matrix $\M{X}_1$ with truncation to $\ell$ terms can be used for dimensionality reduction, i.e. $\M{X}_1 = \M{U} \M{\boldsymbol{\Sigma}} \M{V}^{\top}$, where $ \M{U}\in \mathbb{R}^{mn\times \ell}$ is unitary, $\M{\Sigma}\in \mathbb{R}^{\ell \times \ell}$ is diagonal and $\M{V} \in \mathbb{R}^{(p-1) \times \ell}$ is unitary. After projecting $\M{A}$ onto the left-singular vector matrix $\M{U}$ as $\Tilde{\M{A}} = \M{U}^{\top}\M{A}\M{U} = \M{U}^{\top}\M{X}_2\M{V}\M{\Sigma}^{-1}$, $\Tilde{\M{A}} \in \mathbb{R}^{\ell \times \ell}$, we compute the eigen-decomposition of $\Tilde{\M{A}}$ such that $\Tilde{\M{A}}\M{W} = \M{W}\M{\Lambda}$, 
with $\M{\Lambda} =\text{diag}\left(\lambda_t \right) \in \mathbb{C}^{\ell \times \ell}$ where $\lambda_t, 1\leq t \leq \ell$, are DMD eigenvalues which can be converted to Fourier frequencies via $\omega_t = \frac{\ln{(\lambda_t)}}{\Delta t}$. Assuming video frames with equally spaced time, $\Delta t=1$, $e^{\omega_t} = \lambda_t$.  The DMD modes $\M{\Phi} \in \mathbb{C}^{mn\times \ell}$ are obtained by computing
\begin{equation}
    \M{\Phi} = \M{X}_2\M{V}\M{\Sigma}^{-1}\M{W}.
    \label{eq:dmd_mode_matrix}
\end{equation}
Taking the DMD modes and the DMD frequencies, the original vectorized video frame at time $k=1,2,\dots,p$ can be reconstructed by
\begin{equation}
    \V{x}_k \approx \sum_{1\leq t\leq \ell} b_t \V{\varphi}_t e^{\omega_t(k-1)} = \sum_{1\leq t\leq \ell} b_t \V{\varphi}_t \lambda_{t}^{k-1},
    \label{eq:DMD_approx}
\end{equation}
with $ \V{\varphi}_t = \M{\Phi}_{:,t}$ the $t$-th DMD mode, and the vector $\V{b}=\left[b_1 \, b_2 \, \cdots \, b_{\ell} \right]^{\top}$ contains the initial amplitudes
for the modes and is obtained by solving $\M{\Phi}\V{b} = \V{x}_1$ where $\V{x}_1$ the vectorized first frame.

Assume the set of DMD frequencies $\omega_{\alpha}$, $1\leq \alpha \leq \ell$, satisfies $\|\omega_{\alpha}\|\approx 0$, i.e. $\|\lambda_{\alpha}\| = \|e^{\omega_{\alpha}}\|\approx 1$ for $1\leq \alpha \leq \ell$. Typically, $\alpha$ is a single index rather than a vector. Then the background and the foreground video sequences of $\M{X}$ are reconstructed with the DMD technique for the time vector $\boldsymbol{\theta}=[0, 1, \dots, p-1]$ respectively by
\begin{equation}
    \M{X}^{\text{bg}} \approx  \sum_{1\leq t=\alpha \leq \ell} b_{t} \V{\varphi}_{t} \lambda_{t}^{\boldsymbol{\theta}}; \quad \M{X}^{\text{fg}} \approx \sum_{1\leq t\neq \alpha \leq \ell} b_t \V{\varphi}_t \lambda_{t}^{\boldsymbol{\theta}}.
    \label{eq:dmd_bg_fg}
\end{equation}
We note that, as discussed in \cite{grosek2014dynamic}, since it should be true that the video sequence $\M{X}=\M{X}^{bg}+\M{X}^{fg}$, so a real-valued foreground approximation can be alternatively obtained by
\begin{equation}
    \M{X}^{\text{fg}} \approx \M{X} - |\M{X}^{\text{bg}}|,
    \label{eq:dmd_fg_v2}
\end{equation}
where $|\cdot|$ yields the modulus of each element within the matrix. However, we note that the expression given by the second approximation in Eq.~(\ref{eq:dmd_fg_v2}) is not a compressible representation of the foreground video sequence. For turning the DMD reconstruction into the BMD-ALS initial guesses, we take the background/foreground approximation given in Eq.~\ref{eq:dmd_bg_fg}, which we will show in the next paragraph.

Next, we show that the video reconstruction by the DMD method given in Eq.~(\ref{eq:DMD_approx}) can be equivalently written in a tensor BM-product form. Let us first define the DMD mode matrix $\M{M}_t \in \mathbb{C}^{m \times n}$ to be
\begin{equation}
    \M{M}_t = \reshape(b_t\V{\varphi}_t, [m,n])
\end{equation}
for all $t=1,\dots,\ell$. That is, the DMD mode vector $\V{\varphi}_t$ is scaled by the corresponding initial amplitude $b_t$ and converted into an $m\times n$ matrix. 

Next, we define a third-order tensor triplet
\begin{enumerate}
    \item $\TA$: an $m\times \ell \times n$ DMD mode tensor with lateral slices $\TA(:,t,:) = \M{M}_t$,
    \item $\TB$: an $m\times p \times \ell$ tensor of ones,
    \item $\TC$: an $\ell \times p \times n$ DMD eigenvalue tensor with row vectors, $\TC_{t,:,k} = \lambda_t^{\V{\theta}}$, $1\leq k \leq p$.
\end{enumerate}
Ordering the video frames as lateral slices to form a video tensor $\TX \in \mathbb{R}^{m\times p \times n}$, the DMD reconstruction of Eq.~(\ref{eq:DMD_approx}) can be written in the tensor BMP form as
\begin{equation*}
    \TX \approx \bmp \left(\T{A}, \T{B}, \T{C} \right).
\end{equation*}
We can update $\T{B}$ as discussed in Phase II in Sec.~6.3 from the main text, and then subsequently optimizing each factor tensor using ALS to improve the low BM-rank approximation. The BMD--ALS reconstructed background and foreground tensors are given by 
\begin{equation*}
\begin{split}
    \TX^{\text{bg}}_{\text{BMD--ALS}} &\approx \sum_{1\leq t=\alpha \leq \ell} \bmp\left(\TA_{:,t,:},\TB_{:,:,t},\TC_{t,:,:} \right);\\
    \TX^{\text{fg}}_{\text{BMD--ALS}} &\approx \sum_{1\leq t\neq \alpha \leq \ell} \bmp\left(\TA_{:,t,:},\TB_{:,:,t},\TC_{t,:,:} \right).
\end{split}
\end{equation*}

\subsection{DMD factor interpretation}
We note that the first step of the DMD method which applies a truncated SVD to $\M{X_1}$ is similar to taking the SVDs of spatiotemporal slices of the video tensor (with one frame less). Instead of applying SVD to individual spatiotemporal slices, DMD applies SVD to all slices at the same time. As a result, the left-singular matrix $\M{U}$ similarly captures the spatial information of the video. Particularly, as discussed in Sec.~4.2, the first left-singular vector captures the dominant background scene. However, the DMD method that models the static background image in-fact models a weighted linear combination of all the left-singular vectors (spatial information).

From the eigendecomposition of $\tilde{\M{A}}$, we have 
$\tilde{\M{A}}\M{W}=\M{U}^{\top}\M{X}_{2}\M{V}\M{\Sigma}^{-1}\M{W}=\M{W}\boldsymbol{\Lambda}$, which then gives
$\boldsymbol{\Phi} = \M{X}_{2}\M{V}\M{\Sigma}^{-1}\M{W}=\M{U}\M{W}\boldsymbol{\Lambda}$. Therefore, the DMD modes given in Eq. (\ref{eq:dmd_mode_matrix}) can be re-written as
\begin{equation}
    \boldsymbol{\Phi}_{i,j} =\sum_{1\leq s\leq\ell}\lambda_{j}\M{U}_{i,s}\M{W}_{s,j}
    \implies\boldsymbol{\Phi}_{:,j} =\sum_{1\leq s\leq\ell}\lambda_{j}\M{U}_{:,s}\M{W}_{s,j}.
\end{equation}
Thus, the background DMD mode represented by $\boldsymbol{\Phi}_{:,\alpha}$ when $\|\lambda_{\alpha}\|\approx 1$ for all $\alpha, 1\leq \alpha \leq \ell $, is given by
\begin{equation}
\boldsymbol{\Phi}_{:,\alpha}=\sum_{1\leq s\leq\ell}\lambda_{\alpha}\M{U}_{:,s}\M{W}_{s,\alpha}.
\label{eq:dmd_bg_sum}
\end{equation}
From Eq.~(\ref{eq:dmd_bg_sum}), we see that the DMD reconstruction of the stationary mode is a weighted linear combination of the spatial information captured by the left-singular vectors of the spatiotemporal video matrix. This could potentially explain the observations made in the numerical experiments of the videos in \cite{grosek2014dynamic}: the DMD reconstructed background scene tends to include spurious foreground pixels. Since in the DMD method, the SVD step is applied to vectorized video frames across all $p$ frames, taking the rank-truncation with $\ell = p-1$ almost recovers the original video matrix. The spatial information of the foreground objects moving through time is captured by the set of left-singular vectors $\M{U}_{:,t}$ with $2\leq t \leq \ell$, which are also contained in the background DMD mode $\boldsymbol{\Phi}_{:,\alpha}$. This could also explain another phenomenon about the foreground motions in \cite{grosek2014dynamic}: the moving objects in the foreground create movement trails extrapolating both past and future motions of the objects. Since the uncompressed foreground video by the DMD algorithm is obtained by subtracting the background from the original video, the erroneous foreground motion trails contained in the background DMD reconstruction persisted in the foreground video. 

Importantly, the DMD method does not provide an approximation to the generative video model from Sec.~5.1 in the main text. In addition to the background reconstruction with spurious object motion discussed above, the DMD compressed foreground video frames exhibit another issue. Since the foreground modes $\V{\varphi}_t$, $1\leq t \neq \alpha \leq \ell$, change over time linearly scaled by $\lambda_t^{j-1}$ at time $j$, the non-linear foreground object motions cannot be exactly modeled by the linear change.

We also observe the following:
\begin{itemize}
    \item The DMD method requires processing video segments rather than computing the decomposition of the whole video at once. One reason for taking smaller segments is that it helps with keeping shorter processing time than data acquisition time. Also, longer videos produce longer erroneous movement trails on the approximated background.
    \item DMD modes and frequencies are complex-valued. In \cite{grosek2014dynamic}, the real-valued video pixels are obtained by taking the modulus of the complex values. The code from \cite{kutz2016dynamic} uses the real part of the complex-valued video frame entries for display. In our BMD--ALS video reconstruction with DMD initialization, we will also use real-part solutions.
    \item The DMD method has the potential of a compressed reconstruction of the background scene by setting $\ell=1$ in the SVD step \cite{grosek2014dynamic}. In our work, by allowing $\ell$ to be small, i.e. $1< \ell \ll p$, we are interested in obtaining the compressed representations of the background scene and the foreground object videos.
\end{itemize}

\section{Comparison to other tensor methods}
\subsection{CP form} \label{subsec:bmp_to_cp}
Just as other non-CP tensor approximations can be expressed as a sum of rank-one outer products of tensor, so too, can the BMD.  

We consider $\ell=1$, since the case for $\ell > 1$ is a direct extension.  For simplicity, we assume entries are real.
Define three matrix slices $\MA := \squeeze(\TA_{:,1,:})$ is $m \times n$, $\MB := \TB_{:,:,1}$ is $m \times p$ and 
$\MC := \squeeze(\TC_{1,:,:})$ is $p \times n$.  
We compute the SVDs of each of these matrices.  Superscripts $A, B$, and $C$ will denote to which matrix the SVD components are associated.

The $j$-th column of $\MA$, $\MA_{:,j}$ can be written as $\displaystyle{\sum_{k=1}^{r_A} \sigma_k^A {v}_{kj}^A \V{u}_{k}^A}$, the $j$-th column of $\M{C}$ would be $\displaystyle{\sum_{s=1}^{r_C} \sigma^{C}_s v_{sj}^C \V{u}_s^C}$ and  the matrix $\MB$ is $\displaystyle{\sum_{i=1}^{r_B} {\sigma}^{B}_i \V{u}_i^B (\V{v}_i^{\top})^B}$, where the ranks of $\MA, \MB, \M{C}$ are $r_A,r_B,r_C,$ respectively. 

Now the $j$-th frontal face of $\bmp(\TA,\TB,\TC)$ will be
\begin{equation*}
    \bmp(\TA,\TB,\TC)_{:,:,j} = \diag(\MA_{:,j}) \, \MB \, \diag(\MC_{:,j}).
\end{equation*}
Because a matrix-vector product between a diagonal matrix and a vector can be expressed using Hadamard product, denoted $\odot$, we obtain an expression for the $j$-th frontal slice:
\begin{equation*}
    \bmp(\TA,\TB,\TC)_{:,:,j} = \sum_{k,s,i} \sigma_k^A \sigma_i^B \sigma_s^C (\V{u}_k^A \odot \V{u}_i^B) (\V{v}_i^B \odot \V{u}_s^C)^{\top}.
\end{equation*}
It follows that the outer product of this matrix triple is a third-order tensor that can be expressed as
\begin{equation*} 
\bmp(\TA,\TB,\TC) = \sum_{k,s,i} (\sigma_k^A \sigma_i^B \sigma_s^C)
(\V{u}_k^A \odot \V{u}_i^B) \circ (\V{v}_i^B \odot \V{u}_s^C) \circ (\V{v}_k^A \odot \V{v}_s^C) .\end{equation*}
Thus the tensor rank is bounded by $r_A r_B r_C$ and the vectors have two-norms bounded above by 1. For each $\ell > 1$, we would have a similar expression.  This is only useful in bounding the tensor (CP) rank if the matrix ranks are very small. 

On the other hand, if we have a CP decomposition of $\T{X}\in\mathbb{C}^{m\times p\times n}$,
we can get a bound on the BM-rank by looking at the ranks of the factor matrices. 
\begin{theorem} 
Let $\TX$ have the CP decomposition
$\displaystyle{\T{X}=\sum_{1\leq t\leq r}\M{A}_{:,t}\circ \M{B}_{:,t} \circ \M{C}_{:,t}}$ \footnote{The symbol ``$\circ$'' represents the vector outer product \cite{kolda2009tensor}.}
where $\M{A}\in\mathbb{C}^{m\times r}$, $\M{B}\in\mathbb{C}^{p\times r}$, and $\M{C}\in\mathbb{C}^{n\times r}$ are the factor matrices and $r$ is the (real) CP tensor rank. 

Suppose the factor matrices have ranks $\rho_A,\rho_B$, and $\rho_C$, with $\rho:=\min\{\rho_A,\rho_B,\rho_C\}$. Then the BM-rank is bounded above by $\rho$.
\end{theorem}
\begin{proof}
If $\rho:= \min\{m,p,n\}$, this is trivial, since it is no different than the previous upper bound. Thus, let $\rho < \min\{m,p,n\}$. Due to the orientation independence of the CP decomposition, we assume without loss of generality that $\rho_A = \rho$.

The CP-decomposition of $\TX$ admits an expression of its $k$-th frontal slice \cite{kolda2009tensor} as$\T{X}_{:,:,k} = \M{A}\diag(\M{C}_{k,:})\M{B}^{\top},$
where $\diag(\M{C}_{k,:}) \in \mathbb{R}^{r \times r}$. Then we can take a rank-revealing factorization of $\M{A}$ such that $\M{A} = \M{U}_A \M{V}_A^{\top}$ with $\M{U}_A \in \mathbb{R}^{m \times \rho_A}$, $\M{V}_A \in \mathbb{R}^{r \times \rho_A}$,  $\rho_A =\rho$ columns.
Setting tensors $\TA \in \mathbb{R}^{m\times \rho \times n}$, $\T{K} \in \mathbb{R}^{m\times p \times \rho}$, and $\TC \in \mathbb{R}^{\rho\times p \times n}$ as
\begin{equation*}
\T{A}_{:,:,k} = \M{U}_A;  \qquad 
\T{C}_{:,:,k} = \M{V}_A^\top \diag(\M{C}_{k,:}) \M{B}^{\top}; 
\qquad        
\T{K} =\ones\left(m,p,\rho\right),
\end{equation*}
for $k=1,\ldots,n$, then
$\T{X}=\bmp\left(\T{A},\T{K},\T{C}\right)$. 
\end{proof}

\subsection{Comparison to tensor SVD under $\star_M$}
In \cite{kilmer2021tensor}, the authors describe a tensor SVD for third-order tensors under a specific type of product between pairs of tensors of appropriate dimension.  
What is needed to define the tensor-tensor $\star_{M}$ product is a unitary or orthogonal $n \times n$ matrix, $\M{M}$.  Under the resulting multiplication, truncating the tensor-SVD gives an optimal approximation in the Frobenius norm. 
Here, we take $\M{M} = \M{I}$, so the starting guess $\widehat{\TX}$ defined by the SS-SVD will be exactly the $\ell$ term approximation under the $\star_M$ product.

Importantly, they showed for any choice of orthogonal/unitary $\M{M}$, if $\TX_{\ell}$ is the truncated t-SVDM approximation under $\star_M$, that 
$\| \TX - \TX_\ell \|_F \le \| \M{X} - \M{X}_\ell \|_F$ 
where the $\M{X}$ here is the unstacked $\TX$ as a $mn \times p$ matrix, and $\M{X}_\ell$ denotes the $\ell$-term truncated matrix SVD approximation.  It was shown that strict inequality can be achieved.
Therefore, 
\begin{equation*}
    \| \TX - \bmp(\TA^{(k)},\TB^{(k)},\TC^{(k)}) \|_F < \| \TX - \TX_{\ell} \|_F \le \| \M{X} - \M{X}_{\ell} \|_F,
\end{equation*}
where the superscripts indicate the ALS iteration count of the $\ell$-term BMD approximation and $\widehat{\TX} = \TX_\ell$
is the $\ell$-term t-SVDM approximation with $\M{M}=\M{I}$.



%% file: rals_convergence.tex
Convergence analysis for the ALS algorithm with Tikhonov regularization has been studied for the CP decomposition \cite{karim2019tensor}. We note that a similar analysis can be derived for the regularized ALS algorithm for BM-decomposition given in Section.~7. 

Given a third order data tensor $\T{T}\in\mathbb{R}^{m\times p\times n}$, our goal is to find a tensor BM-decomposition $\T{A}\in\mathbb{R}^{m\times\ell\times n}$, $\T{B}\in\mathbb{R}^{m\times p\times\ell}$, and $\T{C}\in\mathbb{R}^{\ell\times p\times n}$ from the constrained problem
\begin{equation}\label{eq:tikonov-bmd}
    \min_{\T{A},\T{B},\T{C}}\left\Vert \T{X}-\text{bmp}\left(\T{A},\T{B},\T{C}\right)\right\Vert _{F}^{2}+\frac{1}{2}\left(\left\Vert \M{L}_{a}\left(\T{A}\right)\right\Vert _{F}^{2}+\left\Vert \M{L}_{b}\left(\T{B}\right)\right\Vert _{F}^{2}+\left\Vert \M{L}_{c}\left(\T{C}\right)\right\Vert _{F}^{2}\right),
\end{equation}
where $\M{L}_a$, $\M{L}_b$, and $\M{L}_c$ are linear operators acting on the vectorized factor tensors $\T{A}$, $\T{B}$, and
$\T{C}$ respectively. More specifically, for all $1\leq i\leq m, 1\leq j\leq p, 1\leq k\leq n$, $\M{L}_{a} = \underset{k,i}{\oplus}\M{L}$ where $\M{L} = \diag([\lambda_1, \lambda_2, \dots, \lambda_{\ell}])$, $\M{L}_{b} = \underset{i,j}{\oplus} \beta \M{I}_{\ell\times \ell}$ and $\M{L}_{c} = \underset{j,k}{\oplus} \gamma\M{I}_{\ell\times \ell}$. 

Then the regularized ALS subproblems are given by
\begin{equation*}
    \hat{\T{B}}\left[i,j,:\right]=\min_{\V{b}^{\left(i,j\right)}\in\mathbb{R}^{\ell\times1}}\left\Vert \V{y}_{\T{T}}^{\left(i,j\right)}-\M{H}_{\T{A}\T{C}}^{\left(i,j\right)}\V{b}^{\left(i,j\right)}\right\Vert _{F}^{2}+\frac{1}{2}\left\Vert \M{L}_{b}^{\left(i,j\right)}\V{b}^{\left(i,j\right)}\right\Vert _{F}^{2},
\end{equation*}
\begin{equation*}
    \hat{\T{C}}^{\top}\left[j,k,:\right]=\min_{\V{c}^{\left(i,j\right)}\in\mathbb{R}^{\ell\times1}}\left\Vert \V{y}_{\T{T}^{\top}}^{\left(j,k\right)}-\M{H}_{\T{B}\T{A}}^{\left(j,k\right)}\V{c}^{\left(j,k\right)}\right\Vert _{F}^{2}+\frac{1}{2}\left\Vert \M{L}_{c}^{\left(j,k\right)}\V{c}^{\left(j,k\right)}\right\Vert _{F}^{2},
\end{equation*}
\begin{equation*}
    \hat{\T{A}}^{\top^{2}}\left[k,i,:\right]=\min_{\V{a}^{\left(i,j\right)}\in\mathbb{R}^{\ell\times1}}\left\Vert \V{y}_{\T{T}^{\top^{2}}}^{\left(k,i\right)}-\M{H}_{\T{C}\T{B}}^{\left(k,i\right)}\V{a}^{\left(i,j\right)}\right\Vert _{F}^{2}+\frac{1}{2}\left\Vert \M{L}_{a}^{\left(k,i\right)}\V{a}^{\left(k,i\right)}\right\Vert _{F}^{2},
\end{equation*}
for all $0\leq i<m$, $0\leq j<n$, and $0\leq k<p$.
Let $\Psi$ represents the objective function in Eq.~\ref{eq:tikonov-bmd}, then $\Psi:\mathbb{R}^{\ell\left(mp+mn+np\right)}\rightarrow\mathbb{R}^{+}$, where 
\begin{equation}\label{eq:tikhonov-obj}
    \Psi\left(\T{A},\T{B},\T{C}\right)=f\left(\T{A},\T{B},\T{C}\right)+\frac{1}{2}\left(\left\Vert \M{L}_{a}\left(\T{A}\right)\right\Vert _{F}^{2}+\left\Vert \M{L}_{b}\left(\T{B}\right)\right\Vert _{F}^{2}+\left\Vert \M{L}_{c}\left(\T{C}\right)\right\Vert _{F}^{2}\right)
\end{equation}

We will first show that the objective functions in the subproblems are $\mu$-strongly convex.

\begin{definition} A differentiable function $h:\text{dom}\left(h\right)\rightarrow\mathbb{R}$, where $\text{dom}\left(h\right)\subseteq\mathbb{R}^{n}$ is called $\mu$-strongly convex if there exists a constant $\mu>0$ such that 
\begin{equation*}
    h\left(y\right)\geq h\left(x\right)+\nabla h\left(x\right)^{\top}\left(y-x\right)+\frac{\mu}{2}\left\Vert y-x\right\Vert _{2}^{2},\forall\,x,y\in\text{dom}\left(h\right).
\end{equation*}
\end{definition}
\begin{theorem}
The objective function 
\begin{equation*}
f\left(\V{a},\V{b}^{k},\V{c}^{k}\right)+\frac{1}{2}\left\Vert \M{L}_{a}\V{a}\right\Vert _{F}^{2}
\end{equation*}
is $\mu$-strongly convex.
\end{theorem}
\begin{proof}
    The least-squares problem $f\left(\V{a},\V{b}^{k},\V{c}^{k}\right)=\left\Vert \V{y}_{\T{T}^{\top^{2}}}-\M{H}_{\T{C}\T{B}}\V{a}\right\Vert _{F}^{2}$ is convex. The regularization term $\frac{1}{2}\left\Vert \M{L}_{a}\V{a}\right\Vert _{F}^{2}$ has the following first derivative
    \begin{equation*}
        \frac{\partial}{\partial\V{a}}\left\{ \frac{1}{2}\left\Vert \M{L}_{a}\V{a}\right\Vert _{F}^{2}\right\} =\M{L}_{a}^{\top}\M{L}_{a}\V{a}
    \end{equation*}
    and second derivative
    \begin{equation*}
        \frac{\partial^{2}}{\partial\V{a}\partial\V{a}^{\top}}\left\{ \frac{1}{2}\left\Vert \M{L}_{a}\V{a}\right\Vert _{F}^{2}\right\} =\frac{\partial}{\partial\V{a}}\left\{ \M{L}_{a}^{\top}\M{L}_{a}\V{a}\right\} =\M{L}_{a}^{\top}\M{L}_{a},
    \end{equation*}
which is positive-semidefinite, since $\M{L}_{a}$ is a direct sum of $mn$ diagonal matrices with non-negative entries. So there exists a constant $\mu>0$ such that $\frac{\partial^{2}}{\partial\V{a}^{2}}\left\{ \frac{1}{2}\left\Vert \M{L}_{a}\V{a}\right\Vert _{F}^{2}\right\} \succcurlyeq\mu\M{I}$. 

A linear combination of a convex and a $\mu$-strongly convex function is also $\mu$-strongly convex. Hence, the objective function of the subproblem with respect to $\T{A}$ is $\mu$-strongly convex. 
\end{proof}

Similarly, the objective function of the subproblems with respect to $\T{B}$ and $\T{C}$ are also $\mu$-strongly convex. 

\begin{lemma}
    Suppose $\V{a}^{k+1}$ is obtained by solving the subproblem have decrease in the objective function after a single update of $\V{a}$, i.e.
\begin{equation*}
    \Psi\left(\mathbf{a}^{k}\right)-\Psi\left(\mathbf{a}^{k+1}\right)\geq\frac{\mu}{2}\left\Vert \mathbf{a}^{k}-\mathbf{a}^{k+1}\right\Vert _{2}^{2}
\end{equation*}
for some constant $\mu>0$.
\end{lemma}
\begin{proof}
    By the first order optimality, we have
    \begin{equation*}
        \nabla_{a}\Psi\left(\mathbf{a}^{k+1}\right)=\nabla_{a}f\left(\mathbf{a}^{k+1}\right)+\mathbf{L}_{a}^{\top}\mathbf{L}_{a}\mathbf{a}=0.
    \end{equation*}
    Then the strongly convexity of the objective function yields
    \begin{equation*}
        \Psi\left(\mathbf{a}^{k}\right)-\Psi\left(\mathbf{a}^{k+1}\right)\geq\nabla_{a}\Psi\left(\mathbf{a}^{k+1}\right)^{\top}\left(\mathbf{a}^{k}-\mathbf{a}^{k+1}\right)+\frac{\mu}{2}\left\Vert \mathbf{a}^{k}-\mathbf{a}^{k+1}\right\Vert _{F}^{2}\geq\frac{\mu}{2}\left\Vert \mathbf{a}^{k}-\mathbf{a}^{k+1}\right\Vert _{2}^{2}.
    \end{equation*}
\end{proof}
Similar results hold for $\mathbf{b}$ and $\mathbf{c}$. 

Next we want to show that $\Psi$ decreases monotonically at each iteration. 
\begin{theorem}\label{thm:decrease}
    (Sufficiently decrease property) Let $\Psi$ represent the objective function in Eq.~\ref{eq:tikhonov-obj}, let $\boldsymbol{\omega}^{k}=\left(\mathbf{a}^{k},\mathbf{b}^{k},\mathbf{c}^{k}\right)$, then we have
    \begin{equation*}
        \Psi\left(\boldsymbol{\omega}^{k}\right)-\Psi\left(\boldsymbol{\omega}^{k+1}\right)\geq\rho\left\Vert \boldsymbol{\omega}^{k}-\boldsymbol{\omega}^{k+1}\right\Vert _{2}^{2}
    \end{equation*}
    for some constant $\rho>0$. In addition, we have
    \begin{equation*}
        \sum_{k=0}^{\infty}\left\Vert \boldsymbol{\omega}^{k}-\boldsymbol{\omega}^{k+1}\right\Vert _{2}^{2}<\infty.
    \end{equation*}
\end{theorem}
\begin{proof}
    By the previous result in 2, we have 
    \begin{equation*}
        \Psi\left(\mathbf{\boldsymbol{\omega}}^{k}\right)-\Psi\left(\boldsymbol{\omega}^{k+1}\right)\geq\frac{1}{2}\left\{ \mu_{1}\left\Vert \mathbf{a}^{k}-\mathbf{a}^{k+1}\right\Vert _{2}^{2}+\mu_{2}\left\Vert \mathbf{b}^{k}-\mathbf{b}^{k+1}\right\Vert _{2}^{2}+\mu_{3}\left\Vert \mathbf{c}^{k}-\mathbf{c}^{k+1}\right\Vert _{2}^{2}\right\} ,
    \end{equation*}
    Then
    \begin{equation*}
        \lim_{n\rightarrow\infty}\sum_{k=0}^{n-1}\left\Vert \boldsymbol{\omega}^{k}-\boldsymbol{\omega}^{k+1}\right\Vert _{2}^{2}\leq\lim_{n\rightarrow\infty}\frac{1}{\rho}\left(\Psi\left(\mathbf{\boldsymbol{\omega}}^{0}\right)-\Psi\left(\boldsymbol{\omega}^{n}\right)\right)<\infty
    \end{equation*}
    \begin{equation*}
        \implies\sum_{k=0}^{\infty}\left\Vert \boldsymbol{\omega}^{k}-\boldsymbol{\omega}^{k+1}\right\Vert _{2}^{2}<\infty.
    \end{equation*}
\end{proof}
\begin{remark}
    The sequence $\left\{ \boldsymbol{\omega}^{k}\right\} _{k\in\mathbb{N}}$ is bounded, since the regularization terms in the objective function $\Psi$ bounds the blocks $\mathbf{a},\mathbf{b}$, and $\mathbf{c}$, and $\Psi\left(\boldsymbol{\omega}^{k}\right)$ is non-increasing. 
\end{remark}

\begin{theorem}\label{thm:re_cond}
Let $\left\{ \boldsymbol{\omega}^{k}\right\} _{k\in\mathbb{N}}$ be the sequence generated by our regularized ALS algorithm, then there exists a constant $\nu>0$ such that for any $k\in\mathbb{N}$, there is a vector $\eta^{k+1}\in\partial\Psi\left(\boldsymbol{\omega}^{k+1}\right)$ such that
\begin{equation*}
    \left\Vert \eta^{k+1}\right\Vert \leq\nu\left\Vert \boldsymbol{\omega}^{k}-\boldsymbol{\omega}^{k+1}\right\Vert 
\end{equation*}
\end{theorem}
\begin{proof}
Let $k$ be a positive integer. By the first order optimality condition of the subproblems, we have
\begin{equation*}
\begin{split}
&\nabla_{b}f\left(\mathbf{a}^{k},\mathbf{b}^{k+1},\mathbf{c}^{k}\right)+\mathbf{L}_{b}^{\top}\mathbf{L}_{b}\mathbf{b}^{k+1}=0,\\  &\nabla_{c}f\left(\mathbf{a}^{k},\mathbf{b}^{k+1},\mathbf{c}^{k+1}\right)+\mathbf{L}_{c}^{\top}\mathbf{L}_{c}\mathbf{c}^{k+1}=0,\\
&\nabla_{a}f\left(\mathbf{a}^{k+1},\mathbf{b}^{k+1},\mathbf{c}^{k+1}\right)+\mathbf{L}_{a}^{\top}\mathbf{L}_{a}\mathbf{a}^{k+1}=0.
\end{split}
\end{equation*}
Let us define $\eta_{2}^{k+1}:=\nabla_{b}f\left(\boldsymbol{\omega}^{k+1}\right)-\nabla_{b}f\left(\mathbf{a}^{k},\mathbf{b}^{k+1},\mathbf{c}^{k}\right)$, then
\begin{equation*}
    \eta_{2}^{k+1}=\nabla_{b}f\left(\boldsymbol{\omega}^{k+1}\right)+\mathbf{L}_{b}^{\top}\mathbf{L}_{b}\mathbf{b}^{k+1}=\nabla_{b}\Psi\left(\boldsymbol{\omega}^{k+1}\right).
\end{equation*}
Similarly, define $\eta_{3}^{k+1}:=\nabla_{c}f\left(\boldsymbol{\omega}^{k+1}\right)-\nabla_{c}f\left(\mathbf{a}^{k},\mathbf{b}^{k+1},\mathbf{c}^{k+1}\right)$ implies $\eta_{1}^{k+1}=\nabla_{c}\Psi\left(\boldsymbol{\omega}^{k+1}\right)$, and define $\eta_{3}^{k+1}:=\nabla_{a}f\left(\boldsymbol{\omega}^{k+1}\right)-\nabla_{a}f\left(\mathbf{a}^{k+1},\mathbf{b}^{k+1},\mathbf{c}^{k+1}\right)$ implies $\eta_{2}^{k+1}=\nabla_{b}\Psi\left(\boldsymbol{\omega}^{k+1}\right)$. Hence the vector $\eta^{k+1}=\left(\eta_{1}^{k+1},\eta_{2}^{k+1},\eta_{3}^{k+1}\right)\in\partial\Psi\left(\boldsymbol{\omega}^{k+1}\right)$. 

Moreover, since the sequence $\left\{ \boldsymbol{\omega}^{k}\right\} _{k\in\mathbb{N}}$ is bounded, and the objective function $f$ is twice continuously differentiable, then by the mean value theorem, $\nabla f$ is Lipschitz continuous. Hence, there exists a constant $C_{2}>0$ such that
\begin{equation*}
    \begin{split}
        \left\Vert \eta_{2}^{k+1}\right\Vert &=\left\Vert \nabla_{b}f\left(\boldsymbol{\omega}^{k+1}\right)-\nabla_{b}f\left(\mathbf{b}^{k},\mathbf{b}^{k+1},\mathbf{c}^{k}\right)\right\Vert \\
        &=\left\Vert \nabla_{b}f\left(\boldsymbol{\omega}^{k+1}\right)-\nabla_{b}f\left(\mathbf{a}^{k},\mathbf{b}^{k+1},\mathbf{c}^{k}\right)+\nabla_{b}f\left(\mathbf{a}^{k},\mathbf{b}^{k+1},\mathbf{c}^{k}\right)-\nabla_{b}f\left(\mathbf{a}^{k},\mathbf{b}^{k},\mathbf{c}^{k}\right)\right\Vert \\
        &\leq\left\Vert \nabla_{b}f\left(\boldsymbol{\omega}^{k+1}\right)-\nabla_{b}f\left(\mathbf{a}^{k},\mathbf{b}^{k+1},\mathbf{c}^{k}\right)\right\Vert +\left\Vert \nabla_{b}f\left(\mathbf{a}^{k},\mathbf{b}^{k+1},\mathbf{c}^{k}\right)-\nabla_{b}f\left(\mathbf{a}^{k},\mathbf{b}^{k},\mathbf{c}^{k}\right)\right\Vert \\
        &\leq C_{1}\left\Vert \boldsymbol{\omega}^{k+1}-(\mathbf{a}^{k},\mathbf{b}^{k+1},\mathbf{c}^{k})\right\Vert +C_{2}\left\Vert \boldsymbol{\omega}^{k}-\boldsymbol{\omega}^{k+1}\right\Vert \\
        &\leq P_{2}\left\Vert \boldsymbol{\omega}^{k}-\boldsymbol{\omega}^{k+1}\right\Vert 
    \end{split}
\end{equation*}
for some constant $P_{2}\geq C_{1}+C_{2}$

Similarly, there exists cosntants $P_{1},P_{3}>0$ such that $\left\Vert \eta_{1}^{k+1}\right\Vert \leq P_{1}\left\Vert \boldsymbol{\omega}^{k}-\boldsymbol{\omega}^{k+1}\right\Vert$  and $\left\Vert \eta_{3}^{k+1}\right\Vert \leq P_{3}\left\Vert \boldsymbol{\omega}^{k}-\boldsymbol{\omega}^{k+1}\right\Vert$. 

Setting $\nu=\max\left\{ P_{1},P_{2},P_{3}\right\}$  gives $\left\Vert \eta^{k+1}\right\Vert \leq\nu\left\Vert \boldsymbol{\omega}^{k}-\boldsymbol{\omega}^{k+1}\right\Vert$.
\end{proof}

\begin{theorem}
Let $\left\{ \boldsymbol{\omega}^{k}\right\} _{k\in\mathbb{N}}$ be the sequence generated by our regularized ALS algorithm, then $\left\{ \boldsymbol{\omega}^{k}\right\} _{k\in\mathbb{N}}$ converges to the critical point of $\Psi$.
\end{theorem}
\begin{proof} 
By Lemma 2.3 in \cite{attouch2013convergence}, the squared distance function $g(x,y)=\Vert x- y \Vert^2$ for $x,y\in \mathbb{R}^{m}$ is semi-algebraic. Moreover, the univariate function $f(x)=|x|$ is semi-algebraic \cite{karim2019tensor}. Since the addition and the composition of semi-algebraic functions are semi-algebraic, the objective function given in Eq.~\ref{eq:tikhonov-obj} is semi-algebraic. Hence the objective function satisfies the Kurdyka-Łojasiewicz property. Then by Theorem.~\ref{thm:re_cond} and \ref{thm:decrease}, results follow from Theorem 2.9 in \cite{attouch2013convergence}.
\end{proof}

%% file: bmp_paper.bbl
\begin{thebibliography}{10}

\bibitem{attouch2013convergence}
Hedy Attouch, J{\'e}r{\^o}me Bolte, and Benar~Fux Svaiter.
\newblock Convergence of descent methods for semi-algebraic and tame problems:
  proximal algorithms, forward--backward splitting, and regularized
  gauss--seidel methods.
\newblock {\em Mathematical Programming}, 137(1):91--129, 2013.

\bibitem{SBIdataset2016}
Thierry Bouwmans, Lucia Maddalena, and Alfredo Petrosino.
\newblock Scene background initialization (sbi) dataset, 2016.

\bibitem{bouwmans2017scene}
Thierry Bouwmans, Lucia Maddalena, and Alfredo Petrosino.
\newblock Scene background initialization: A taxonomy.
\newblock {\em Pattern Recognition Letters}, 96:3--11, 2017.

\bibitem{candes2011robust}
Emmanuel~J Cand{\`e}s, Xiaodong Li, Yi~Ma, and John Wright.
\newblock Robust principal component analysis?
\newblock {\em Journal of the ACM (JACM)}, 58(3):1--37, 2011.

\bibitem{carroll1970analysis}
J~Douglas Carroll and Jih-Jie Chang.
\newblock Analysis of individual differences in multidimensional scaling via an
  n-way generalization of “eckart-young” decomposition.
\newblock {\em Psychometrika}, 35(3):283--319, 1970.

\bibitem{de2000multilinear}
Lieven De~Lathauwer, Bart De~Moor, and Joos Vandewalle.
\newblock A multilinear singular value decomposition.
\newblock {\em SIAM journal on Matrix Analysis and Applications},
  21(4):1253--1278, 2000.

\bibitem{de2008decompositions}
Lieven De~Lathauwer and Dimitri Nion.
\newblock Decompositions of a higher-order tensor in block terms—part iii:
  Alternating least squares algorithms.
\newblock {\em SIAM journal on Matrix Analysis and Applications},
  30(3):1067--1083, 2008.

\bibitem{erichson2019compressed}
N~Benjamin Erichson, Steven~L Brunton, and J~Nathan Kutz.
\newblock Compressed dynamic mode decomposition for background modeling.
\newblock {\em Journal of Real-Time Image Processing}, 16:1479--1492, 2019.

\bibitem{erichson2016randomized}
N~Benjamin Erichson and Carl Donovan.
\newblock Randomized low-rank dynamic mode decomposition for motion detection.
\newblock {\em Computer Vision and Image Understanding}, 146:40--50, 2016.

\bibitem{garcia2020background}
Belmar Garcia-Garcia, Thierry Bouwmans, and Alberto Jorge~Rosales Silva.
\newblock Background subtraction in real applications: Challenges, current
  models and future directions.
\newblock {\em Computer Science Review}, 35:100204, 2020.

\bibitem{gnang2011spectral}
Edinah~K Gnang, Ahmed Elgammal, and Vladimir Retakh.
\newblock A spectral theory for tensors.
\newblock In {\em Annales de la Facult{\'e} des sciences de Toulouse:
  Math{\'e}matiques}, volume~20, pages 801--841, 2011.

\bibitem{gnang2017spectra}
Edinah~K Gnang and Yuval Filmus.
\newblock On the spectra of hypermatrix direct sum and kronecker products
  constructions.
\newblock {\em Linear Algebra and its Applications}, 519:238--277, 2017.

\bibitem{gnang2020bhattacharya}
Edinah~K Gnang and Yuval Filmus.
\newblock On the bhattacharya-mesner rank of third order hypermatrices.
\newblock {\em Linear Algebra and its Applications}, 588:391--418, 2020.

\bibitem{gnang2020symmetrization}
Edinah~K Gnang and Fan Tian.
\newblock A symmetrization approach to hypermatrix svd.
\newblock {\em arXiv preprint arXiv:2004.10368}, 2020.

\bibitem{golub2013matrix}
Gene~H Golub and Charles~F Van~Loan.
\newblock {\em Matrix computations}.
\newblock JHU press, 2013.

\bibitem{grippo2000convergence}
Luigi Grippo and Marco Sciandrone.
\newblock On the convergence of the block nonlinear gauss-seidel method under
  convex constraints.
\newblock {\em Operations research letters}, 26(3):127--136, 2000.

\bibitem{grosek2014dynamic}
Jacob Grosek and J~Nathan Kutz.
\newblock Dynamic mode decomposition for real-time background/foreground
  separation in video.
\newblock {\em arXiv preprint arXiv:1404.7592}, 2014.

\bibitem{haq2020dynamic}
Israr~Ul Haq, Keisuke Fujii, and Yoshinobu Kawahara.
\newblock Dynamic mode decomposition via dictionary learning for foreground
  modeling in videos.
\newblock {\em Computer Vision and Image Understanding}, 199:103022, 2020.

\bibitem{kajo2018svd}
Ibrahim Kajo, Nidal Kamel, Yassine Ruichek, and Aamir~Saeed Malik.
\newblock Svd-based tensor-completion technique for background initialization.
\newblock {\em IEEE Transactions on Image Processing}, 27(6):3114--3126, 2018.

\bibitem{karim2019tensor}
RAMIN~GOUDARZI KARIM.
\newblock {\em TENSOR DECOMPOSITIONS AND RANK APPROXIMATION OF TENSORS WITH
  APPLICATIONS}.
\newblock PhD thesis, The University of Alabama in Huntsville, 2019.

\bibitem{karim2020accurate}
Ramin~Goudarzi Karim, Guimu Guo, Da~Yan, and Carmeliza Navasca.
\newblock Accurate tensor decomposition with simultaneous rank approximation
  for surveillance videos.
\newblock In {\em 2020 54th Asilomar Conference on Signals, Systems, and
  Computers}, pages 842--846. IEEE, 2020.

\bibitem{kiers2000towards}
Henk~AL Kiers.
\newblock Towards a standardized notation and terminology in multiway analysis.
\newblock {\em Journal of Chemometrics: A Journal of the Chemometrics Society},
  14(3):105--122, 2000.

\bibitem{kilmer2021tensor}
Misha~E Kilmer, Lior Horesh, Haim Avron, and Elizabeth Newman.
\newblock Tensor-tensor algebra for optimal representation and compression of
  multiway data.
\newblock {\em Proceedings of the National Academy of Sciences},
  118(28):e2015851118, 2021.

\bibitem{kilmer2011factorization}
Misha~E Kilmer and Carla~D Martin.
\newblock Factorization strategies for third-order tensors.
\newblock {\em Linear Algebra and its Applications}, 435(3):641--658, 2011.

\bibitem{knuth1997art}
Donald~Ervin Knuth.
\newblock {\em The Art of Computer Programming: Fundamental Algorithms.},
  volume~1.
\newblock Addison-Wesley, 1997.

\bibitem{kolda2009tensor}
Tamara~G Kolda and Brett~W Bader.
\newblock Tensor decompositions and applications.
\newblock {\em SIAM review}, 51(3):455--500, 2009.

\bibitem{kutz2016dynamic}
J~Nathan Kutz, Steven~L Brunton, Bingni~W Brunton, and Joshua~L Proctor.
\newblock {\em Dynamic mode decomposition: data-driven modeling of complex
  systems}.
\newblock SIAM, 2016.

\bibitem{kutz2015multi}
J~Nathan Kutz, Xing Fu, Steve~L Brunton, and N~Benjamin Erichson.
\newblock Multi-resolution dynamic mode decomposition for foreground/background
  separation and object tracking.
\newblock In {\em 2015 IEEE International Conference on Computer Vision
  Workshop (ICCVW)}, pages 921--929. IEEE, 2015.

\bibitem{li2004statistical}
Liyuan Li, Weimin Huang, Irene Yu-Hua Gu, and Qi~Tian.
\newblock Statistical modeling of complex backgrounds for foreground object
  detection.
\newblock {\em IEEE Transactions on image processing}, 13(11):1459--1472, 2004.

\bibitem{li2013some}
Na~Li, Stefan Kindermann, and Carmeliza Navasca.
\newblock Some convergence results on the regularized alternating least-squares
  method for tensor decomposition.
\newblock {\em Linear Algebra and its Applications}, 438(2):796--812, 2013.

\bibitem{li2022tensor}
Zina Li, Yao Wang, Qian Zhao, Shijun Zhang, and Deyu Meng.
\newblock A tensor-based online rpca model for compressive background
  subtraction.
\newblock {\em IEEE Transactions on Neural Networks and Learning Systems},
  2022.

\bibitem{luo2023multidimensional}
Qilun Luo, Ming Yang, Wen Li, and Mingqing Xiao.
\newblock Multidimensional data processing with bayesian inference via
  structural block decomposition.
\newblock {\em IEEE Transactions on Cybernetics}, 2023.

\bibitem{maddalena2015towards}
Lucia Maddalena and Alfredo Petrosino.
\newblock Towards benchmarking scene background initialization.
\newblock In {\em New Trends in Image Analysis and Processing--ICIAP 2015
  Workshops: ICIAP 2015 International Workshops, BioFor, CTMR, RHEUMA, ISCA,
  MADiMa, SBMI, and QoEM, Genoa, Italy, September 7-8, 2015, Proceedings 18},
  pages 469--476. Springer, 2015.

\bibitem{mesner1990association}
Dale~M Mesner and Prabir Bhattacharya.
\newblock Association schemes on triples and a ternary algebra.
\newblock {\em Journal of Combinatorial Theory, Series A}, 55(2):204--234,
  1990.

\bibitem{mesner1994ternary}
Dale~M Mesner and Prabir Bhattacharya.
\newblock A ternary algebra arising from association schemes on triples.
\newblock {\em Journal of algebra}, 164(3):595--613, 1994.

\bibitem{mocks1988topographic}
J~Mocks.
\newblock Topographic components model for event-related potentials and some
  biophysical considerations.
\newblock {\em IEEE transactions on biomedical engineering}, 35(6):482--484,
  1988.

\bibitem{navasca2008swamp}
Carmeliza Navasca, Lieven De~Lathauwer, and Stefan Kindermann.
\newblock Swamp reducing technique for tensor decomposition.
\newblock In {\em 2008 16th European Signal Processing Conference}, pages 1--5.
  IEEE, 2008.

\bibitem{oseledets2011tensor}
Ivan~V Oseledets.
\newblock Tensor-train decomposition.
\newblock {\em SIAM Journal on Scientific Computing}, 33(5):2295--2317, 2011.

\bibitem{redman2021koopman}
William~T Redman.
\newblock On koopman mode decomposition and tensor component analysis.
\newblock {\em Chaos: An Interdisciplinary Journal of Nonlinear Science},
  31(5):051101, 2021.

\bibitem{sobral2015online}
Andrews Sobral, Sajid Javed, Soon Ki~Jung, Thierry Bouwmans, and El-hadi
  Zahzah.
\newblock Online stochastic tensor decomposition for background subtraction in
  multispectral video sequences.
\newblock In {\em Proceedings of the IEEE International Conference on Computer
  Vision Workshops}, pages 106--113, 2015.

\bibitem{miguel2021cloud}
Miguel~\'{A}. $\text{Padri\~{n}\'{a}n}$.
\newblock Cloud video.
\newblock
  \url{https://www.pexels.com/video/white-clouds-on-the-blue-sky-6772574/},
  2021.

\bibitem{siam2022talk}
Fan Tian, Misha~E. Kilmer, Eric~L. Miller, Abani Patra, and Anna Konstorum.
\newblock Approximate tensor bm product decomposition for temporal analysis of
  third-order data.
\newblock \url{https://meetings.siam.org/sess/dsp_talk.cfm?p=122285}, 9 2022.

\bibitem{tian2005robust}
Ying-Li Tian, Max Lu, and Arun Hampapur.
\newblock Robust and efficient foreground analysis for real-time video
  surveillance.
\newblock In {\em 2005 IEEE Computer Society Conference on Computer Vision and
  Pattern Recognition (CVPR'05)}, volume~1, pages 1182--1187. IEEE, 2005.

\bibitem{tucker1966some}
Ledyard~R Tucker.
\newblock Some mathematical notes on three-mode factor analysis.
\newblock {\em Psychometrika}, 31(3):279--311, 1966.

\bibitem{uschmajew2012local}
Andr{\'e} Uschmajew.
\newblock Local convergence of the alternating least squares algorithm for
  canonical tensor approximation.
\newblock {\em SIAM Journal on Matrix Analysis and Applications},
  33(2):639--652, 2012.

\bibitem{yang2022global}
Yuning Yang.
\newblock On global convergence of alternating least squares for tensor
  approximation.
\newblock {\em Computational Optimization and Applications}, pages 1--21, 2022.

\bibitem{ZhangGit2014}
Zemin Zhang.
\newblock Tensor completion and tensor rpca.
\newblock
  \url{https://github.com/jamiezeminzhang/Tensor_Completion_and_Tensor_RPCA/tree/master?tab=readme-ov-file},
  2014.

\bibitem{zhang2014novel}
Zemin Zhang, Gregory Ely, Shuchin Aeron, Ning Hao, and Misha Kilmer.
\newblock Novel methods for multilinear data completion and de-noising based on
  tensor-svd.
\newblock In {\em Proceedings of the IEEE conference on computer vision and
  pattern recognition}, pages 3842--3849, 2014.

\end{thebibliography}
